\documentclass[notitlepage,leqno,10pt]{amsart}

\usepackage{xcolor}

\usepackage{latexsym}
\usepackage{amsmath}
\usepackage{amsfonts}
\usepackage{amssymb}
\usepackage{graphicx}
\usepackage{esint}
\usepackage{slashed}

\renewcommand{\a}{\alpha}
\renewcommand{\b }{\beta}
\renewcommand{\d}{\delta}
\newcommand{\D }{\Delta}

\newcommand{\e}{\varepsilon}
\newcommand{\g}{\gamma}

\renewcommand{\l}{\lambda}
\renewcommand{\L}{\Lambda}

\newcommand{\n}{\nabla}
\newcommand{\var}{\varphi}

\newcommand{\s}{\sigma}

\newcommand{\ov}{\overline}
\newcommand{\intbar}{\mathop{\int\makebox(-13.5,0){\rule[4pt]{.7em}{0.3pt}}%
\kern-6pt}\nolimits}

\newcommand{\be}{\begin{equation}}
\newcommand{\ee}{\end{equation}}
\newenvironment{pf}{\noindent{\sc Proof}.\enspace}{\rule{2mm}{2mm}\medskip}
\newenvironment{pfn}{\noindent{\sc Proof}}{\rule{2mm}{2mm}\medskip}

\newcommand{\R}{\mathbb{R}}

\newcommand{\N}{\mathbb{N}}
\newcommand{\pa}{\partial}

\author{Pierpaolo Esposito}
\address{Pierpaolo Esposito, Dipartimento di Matematica e Fisica, Universit\`a degli Studi Roma Tre,\,Largo S. Leonardo Murialdo 1, 00146 Roma, Italy}
\email{esposito@mat.uniroma3.it}

\author{Andrea Malchiodi}
\address{Andrea Malchiodi, Scuola Normale Superiore, \,Piazza dei Cavalieri , 56126 Pisa, Italy}
\email{andrea.malchiodi@sns.it}

\date{}

\title{Critical metrics for log-determinant functionals  in conformal geometry}

\begin{document}

\newtheorem{lem}{Lemma}[section]
\newtheorem{pro}[lem]{Proposition}
\newtheorem{thm}[lem]{Theorem}
\newtheorem{rem}[lem]{Remark}
\newtheorem{cor}[lem]{Corollary}
\newtheorem{df}[lem]{Definition}

\def\blue{\color{blue}}
\def\red{\color{orange}}

\renewcommand{\theequation}{\arabic{section}.\arabic{equation}}

\maketitle

\begin{center}

\bigskip


\bigskip

\centerline{\bf AMS subject classification: 35G20, 35B44, 35J35, 58J50}

\end{center}

\begin{abstract} We consider critical points of a class of functionals 
on compact four-dimensional manifolds arising from 
 {\em Regularized Determinants} for conformally covariant operators,  whose explicit form was derived in \cite{BO}, extending Polyakov's formula. These correspond 
to solutions of elliptic equations of Liouville type that are  quasilinear, of mixed orders and of critical type. After 
studying existence, asymptotic behaviour and uniqueness of {\em fundamental solutions}, we prove a quantization property 
under blow-up, and then derive existence results via critical point theory. 
\end{abstract}

\section{Introduction}\label{s:in}

Consider a compact Riemannian manifold $(M,g)$  without boundary of dimension $n$, 
with Laplace-Beltrami operator $\Delta_g$.  
By {\em Weyl's asymptotic formula} it is known that the
eigenvalues $\l_j$ of $-\Delta_g$ obey the limiting law 
$\lambda_j \sim j^{2/n}$ as $j \rightarrow \infty$.  
The {\em determinant} of  $-\Delta_g$ is formally  
 the product of all its eigenvalues, with a rigorous 
definition that can be obtained via holomorphic extension 
of the {\em zeta function} 
\begin{align*} 
\zeta(s) = \sum_{j=1}^{\infty} \lambda_j^{-s}.
\end{align*} 
The behaviour of the $\l_j$'s implies that $\zeta(s)$ is analytic for $\mbox{Re}(s) > n/2$: 
it is possible anyway to meromorphically extend $\zeta$ 
so that  it becomes regular near $s = 0$ (see
\cite{RS}).  From the formal calculation 
$\zeta^{\prime}(0) = - \displaystyle \sum_{j = 1}^{\infty} \log \lambda_j = - \log \det (-\Delta_g)$ 
one then defines  
\begin{align*} 
\det(-\Delta_g) = e^{- \zeta^{\prime}(0)}.
\end{align*} 

Recall that in two dimensions the Laplace-Beltrami operator is conformally covariant in the 
sense that 
\begin{equation}\label{eq:cov-lapl}
\Delta_{\tilde{g}} = e^{-2w}\Delta_g, \quad \qquad 
\tilde{g} = e^{2w}g. 
\end{equation}
This property, as well as the transformation law for the Gaussian curvature 
\begin{align} \label{GCE}
- \Delta_g w + K_g = K_{\tilde{g}} e^{2w},
\end{align}
allowed Polyakov in \cite{Polyakov} to determine the logarithm of the ratio of 
 regularized determinants of two conformally-equivalent metrics with the 
 same area on a compact surface: 
\begin{align} \label{Polyform}
\log \frac{\det (-\Delta_{\tilde{g}})}{\det (-\Delta_g)} =
-\frac{1}{12\pi} \int_{\Sigma} (|\nabla w|_g^2 + 2K_g w)\ dv_g. 
\end{align}
The  Gaussian curvature $K_g$ appears in the above formula since it is 
possible to rewrite the zeta function as an integral of a trace 
$$
  \zeta(s) = \frac{1}{\Gamma(s)} \int_0^\infty Tr \left( e^{\Delta_g \, t} - \frac{1}{{ \hbox{Area}_g(\Sigma)}} \right) dt, 
$$
where $\Gamma(s)$ is Euler's Gamma function and $e^{\Delta_g \, t}$ is the heat kernel 
on $(\Sigma, g)$. The latter kernel, for $t$ small, has the asymptotic profile of the  
Euclidean one, with next-order corrections involving the Gaussian curvature and its 
covariant derivatives, as shown in \cite{MP}. 

Using \eqref{GCE} and Polyakov's formula it is easy to show that critical points of the 
regularized determinant in a given conformal class give rise to 
constant Gaussian curvature metrics. In \cite{OPS2,OPS1} Osgood, Phillips and Sarnak 
proved existence of extremals for all given topologies: uniqueness holds for non-positive 
Euler characteristic, while in the positive case there are as many solutions as M\"obius maps. 
The M\"obius action is indeed employed to  fix a {\em center of mass} gauge, in the 
spirit of \cite{Aubin}, to  exploit an {\em improved} Moser-Trudinger 
type inequality. Still in \cite{OPS2,OPS1} the authors used 
formula \eqref{Polyform} in order to derive compactness of isospectral metrics 
on closed surfaces with a given topology. This result was then extended to the 
three-dimensional case in \cite{CY90}, for metrics within a fixed conformal class.

\

In four dimension formulas similar to \eqref{Polyform} were obtained for regularized 
determinants of operators enjoying covariance properties analogous to \eqref{eq:cov-lapl}. 
More precisely, a differential operator $A_g$ (depending on the metric) 
 is said to be {\em conformally covariant of
bi-degree $(a,b)$} if
\begin{align} \label{Ccovar}
A_{\tilde{g}}\psi = e^{-bw} A_g (e^{aw}\psi), \quad \tilde{g} = e^{2w} g,
\end{align}
for each smooth function $\psi$ (or even for a smooth section of a vector bundle).  
One such example is the {\em conformal Laplacian} in dimension $n \geq 3$ 
\begin{align*} 
L_g = -\Delta_g + \frac{(n-2)}{4(n-1)}R_g,
\end{align*}
where $R_g$ is the scalar curvature: this operator satisfies 
\eqref{Ccovar} with $a = \frac{n-2}{2}$ and $b = \frac{n+2}{2}$. 
Other examples include the {\em Dirac operator} 
$\slash \!\!\!\! D_g$, which satisfies \eqref{Ccovar} with $a = \frac{n-1}{2}$, $b = \frac{n+1}{2}$, 
and  the  {\em Paneitz operator} in four dimensions 
\begin{equation} \label{eq:Paneitz}
P_g \psi = \Delta_g^2 \psi - {\rm div} \left( \frac{2}{3}R_g \n \psi - 2 Ric_g(\cdot, \n \psi) \right), 
\end{equation}
that satisfies \eqref{Ccovar} with $a = 0$ and $b = 4$.

Branson and {\O}rsted  generalized in \cite{BO} Polyakov's formula to four-dimensional 
manifolds $(M,g)$, proving the following result: the logarithmic ratio of 
two regularized determinants is the linear  combination of three universal functionals, with 
coefficients depending on the specific operator. More precisely, if $A = A_g$ 
is conformally covariant and has no kernel (otherwise, see Remark \ref{r:ker}), then one has  
\begin{align} \label{FAdef}
F_{A}[w] = \log \frac{\det A_{\tilde{g}}}{\det A_g} = \gamma_1(A) I[w]
+ \gamma_2(A) II[w] + \gamma_3(A) III[w], \quad \tilde{g}=e^{2w}g,
\end{align}
where $(\gamma_1,\gamma_2,\gamma_3) \in \R^3$ 
and $I,II,III$ are defined as 
\begin{eqnarray*} 
&& I[w] = 4 \int_M w| W_g|_g^2\ dv_g - \big( \int_M |W_g|_g^2\ dv_g \big) \log \fint_M
e^{4w}\ dv_g\\
&& II[w] = \int_M w P_g w\ dv_g +4 \int_M Q_g w \ dv_g- \big( \int_M Q_g \ dv_g \big) \log
\fint_M e^{4w}\ dv_g \\
&& III[w] = 12 \int_M (\Delta_g w + |\nabla w|_g^2 )^2\ dv_g - 4 \int_M ( w
\Delta_g R_g + R_g |\nabla w|_g^2 )\ dv_g.
\end{eqnarray*}
Here  $W_g$ is the Weyl curvature tensor, and $Q_g$ the {\em Q-curvature} of $(M,g)$
\begin{align*}
Q_g = \frac{1}{12}( -\Delta_g R_g + R_g^2 - 3|Ric_g|_g^2).
\end{align*}
The latter quantity is a natural higher-order counterpart of the Gaussian 
curvature, and transforms conformally via the Paneitz operator by the law 
\begin{align*}
P_g w + 2 Q_g = 2Q_{\tilde{g} }e^{4w}, \quad \tilde{g}=e^{2w}g,
\end{align*}
totally analogous to \eqref{GCE}.
The above three functionals are geometrically natural as their critical 
points can be characterized by the conditions 
\begin{eqnarray*} 
\tilde{g} = e^{2w}g\ \mbox{is a critical point of }I\
&\Longleftrightarrow& |W_{\tilde{g}}|_{\tilde{g}}^2 = const.\\
\tilde{g} = e^{2w}g\ \mbox{is a critical point of }II\
& \Longleftrightarrow&  Q_{\tilde{g}} = const.\\
\tilde{g} = e^{2w}g\ \mbox{is a critical point of }III\
&\Longleftrightarrow& \Delta_{\tilde{g}} R_{\tilde{g}} = 0.
\end{eqnarray*}
Notice that, since $M$ is compact, the last condition yields a 
Yamabe metric,  with constant scalar curvature.

The Euler-Lagrange equation for $F_A$ implies constancy of a scalar quantity $U_g$, which we 
call {\em $U$-curvature}, defined as 
\begin{align} \label{Udef}
U_g= \gamma_1 |W_g|_g^2 +\gamma_2 Q_g - \gamma_3 \Delta_g R_g.
\end{align}
In terms of the conformal factor the stationarity equation is 
\begin{eqnarray} \label{EL1} 
&&\mathcal{N}_g(w) +U_g = \mu e^{4w}; \\
&&\mathcal{N}(w)=\frac{\gamma_2}{2} P_g w+6\gamma_3\Delta_g (\Delta_g w +
|\nabla w|_g^2) - 12\gamma_3 {\rm div} \big[ (\Delta_g w
+ |\nabla w|_g^2) \nabla w \big] + 2 \gamma_3 {\rm div} (R_g \nabla w), \label{linearL}
\end{eqnarray}
where
\begin{align*} 
\mu = -\frac{\kappa_A}{\int_M e^{4w} dv_g}; \qquad \qquad 
\kappa_A = -\gamma_1 \int_M |W_g|_g^2\ dv_g - \gamma_2 \int_M Q_g \ dv_g.
\end{align*}
We note that $k_A$ is a conformal invariant, since $\int_M Q_g \ dv_g$ is, and that the above equation \eqref{EL1}
corresponds to solving $U_{\tilde g} \equiv \mu$. 

 For example, one has
\begin{align*} 
\gamma_1(L_g) = 1, \qquad \gamma_2(L_g)= -4, \qquad \gamma_3(L_g) = -2/3
\end{align*}
for the conformal Laplacian and
\begin{align*} 
\gamma_1(\slash \!\!\!\! D_g^2) = -7, \qquad \gamma_2(\slash \!\!\!\! D_g^2)= -88, \qquad \gamma_3(\slash \!\!\!\! D_g^2) = -\frac{14}{3}
\end{align*}
for the square of the Dirac operator $\slash \!\!\!\! D_g$. For the Paneitz operator, instead, one has 
\begin{align*}
\gamma_1(P_g) = -\frac{1}{4}, \qquad \gamma_2(P_g)= -14, \qquad \gamma_3(P_g) = 8/3.
\end{align*}

\

Concerning extremality of functionals that are linear combinations of $I, II$ and 
$III$, as in \eqref{FAdef}, Chang and Yang \cite{ChangYangAnnals} proved an existence result (with a sign-reverse notation)
under the conditions  $\gamma_2, \gamma_3 > 0$ and  $\kappa_A < 8 \pi^2 \gamma_2$.

The latter inequality (showed in \cite{Gursky1} to hold in positive Yamabe class, 
except for manifolds conformal to the round sphere) was used with a geometric version of a Moser-Trudinger type inequality: in \cite{Adams} an estimate on the (logarithmic) integral of the exponential 
of the conformal factor was derived in terms of the squared norm of the Laplacian, while in \cite{ChangYangAnnals} 
in terms of the quadratic form induced by the Paneitz operator, which 
is conformally covariant. Uniqueness was also proved for the case $k_A < 0$, using the 
 convexity of the functional $F_A$; see also \cite{BCY} for the case of the round sphere, 
where extremals were classified as M\"obius maps (and as unique critical points in 
\cite{GurClass}).  Extremal properties of the round metric on $S^n$ in general 
even dimension were studied in \cite{Mol12}.   Regularity of arbitrary extremals was proved in 
\cite{CGYAJM}, and extended in \cite{UVMRL} to other critical points. The existence result in 
\cite{ChangYangAnnals} was used in \cite{Gur98} to derive optimal bounds on the Weyl 
functional and to prove some rigidity results for K\"ahler-Einstein metrics.

\

Due to the above results, one has a satisfactory existence theory on manifolds of positive Yamabe 
class. It is the aim of this paper to derive it also for manifolds of more general type. One fact that distinguishes two and four dimensions from the conformal point of view is that in the latter case Gauss-Bonnet 
integrals can be larger than those on the round sphere of equal dimension. For example, the total integral of $Q$-curvature 
on four-manifolds of negative Yamabe class can be arbitrarily large. This fact causes the lack of 
one-side control on the functional $II$ in terms of the Moser-Trudinger inequality, which was available in 
\cite{ChangYangAnnals}. Nevertheless, in \cite{DM08} conformal metrics with constant $Q$-curvature 
were found as saddle-type critical points of $II$. The main tool to produce these was a 
variational min-max scheme that used suitable improvements of the Moser-Trudinger inequality 
for conformal factors whose volume is macroscopically {\em spread} over the underlying manifold $M$. 
Such kind of improvement  was derived in two dimensions in \cite{Aubin} for the case of the 
round sphere (see also \cite{Moser}) and in \cite{CL91} for general surfaces. With improved inequalities at hand, 
it was then possible in \cite{DM08} to characterize low-sublevels of the functional $II$, showing that if 
$\int_M Q_g dv_g < 8 (k+1) \pi^2$ for some $k \in \N$, and if $II(w)$ is sufficiently low, then 
the {\em conformal volume} $e^{4w}$ approaches distributionally a measure supported on 
at most $k$ points of $M$. This geometric characterization of the Euler-Lagrange functional $II$ 
allowed to produce {\em Palais-Smale sequences}, namely approximate solutions to the prescribed 
$Q$-curvature equation. Using also a {\em monotonicity argument} from \cite{str} one can 
replace Palais-Smale sequences by sequences of solutions to approximate equations, which might carry more 
information than general Palais-Smale sequences. 

Here comes the other main aspect of the prescribed $Q$-curvature equation: compactness. 
One would like to show that the latter solutions converge to a solution of the original 
problem. This is actually the result of the two independent papers \cite{DR06}
and \cite{Mal06}:  there it is proved that non-compact sequences of solutions 
develop after rescaling a finite number of {\em bubbles}, the conformal factors 
of the stereographic projection from $\mathbb{S}^4$ to $\R^4$. Each of them carries $8 \pi^2$ 
in $Q$-curvature, and in the latter  work it is shown that no other residual volume can 
occur. A contradiction to loss of compactness is then reached assuming that the 
initial total $Q$-curvature $\int_M Q_g dv_g$ is not a integer multiple of $8 \pi^2$.

\

The first among our results is an analogous compactness property for log-determinant functionals.

\begin{thm}\label{p:minimal-blow-up}
Suppose $M$ is a compact  four-manifold and that $\gamma_2, \gamma_3 \neq 0$, with $\frac{\gamma_2}{\gamma_3} \geq 6$. 
Suppose also that $(w_n)_n$ is a sequence of smooth solutions of  
$$ \mathcal{N}_g (w_n) + \tilde  U_n = \mu_n e^{4 w_n}\qquad \hbox{in }M,$$ 
where $\mathcal{N}_g$ is given by \eqref{linearL}. Assume that $\int_M e^{4 w_n} dv_g =1$, $\mu_n=\int_M \tilde U_n dv_g$ and $\tilde U_n \to U_g$ $C^1-$uniformly in $M$ as $n \to +\infty$. Up to a subsequence, we have one of the following two alternatives:

\begin{itemize}
\item[i)] $(w_n-\fint_M w_n \ dv_g)_n$ is uniformly bounded in $C^{4,\a}(M)$-norm; 

\item[ii)] $(w_n)_n$ blows up, i.e. $\max_M w_n \to +\infty$, and one has that $\fint_M w_n \ dv_g \to - \infty$ and 
$$ \mu_n e^{4 w_n} \rightharpoonup  \sum_{i=1}^l 8 \pi^2 \gamma_2 \d_{p_i}
$$
in the weak sense of distributions for distinct points $p_1,\dots,p_l \in M$.
\end{itemize}
As a consequence, solutions stay compact if $\int_M U_g dv_g \notin 8 \pi^2 \gamma_2  \mathbb{N}$. 
\end{thm}

 \begin{rem}\label{r:Ug}
 In Theorem \ref{p:minimal-blow-up}, it is possible to replace the limit of $\tilde U_n$ by any smooth function $\tilde{U}$.
\end{rem} 

\medskip Well-known results of the above type were proved for second-order Liouville equations in 
\cite{BrMe,ChLi,LiSha}, in presence of singular sources  in \cite{BaTa} and in the fourth-order case \cite{ARS,LiWe,LWW,Rob,RoWe,Wei}. 
The counterpart of Theorem \ref{p:minimal-blow-up} for $Q$-curvature in \cite{DR06,Mal06} 
relied extensively on the Green's representation formula for the Paneitz operator, which is linear. A related quantization 
result was proved in \cite{EsMo} for a Liouville-type $n$-Laplace equation in $n-$dimensional euclidean domains, the equation there of second order allowing truncation techniques towards a-priori estimates (see also \cite{Esp2018} for a classification result of entire solutions). Here, being our operator quasi-linear and of mixed type, none of these 
arguments can be applied and we need to devise new arguments.

 In Section \ref{s:gen-est} we  derive some uniform control of subcritical type on blowing-up 
 solutions, followed by a Caccioppoli-type inequality and a uniform BMO estimate, which  is a 
 natural one since blow-up  is expected to occur with a logarithmic profile. 
 In Section \ref{linear} we develop a general {\em linear} theory for the operator $\mathcal{N}$ 
  in \eqref{linearL}, solving for arbitrary measures in the R.H.S.. Solutions will be found by 
  a limiting procedure with smooth approximations (SOLA: see the terminology there), and  the solvability theory will exploit in a crucial way a nonlinear Hodge decomposition technique. For a R.H.S. given as a linear combination of Dirac masses, a corresponding SOLA is referred to as a {\em fundamental solution} and uniqueness in general fails unless $\gamma_2=6\gamma_3$.
 
In Section \ref{s:fund-sol} we show however that any {\em fundamental solution} satisfies weighted $W^{2,2}-$estimates, allowing via techniques developed in 
 \cite{UVMRL} to prove its 
 logarithmic behaviour near the singularities. 

   There is a vast literature concerning existence and uniqueness issues 
 for problems involving the $p-$Laplace operator, let us  just quote \cite{BBGGPV,BoccGa,DHM2,GIS} and references therein. 
While for the latter both maximum principles and truncation arguments are available, it is not the case for our 
problem, and we had therefore to rely on different arguments.

 With  the asymptotics  of fundamental solutions at hand, we can finally pass to the blow-up analysis of \eqref{EL1}. First, 
 via a Pohozaev type identity, scaling arguments and an epsilon-regularity result we prove a 
 quantization for the volume accumulation at blow-up points. After this, we can then determine that 
 there is no absolutely continuous part in the limit volume measure, after blow-up, leading to Theorem \ref{p:minimal-blow-up}. 
 We collect in an appendix some useful auxiliary results.

\medskip

\noindent As an application of Theorem \ref{p:minimal-blow-up} we have the following existence theorem.

\begin{thm}\label{t:ex}
Assume $\gamma_2, \gamma_3 \neq 0$ and $\frac{\gamma_2}{\gamma_3} \geq 6$.  Suppose $M$ is a compact  four-manifold such that $\int_M U_g dv_g \notin   8 \pi^2 \gamma_2 \mathbb{N}$. Then there exists a conformal metric $\tilde{g}$ with constant $U$-curvature. 
\end{thm}

 Examples to which the latter theorem applies include 
(suitable) products of  negatively-curved surfaces, hyperbolic manifolds 
or their perturbations.

\begin{rem}\label{r:ker}
In case of trivial kernel, both log-determinants of  $L_g$ and $\slash \!\!\!\! D_g^2$ fit in the assumptions of Theorems \ref{p:minimal-blow-up} and \ref{t:ex}.

In general, if a conformally-covariant operator $A$ has a non-trivial kernel,  
some additional quantities appear in \eqref{FAdef}, see Remark 2.2 in \cite{BO}. 
If $A$ has order $2 \ell$, on the R.H.S. of \eqref{FAdef} 
one should add the term 
\begin{equation}\label{eq:add}
2 \ell \int_M\left( w \int_0^1 \Phi_t^2 e^{4 t w} dt \right)dv_g - \frac{1}{2} \ell \, q[A] \log \frac {\int_M 
  e^{4 w} dv_g}{Vol_g(M)}.  
\end{equation}
Here $q[A]$ stands for the dimension of the kernel of $A$, while $\Phi_t^2(x) = \displaystyle \sum_{j=1}^{q[A]} 
\varphi_{j,t}^2(x)$, with $(\varphi_{j,t})_j$ an orthonormal  basis of elements of the kernel 
with respect to the metric  $e^{2 t w} g$. 

For example if $A = L$, the conformal Laplacian, and if the kernel is one-dimensional, denote 
by $\var_1$ an element of the kernel normalized in $L^2$ with respect to $dv_g$. Then,
 recalling that \eqref{Ccovar} holds with $a = 1$, we find that 
 $$
   \Phi_t^2(x)  = \frac{e^{-2tw(x)} \var_1^2(x)}{\int_M e^{2 t w(y)} \var_1^2(y) dv_g(y)}.
 $$ 
Therefore, the extra-term  in \eqref{eq:add} becomes 
$$
  2  \int_M\left(  \int_0^1 \left( \frac{e^{2t w(x)} \var_1^2(x) w(x)}{\int_M e^{2 t w(y)} \var_1^2(y) dv_g(y)} \right)  dt \right)dv_g(x) - \frac{1}{2}  \log \frac {\int_M 
    e^{4 w} dv_g}{Vol_g(M)}.  
$$
Noticing that 
$$
2 \int_M  \left( \frac{e^{2t w(x)} \var_1^2(x) w(x)}{\int_M e^{2 t w(y)} \var_1^2(y) dv_g(y)} \right)  dv_g(x) = 
\frac{d}{dt} \log \int_M e^{2 t w(x)}  \var_1^2(x) dv_g(x),  
$$
the expression in \eqref{eq:add} finally becomes 
$$
 \log \int_M e^{2 w(x)}  \var_1^2(x) dv_g(x) - \frac{1}{2}  \log \frac {\int_M 
     e^{4 w} dv_g}{Vol_g(M)}.  
$$
We will not analyze this term in the present paper. 
\end{rem}

\medskip

The proof of Theorem \ref{t:ex}, given in Section \ref{s:mt-ex}  is variational and mainly inspired from \cite{ChangYangAnnals,DM08}, 
where the $Q$-curvature problem was treated. First, using the results in Section \ref{s:gen-est}, one can obtain a 
sharp Moser-Trudinger inequality involving combinations of the functionals $I$,  $II$ and $III$. The latter  is then improved under suitable conditions on the 
distribution of conformal volume. This allows to apply a general min-max scheme, relying also 
on the construction of test functions with low energy and a prescribed (multiple) concentration 
behaviour of the conformal volume.

\ 

It would be interesting to consider on general manifolds cases with $\gamma$'s of opposite signs, like for the 
determinant of the Paneitz operator (see \cite{Connes94}, IV.4.$\gamma$). This issue is quite hard, as the 
two main terms in the nonlinear operator have competing effects. It is indeed 
studied so far only in particular cases with ODE techniques, see for example \cite{GuMa12}.

\

  \noindent {\bf Notation.} We will work on a compact four-dimensional Riemannian manifold 
  $M$  without boundary endowed with a background metric $g$. When considering this metric, the index $g$ relative to 
  it will be omitted in symbols like $\Delta_g, P_g, dv_g$, etc. Spaces of $L^p$ functions with 
  respect to $dv_g$ will be simply denoted by $L^p$, $p \geq 1$, with norm 
  $\| \cdot \|_p$, and similarly for Sobolev spaces. 
  When the domain of integration is omitted, we mean that it coincides with the whole $M$. The injectivity radius of $(M,g)$ will be denoted by $i_0$ and $B_r$ will denote a generic geodesic ball in $M$. The symbols $\ov{w}$, $\ov{w}^A$ and $\ov{w}^r$ will stand for $\fint_M w \ dv_g$, $\fint_A w \ dv_g$ and $\fint_{B_r} w\  dv_g$, respectively.

 \
  
  \noindent {\bf Acknowledgments.}
  A.M. has been supported by the project {\em Geometric Variational Problems} and {\em Finanziamento a supporto della ricerca di base} from Scuola Normale Superiore and by MIUR Bando PRIN 2015 2015KB9WPT$_{001}$.  P.E. has been supported by MIUR Bando PRIN 2015 2015KB9WPT$_{008}$. As members, they are both partially supported by GNAMPA as part of INdAM.

\section{Some basic estimates} \label{s:gen-est}
\setcounter{equation}{0}  

\noindent In this section we will derive some uniform estimates for smooth solutions of \eqref{EL1} with 
a general R.H.S. by just assuming $\frac{\gamma_2}{\gamma_3}>\frac{3}{2}$. To this aim, recall the 
definition of the   quasilinear differential operator $\mathcal{N}$ in \eqref{linearL}. 
Integrating by parts, notice that the main order term in $\langle \mathcal{N}(w),w\rangle$ has the form
$$(\frac{\gamma_2}{2}+6\gamma_3) \int (\Delta w)^2 dv+18\gamma_3 \int \Delta w |\nabla w|^2  dv
+12 \gamma_3 \int |\nabla w|^4 dv,$$
which can be easily seen to have a sign by a squares completion provided $\frac{\gamma_2}{\gamma_3}>\frac{3}{2}$. In the next section, we will further strengthen the a-priori estimates when $\frac{\gamma_2}{\gamma_3} \geq 6$ and deduce  uniqueness properties when $\frac{\gamma_2}{\gamma_3}=6$.

\medskip \noindent In order to include also local estimates, test \eqref{linearL} against $\varphi=\chi^4 \psi (w-c)$, 
where $c \in \mathbb{R}$, $\psi \in C^2(\mathbb{R})$ (bounded, and with bounded first- and second-order derivatives) and $\chi \in C^\infty(M)$, to get
\begin{eqnarray}
&& \langle \mathcal{N} (w),\varphi \rangle= (\frac{\gamma_2}{2}+6\gamma_3) \int \chi^4 \psi' (\Delta w)^2 dv+
 \int \chi^4 [18 \gamma_3 \psi'+(\frac{\gamma_2}{2}+6\gamma_3) \psi''] \Delta w |\nabla w|^2 dv \label{00828} \\
&&+6 \gamma_3 \int \chi^4 (2\psi'+\psi'') |\nabla w|^4  dv+ \int \chi^4  \psi' [(\frac{\gamma_2}{3}-2\gamma_3) 
R |\nabla w|^2-\gamma_2 Ric(\nabla w,\nabla w)] dv +\mathcal{R}, \nonumber
\end{eqnarray}
with
\begin{eqnarray*}
\mathcal R & = &\int [ (\frac{\gamma_2}{2}+6\gamma_3) \Delta w+6\gamma_3 |\nabla w|^2 ]
 [\psi \Delta \chi^4 +2\psi' \langle \nabla \chi^4, \nabla w \rangle] dv 
+12 \gamma_3 \int (\Delta w+|\nabla w|^2) \psi \langle \nabla w, \nabla \chi^4\rangle dv\\
&&+\int \psi[(\frac{\gamma_2}{3}-2\gamma_3) R \langle \nabla w,\nabla \chi^4 \rangle -\gamma_2  Ric(\nabla w,\nabla \chi^4)] dv,
\end{eqnarray*}
where the argument of $\psi$ has been omitted for simplicity.
\begin{rem} \label{rem0} When $\partial M \not= \emptyset$,   \eqref{00828} still holds for $\chi \in C_0^\infty (M)$: 
this will be useful in Section \ref{blowup}. 
\end{rem}

\medskip \noindent The first use of \eqref{00828} concerns global bounds for weighted $W^{2,2}-$norms in $M$: 
\begin{thm} \label{t:weak-norms} Let $\frac{\gamma_2}{\gamma_3}>\frac{3}{2}$. Assume $\ov{f} =0$ and $\|f\|_{1} \leq C_0$ for some $C_0>0$. Then there exists $C>0$ so that
\begin{equation} \label{00839}\int \frac{(\Delta w)^2+|\nabla w|^4}{[1+(w-\ov{w})^2]^{\frac{2}{3}}} dv
\leq C \end{equation}
for every smooth solution $w$ of $\mathcal{N}(w)=f$ in $M$. Moreover, given $1\leq q<2$ there exists $C>0$ so that
\begin{equation} \label{new0855}
\|w-\ov{w}\|_{W^{2,q}} \leq C
\end{equation} 
for any such solution $w$.
\end{thm}
\begin{pf} Let $\chi \equiv 1$, $c=\ov{w}$ and $\psi \in C^2(\mathbb{R})$ be so that $2\psi'+\psi'' > 0$. Then $\mathcal R=0$ and by a squares completion the (re-normalized) main order term in \eqref{00828} satisfies, thanks to $\beta=\frac{\gamma_2}{\gamma_3}>\frac{3}{2}$, the inequality
\begin{eqnarray}
&& (\beta+12 ) \int \psi' (\Delta w)^2 dv+
 \int [36 \psi'+(\beta+12) \psi''] \Delta w |\nabla w|^2 dv +12 \int (2\psi'+\psi'') |\nabla w|^4  dv \label{n244}\\
&& \geq \int \frac{ 48{ [2\beta -3 -2\delta(\beta+12)]}(\psi')^2 -24(1+2\delta)(\beta+12 ) \psi' \psi'' -  (\beta+12)^2 (\psi'')^2}{48(1-\delta)(2 \psi' + \psi'')} (\Delta w)^2 dv \nonumber \\
&&+12 \delta \int (2\psi'+\psi'') |\nabla w|^4  dv \nonumber
\end{eqnarray}
for any $0<\delta <1$, in view of the positivity of
$$\int [ \frac{36 \psi'+(\beta+12) \psi''}{\sqrt{48(1-\delta)(2 \psi' + \psi'')}}\Delta w+\sqrt{12(1-\delta)(2 \psi' + \psi'')}|\nabla w|^2]^2 \ dv.$$
Fix $0<\delta<\frac{2\beta-3}{4(\beta+12)}$ and set
$\psi(t)=\int_{-\infty}^t \frac{ds}{(M_0+s^2)^{\frac{2}{3}}}$, $M_0\geq 1$. Since
\begin{equation} \label{01513}
|\frac{\psi''}{\psi'}|=\frac{4}{3}  \frac{|t|}{M_0+t^2}\leq \frac{2 }{3 \sqrt{M_0}} ,
\end{equation}
we can find $M_0 \geq 1$ large so that
\begin{equation} \label{01130}
\frac{ 48{ [2\beta-3-2 \delta(\beta+12) ]}(\psi')^2 -24(1+2\delta)(\beta+12 ) \psi' \psi'' -  (\beta+12)^2 (\psi'')^2}{48(1-\delta)(2 \psi' + \psi'')}, \ 12 \delta(2 \psi'+\psi'') \geq \delta^2 \psi' .\end{equation}
Thanks to \eqref{01130} we have that
\begin{eqnarray*}
 (\beta+12) \int \psi' (\Delta w)^2 dv&+&
 \int [36 \psi'+(\beta+12) \psi''] \Delta w |\nabla w|^2 dv + 12 \int (2\psi'+\psi'') |\nabla w|^4  dv \\ 
 & \geq & \delta^2 \int \psi' [(\Delta w)^2 +|\nabla w|^4 ]dv,
\end{eqnarray*}
and then
\begin{eqnarray} \label{1331}
\int \frac{(\Delta w)^2 +|\nabla w|^4 }{[1+(w-\ov{w})^2]^{\frac{2}{3}}}dv \leq C_1 \left(\|f \|_{1} + \int |\nabla w|^2 dv\right)
\end{eqnarray}
for some $C_1>0$ in view of $ M_0^{-\frac{2}{3}}(1+t^2)^{-\frac{2}{3}} \leq \psi' \leq 1$ and $0\leq \psi \leq \int_{\mathbb{R}} \frac{ds}{(1+s^2)^{\frac{2}{3}}}$. From \eqref{1331} and  H\"older's inequality we obtain 
\begin{eqnarray*} 
\int |\nabla w|^2 dv &\leq &  \int [1+|w-\ov{w}|^{\frac{2}{3}}]  \frac{|\nabla w|^2}{[1+(w-\ov{w})^2]^{\frac{1}{3}}} dv
\leq \|1+|w-\ov{w}|^{\frac{2}{3}} \|_{2}
\left(\int \frac{|\nabla w|^4}{[1+(w-\ov{w})^2]^{\frac{2}{3}}} dv \right)^{\frac{1}{2}}\\
&\leq & C_1^{\frac{1}{2}} \left(|M|^{\frac{1}{2}}+\|w-\ov{w}\|_{\frac{4}{3}}^{\frac{2}{3}} \right)
\left( \|f \|_{1} + \int |\nabla w|^2 dv \right)^{\frac{1}{2}} \leq
C_2 (1+\int |\nabla w|^2 dv)^{\frac{5}{6}}
\nonumber
\end{eqnarray*}
for some $C_2>0$ in view of Poincar\'e's inequality. By  Young's inequality we then have  $\int |\nabla w|^2 dv \leq C$ for some $C>0$, and in turn by \eqref{1331}  we deduce the validity of \eqref{00839}. 

Similarly, since $W^{2,q}(M)$ embeds continuously into $ L^{\frac{4q}{3(2-q)}}(M)$ by  Sobolev's Theorem, for any $1\leq q<2$ there holds
\begin{eqnarray*}
\int |\Delta w|^q dv &\leq &  \int [1+|w-\ov{w}|^{\frac{2q}{3}}] \frac{|\Delta w|^q}{[1+(w-\ov{w})^2]^{\frac{q}{3}}} dv
\leq \| 1+|w-\ov{w}|^{\frac{2q}{3}} \|_{{\frac{2}{2-q}}} 
\left(\int \frac{(\Delta w)^2}{[1+(w-\ov{w})^2]^{\frac{2}{3}}} dv \right)^{\frac{q}{2}}\\
&\leq & C_3 \left(1  +\|w-\ov{w}\|_{W^{2,q}}^{\frac{2q}{3}} \right)
\end{eqnarray*}
for some $C_3>0$ in view of \eqref{00839}. Since $(\int |\Delta w|^q dv)^{\frac{1}{q}}$ is equivalent to the $W^{2,q}-$norm on the functions in $W^{2,q}(M)$ with zero average, by  Young's inequality we then have the validity of \eqref{new0855} for some uniform $C>0$.
\end{pf}

\medskip \noindent Once global bounds on $W^{2,q}-$norms have been derived for $1 \leq q<2$, we will make use once more of \eqref{00828} to establish Caccioppoli-type estimates:
\begin{thm} \label{Caccioppoli} Let $\frac{\gamma_2}{\gamma_3}>\frac{3}{2}$. There exist $C>0$ and $k_0>0$ so that
\begin{equation} \label{01510}\int_{\{ |w-c|<k \} \cap B_\rho} [(\Delta w)^2+|\nabla w|^4] dv
\leq \frac{C}{(r-\rho)^4} \int_{B_r \setminus B_\rho} (1+(w-c)^4) \ dv+C k \int_{B_r} |f| \ dv
\end{equation}
for any $0<\rho<r<i_0$, $c \in \mathbb{R}$, $k \geq k_0$ and any smooth solution $w$ of $\mathcal{N}(w)=f$ in $M$ with $\ov{f}=0$. Here $B_\rho$ and $B_r$ are centered at the same point.
\end{thm}
\begin{pf} Let $\chi \in C_0^\infty (B_r)$ be so that $0\leq \chi \leq 1$, $\chi=1$ in $B_\rho$ and
\begin{equation} \label{02141}
(r-\rho)|\nabla \chi|+(r-\rho)^2 |\Delta \chi| \leq C  .
\end{equation}
Letting $\Psi$ be the odd extension to $\R$ of
$$\Psi(s)=\left\{ \begin{array}{ll} s &\hbox{if }0\leq s\leq 1\\
8-9s^{-\frac{1}{3}} +2 s^{-1} & \hbox{if }s>1, \end{array} \right.$$
we have that $\Psi \in C^2(\mathbb{R})$ satisfies $|\Psi''| \leq 4 \Psi'$, $0<\Psi' \leq 1$, $\Psi^2 \leq 8^2 s^2 \Psi'$ and $\Psi^4 \leq 8^4 s^4 (\Psi')^3$ in $\mathbb{R}$. Hence, $\psi(s)=k \Psi(\frac{s}{k})$ is a $C^2-$function so that $0<\psi'\leq 1$, 
\begin{equation} \label{02131}
\sup_{s\in \mathbb{R}} \frac{|\psi''(s)|}{\psi'(s)} \leq \frac{4}{k}
\end{equation}
and
\begin{equation} \label{02131bis}
\sup_{s \in \mathbb{R}} \frac{\psi^2(s)}{s^2 \psi'(s)} \leq 8^2, \qquad \sup_{s \in \mathbb{R}} \frac{\psi^4(s)}{s^4 (\psi'(s))^3} \leq 8^4.
\end{equation}
By  Young's inequality we have that
\begin{eqnarray*}
\int [|\Delta w|+|\nabla w|^2] |\psi| |\Delta \chi^4 | \, dv&\leq& \frac{C}{(r-\rho)^2} \int_{B_r \setminus B_\rho} [|\Delta w|+|\nabla w|^2] \chi^2 |\psi| \\
&\leq & \epsilon \int \psi' \chi^4 [(\Delta w)^2+|\nabla w|^4] dv+\frac{C_\epsilon}{(r-\rho)^4} \int_{B_r \setminus B_\rho} |w-c|^2 \, dv
\end{eqnarray*}
in view of \eqref{02141} and \eqref{02131bis}, where $\psi$ stands for $\psi(w-c)$. Similarly, there holds
\begin{eqnarray*}
&& \int [|\Delta w|+|\nabla w|^2] (\psi'+|\psi|)  |\nabla \chi^4| |\nabla w| dv  \leq \frac{C}{r-\rho} \int_{B_r \setminus B_\rho} [|\Delta w|+|\nabla w|^2] (\psi'+|\psi|)  \chi^3 |\nabla w| dv \\
&& \leq \epsilon  \int  \psi' \chi^4[(\Delta w)^2+|\nabla w|^4]dv+\epsilon \int \psi' \chi^4 |\nabla w|^4 dv
+\frac{C_\epsilon'}{(r-\rho)^4} \int_{B_r \setminus B_\rho}  \frac{(\psi'+|\psi|)^4}{(\psi')^3} dv \\
&&\leq 2 \epsilon \int  \psi' \chi^4[(\Delta w)^2+|\nabla w|^4]dv+\frac{C_\epsilon}{(r-\rho)^4} \int_{B_r \setminus B_\rho}  (1+(w-c)^4) dv, 
\end{eqnarray*}
and
$$\int |\psi| |\nabla w||\nabla \chi^4| dv\leq  \frac{C}{r-\rho} \int_{B_r \setminus B_\rho} |\psi| |\nabla w|\chi^3 dv \leq \epsilon \int \psi' \chi^4 |\nabla w|^4 dv+\frac{C_\epsilon}{(r-\rho)^4} \int_{B_r \setminus B_\rho}  (w-c)^4 dv+C_\epsilon$$
in view of \eqref{02141} and \eqref{02131bis}. In conclusion, for all $\epsilon>0$ there exists $C_\epsilon>0$ so that $\mathcal R$ in \eqref{00828} satisfies
\begin{equation} \label{01501}
|\mathcal R| \leq C \epsilon  \int  \psi' \chi^4[(\Delta w)^2+|\nabla w|^4]dv+ \frac{C_\epsilon}{(r-\rho)^4} \int_{B_r \setminus B_\rho} (1+(w-c)^4) dv
\end{equation}
for some $C>0$. Since $\frac{|\psi''(s)|}{\psi'(s)} $ can be made as small as we need for $k$ large thanks to \eqref{02131}, we are in the same situation as with \eqref{01513} and, arguing as in the proof of Theorem \ref{t:weak-norms}, there exists $k_0>0$ large so that
\begin{eqnarray}\nonumber  
 \Big|(\frac{\gamma_2}{2}+6 \gamma_3) \int \chi^4 \psi' (\Delta w)^2 dv&+&
 \int \chi^4 [18 \gamma_3 \psi'+(\frac{\gamma_2}{2}+6\gamma_3) \psi''] \Delta w |\nabla w|^2 dv \\
 &+& 6 \gamma_3 \int \chi^4 (2\psi'+\psi'') |\nabla w|^4  dv \Big| \geq   \delta^2 \int  \psi' \chi^4 [(\Delta w)^2 +|\nabla w|^4 ]dv  \label{01516} 
\end{eqnarray}
for some $\delta>0$ and all $k\geq k_0$. Since $\int  \psi' \chi^4 |\nabla w|^2 dv \leq \epsilon \int  \psi' \chi^4 |\nabla w|^4 dv+C_\epsilon $ and $|\int f \chi^4 \psi \, dv| \leq 8k \int_{B_r}|f| \ dv$, by inserting \eqref{01501}-\eqref{01516} into \eqref{00828} for $\epsilon>0$ small we deduce the validity of \eqref{01510} for all $k\geq k_0$ in view of $\chi^4 \psi' (w-c) \geq \chi_{\{ |w-c|<k \}\cap B_\rho}$.
\end{pf} 

\medskip \noindent The aim is now  to control the mean oscillation 
$$[w]_{BMO}=\left(\sup_{0<r<i_0} \fint_{B_r} (w-\ov{w}^r)^4 dv\right)^{\frac{1}{4}} $$
of a solution $w$. Our approach in this step heavily relies on the ideas developed in \cite{DHM2}, where Caccioppoli-type estimates like in Theorem \ref{Caccioppoli} were crucial to establish BMO-bounds. We believe that $L^{4,\infty}-$estimates on $\nabla w$ are still true as in \cite{DHM2} but it is not clear which are the optimal bounds for $\Delta w$. We will not  pursue more this line since the following BMO-estimates are enough for our purposes.
\begin{thm} \label{BMO} Let $\frac{\gamma_2}{\gamma_3}>\frac{3}{2}$. Assume $\ov{f}=0$ and $\|f\|_{1} \leq C_0$ for some $C_0>0$. There exists $C>0$ such that for any smooth solution $w$ of $\mathcal{N}(w)=f$ in $M$ one has 
\begin{equation} \label{01745}
[w]_{BMO} \leq C. 
\end{equation}
\end{thm}
\begin{proof} If \eqref{01745} does not hold, we can find smooth solutions $w_n$ of $\mathcal{N} (w_n)=f_n$ so that $[w_n]_{BMO} \to +\infty$ as $n \to +\infty$, with $\ov{f}_n =0$ and $\|f_n\|_1\leq C_0$. By definition we can find $0<r_n<i_0$,  $x_n \in M$ so that
\begin{equation} \label{01236}
\fint_{B_{r_n}(x_n)} (w_n-\ov{w}_n^{r_n})^4 \ dv \geq \frac{1}{2} [w_n]_{BMO}^4.
\end{equation}
Since $[w_n]_{BMO} \to +\infty$ as $n \to +\infty$, up to a subsequence we can assume that $r_n \to 0$ as $n \to +\infty$ in view of 
$$\sup_{n \in \mathbb{N}} \sup_{\delta<r<i_0} \fint_{B_r} (w_n-\ov{w}_n^r)^4 dv<+\infty $$ 
for all $0<\delta \leq i_0$, as it follows by the Poincar\'e-Sobolev's embedding
$$\left( \int_{B_r} |w_n-\ov{w}_n^r|^4 dv \right)^{\frac{1}{4}} \leq C \left( \int_{B_r} |\nabla w_n|^2 dv\right)^{\frac{1}{2}}$$
and Theorem \ref{t:weak-norms}. Letting $exp_{x_n}: B_{i_0}(0) \to B_{i_0}(x_n)$ be the exponential map at $x_n$, for $|y|<\frac{i_0}{r_n}$ introduce the rescaled metric $g_n(y)=g (exp_{x_n}(r_ny))$  and the rescaled functions 
$$u_n(y)=\frac{w_n(exp_{x_n}(r_n y))- \ov{w}_n^{r_n}}{[w_n]_{BMO}}.$$
We have that
\begin{equation} \label{01850}
\int_{B_1(0)} u_n \ dv_{g_n}=0, \quad \int_{B_1(0)} u_n^4 \ dv_{g_n} \geq \frac{\hbox{vol}(B_{r_n}(x_n))}{2r_n^4},\quad 
\int_{B_r(0)} (u_n-\ov{u}_n^r)^4 \ dv_{g_n} \leq \frac{\hbox{vol}(B_{r r_n}(x_n))}{r_n^4} 
\end{equation}
for all $ r<\frac{i_0}{r_n}$ in view of \eqref{01236}, where $\ov{u}_n^r=\fint_{B_r(0)} u_n \ dv_{g_n}$ is the average of $u_n$ on $B_r(0)$ w.r.t. $g_n$. Neglecting the term involving the Laplacian, we can rewrite the estimate \eqref{01510} in terms of $u_n$ as
\begin{equation} \label{01502}
\int_{\{ |u_n-c|<k \} \cap B_\rho(0)} |\nabla u_n|_{g_n}^4 dv_{g_n}
\leq \frac{C}{(r-\rho)^4} \int_{B_r(0) \setminus B_\rho(0)} \left[\frac{1}{[w_n]_{BMO}^4}+(u_n-c)^4 \right]\ dv_{g_n}+\frac{C k \|f_n \|_{1}}{[w_n]_{BMO}^3}  
\end{equation}
for any $0<\rho<r<\frac{i_0}{r_n}$, $c \in \mathbb{R}$ and $k \geq \frac{k_0}{[w_n]_{BMO}}$. Since $\hbox{vol}(B_{r r_n}(x_n))\leq C (rr_n)^4$ for all $0<r<\frac{i_0}{r_n}$ there holds
\begin{equation} \label{01554}
\int_{B_r(0)} (u_n-\ov{u}_n^r)^4 \ dv_{g_n} \leq C r^4 \qquad \forall \ 0<r<\frac{i_0}{r_n}
\end{equation}
thanks to \eqref{01850}, and we can apply \eqref{01502} with $\rho=\frac{r}{2}$ and $c=\ov{u}_n^r$ to get 
\begin{equation} \label{01543}\int_{\{ |u_n-\ov{v}_n^r|<k \} \cap B_{\frac{r}{2}}(0)} |\nabla u_n|_{g_n}^4 dv_{g_n}
\leq C( \frac{1}{[w_n]_{BMO}^4} +1)+\frac{C k \|f_n \|_{1}}{[w_n]_{BMO}^3}  
\end{equation}
in view of \eqref{01554}. Since
$$|\ov{u}_n^r| \int_{B_1(0)} dv_{g_n} \leq \int_{B_1(0)} |u_n-\ov{u}_n^r| \ dv_{g_n}  \leq
C \left(\int_{B_r(0)} (u_n-\ov{u}_n^r)^4 \ dv_{g_n}\right)^{\frac{1}{4}} \left(\int_{B_1(0)} dv_{g_n}\right)^{\frac{3}{4}} \leq C_0 r \int_{B_1(0)} dv_{g_n}$$
for all $1\leq r <\frac{i_0}{r_n}$ in view of \eqref{01850} and \eqref{01554}, we have that $\{|u_n|<k \} \subset \{|u_n-\ov{u}_n^r|<2k\}$ and then 
\begin{equation} \label{01757} \int_{\{ |u_n|<k \} \cap B_{\frac{r}{2}}(0)} |\nabla u_n|_{g_n}^4 dv_{g_n}
\leq C \left(1+\frac{ k \|f_n \|_{1}}{[w_n]_{BMO}^3}  \right)
\end{equation}
for all $1\leq r <\frac{i_0}{r_n}$ and $k > C_0 r$ in view of \eqref{01543}. From \eqref{01757} and $\int_{B_1(0)} u_n \ dv_{g_n}=0$ it is rather classical to derive that $u_n$ is uniformly bounded in $W^{1,q}_{loc}(\mathbb{R}^4)$ for all $1 \leq q<4$, see for example Lemma 2.3 in \cite{DHM2} and the proof of  Lemma 10 in \cite{DHM1}. Up to a subsequence, we can assume that $u_n \rightharpoonup u$ in $W^{1,q}_{loc}(\mathbb{R}^4)$ for all $1\leq q <4$. Letting $\varphi_k \in C_0^\infty (-k,k)$ so that $\varphi_k(s)=s$ for $s\in [-\frac{k}{2},\frac{k}{2}]$, by $|\varphi_k'|\leq C_k$ and \eqref{01502} we deduce that
\begin{equation} \label{01254}
\int_{B_\rho(0)} |\nabla \varphi_k(u_n-c)|^4 dx
\leq \frac{C}{(r-\rho)^4} \int_{B_r(0) \setminus B_\rho(0)} \left[\frac{1}{[w_n]_{BMO}^4}+(u_n-c)^4 \right] \ dx+\frac{C k \|f_n \|_{1}}{[w_n]_{BMO}^3}  
\end{equation}
for any $0<\rho<r<\frac{i_0}{r_n}$, $c \in \mathbb{R}$ and $k \geq \frac{k_0}{[w_n]_{BMO}}$. Since $\nabla \varphi_k(u_n-c) \rightharpoonup \nabla \varphi_k(u-c)$ in $L^4_{loc}(\mathbb{R}^4)$ in view of $u_n \to u$ in $L^q_{loc}(\mathbb{R}^4)$ for all $q\geq 1$ as $n \to +\infty$, by weak lower semi-continuity of the $L^4-$norm we can let $n \to +\infty$ in \eqref{01254} to get
$$\int_{\{ |u-c|<\frac{k}{2} \} \cap B_\rho(0)} |\nabla u|^4 dx
\leq \frac{C}{(r-\rho)^4} \int_{B_r(0) \setminus B_\rho(0)} (u-c)^4 \ dx  $$
and then by the Monotone Convergence Theorem as $k \to +\infty$
\begin{equation} \label{01808}
\int_{B_\rho(0)} |\nabla u|^4 dx \leq \frac{C}{(r-\rho)^4} \int_{B_r(0) \setminus B_\rho(0)} (u-c)^4 \ dx  
\end{equation}
for any $0<\rho<r$, $c \in \mathbb{R}$ and $k >0$. Similarly, by letting $n \to +\infty$ into \eqref{01757} we deduce that
$$ \int_{\{ |u|<\frac{k}{2} \} \cap B_{\frac{r}{2}}(0)} |\nabla u|^4 dx \leq C $$
for all $r \geq 1$ and $k > C_0 r$, and then by the Monotone Convergence Theorem we get $\int_{\mathbb{R}^4} |\nabla u|^4 dx <+\infty$ as $k,r \to +\infty$. Taking $\rho=\frac{r}{2}$ and $c=\fint_{B_r(0) \setminus B_{\frac{r}{2}}(0)} u \, dx$ in \eqref{01808}, by  Poincar\'e's inequality one finally deduces 
$$\int_{B_{\frac{r}{2}}(0)} |\nabla u|^4 dx
\leq \frac{C}{r^4} \int_{B_r (0) \setminus B_{\frac{r}{2}}(0)} (u-c)^4 \ dx  \leq C' \int_{B_r(0) \setminus B_{\frac{r}{2}}(0)} |\nabla u|^4 dx \to 0$$
as $r \to +\infty$ in view of $\int_{\mathbb{R}^4} |\nabla u|^4 dx <+\infty$, leading to $\nabla u=0$ a.e. in $\mathbb{R}^4$. By \eqref{01850} and $g_n \to \delta_{eucl}$ locally uniformly as $n \to +\infty$ we have that $u=0$ a.e. in view $\int_{B_1(0)} u \, dx=0$, in contradiction with $\int_{B_1(0)} u^4 dx \geq \frac{\omega_4}{6}$.
\end{proof}

\section{General ``linear" theory} \label{linear}
\setcounter{equation}{0}  

\noindent We aim to develop a comprehensive theory for the operator $\mathcal{N}$ in \eqref{linearL} when $\frac{\gamma_2}{\gamma_3} \geq 6$. In this section we are interested in existence issues for a general Radon measure $\mu$ and Solutions will be Obtained as Limits of smooth Approximations, from now on referred to as SOLA (see \cite{BBGGPV,BoccGa}). On the other hand since, as we will see, blow-up sequences give rise in the limit to a solution with a linear combination $\mu_s$ of Dirac masses as R.H.S., it will be crucial to establish in the next section the logarithmic behaviour of any of such singular solutions, referred to as a \emph{fundamental solution} of $\mathcal{N}$ corresponding to $\mu_s$. We will guarantee  that SOLA's will be unique just when $\gamma_2=6\gamma_3$.

\medskip \noindent The assumption $\frac{\gamma_2}{\gamma_3} \geq 6$ is crucial to have some monotonicity property on $\mathcal{N}$, expressed by a sign  for the main order term in expressions of the form $\langle \mathcal{N}(w_1)-\mathcal{N}(w_2),w_1-w_2 \rangle$. When $\gamma_2=6 \gamma_3$ the lower-order terms cancel out and uniqueness is in order, as already noticed in \cite{ChangYangAnnals}. The operator $\mathcal{N}(w)$ in \eqref{linearL} is considered here in the following distributional sense:
\begin{eqnarray*}
\langle \mathcal{N}(w), \varphi \rangle&=& \frac{\gamma_2}{2} \int \Delta w \Delta \varphi \, dv- \gamma_2 \int \hbox{Ric}(\n w,\nabla \varphi) dv+6\gamma_3 \int (\Delta w+|\nabla w|^2) \Delta \varphi \, dv \\
&&+12\gamma_3 \int (\Delta w+|\nabla w|^2)\langle \nabla w,\nabla \varphi \rangle dv+(\frac{\gamma_2}{3}	-2\gamma_3) \int R \langle \nabla w,\nabla \varphi \rangle dv 
\end{eqnarray*} 
for all $\varphi \in C^\infty(M)$, provided $\nabla w \in L^3$ and $\nabla^2 w \in L^{\frac{3}{2}}$. We have the following result.
\begin{pro} \label{crucProp}
There holds
\begin{eqnarray}  \langle \mathcal{N}(w_1)-\mathcal{N}(w_2),\varphi \rangle
&=& 3\gamma_3 \int \Delta_{\hat g} p \ \Delta_{\hat g} \varphi  \, dv_{\hat g}+6\gamma_3 \int \langle \nabla^2_{\hat g} p,\nabla^2_{\hat g} \varphi \rangle_{\hat g} dv_{\hat g}+3 \gamma_3  \int |\nabla p|_{\hat g}^2 \langle \nabla p,\nabla \varphi\rangle_{\hat g}  dv_{\hat g}\nonumber \\
&&+(\frac{\gamma_2}{2}-3\gamma_3) \int \Delta p \Delta \varphi \, dv +(2\gamma_3-\frac{\gamma_2}{3}) \int [3 \hbox{Ric}(\n p,\nabla \varphi)-R \langle \nabla p,\nabla \varphi \rangle]
dv  \label{difference}  
\end{eqnarray}
for all $\varphi \in C^\infty(M)$ provided $\mathcal{N}(w_1)$ and $\mathcal{N}(w_2)$ exist in a distributional sense, where $p=w_1-w_2$, $q=w_1+w_2$ and $\hat g=e^q g$.
\end{pro}
\begin{pf}  Notice that when $w_1 = w_2$, $q = 2 w_i$ and hence our notation for the conformal metric $\hat g = e^{q} g$ is consistent with out previous one. 
Since $\hat g=e^q g$ has derivatives in a weak sense up to order two, the Riemann tensor of $\hat g$ and all the geometric quantities which involve at most second-order derivatives make sense. One can easily check that
\begin{eqnarray}\label{12391}
&& dv_{\hat g}=e^{2q} dv, \quad e^q \Delta_{\hat g} w=\Delta w+\langle \nabla q,\nabla w\rangle,\quad e^{2q}|\nabla w|_{\hat g}^4=|\nabla w|^4,\\
&& \nabla^2_{\hat g}w = \nabla^2 w - \frac{1}{2}d w \otimes d q -\frac{1}{2}d q  \otimes d w +\frac{1}{2} \langle \nabla q, \nabla  w \rangle g .
\label{13261}
\end{eqnarray}
Since $w_1=\frac{p+q}{2}$ and $w_2=\frac{q-p}{2}$ we have that
\begin{equation} \label{n1129}
\int [(\Delta w_1+|\nabla w_1|^2)-(\Delta w_2+|\nabla w_2|^2)] \Delta \varphi \, dv= \int (\Delta p+\langle \nabla p,\nabla q \rangle) \Delta \varphi \, dv, 
\end{equation}
and
\begin{eqnarray} \label{n1142}
&& \int \langle (\Delta w_1+|\nabla w_1|^2) \nabla w_1- (\Delta w_2 +|\nabla w_2|^2)\nabla w_2,\nabla \varphi \rangle dv=
\frac{1}{2} \int (\Delta p+\langle \nabla p,\nabla q \rangle)  \langle  \nabla q,\nabla \varphi \rangle dv\\
&&+\frac{1}{4} \int (2\Delta q +|\nabla p|^2+|\nabla q|^2) \langle \nabla p,\nabla \varphi \rangle dv. \nonumber \end{eqnarray}
By \eqref{n1129}-\eqref{n1142} we deduce that
\begin{eqnarray}
\label{n1205}
&& 2\int \langle (\Delta w_1+|\nabla w_1|^2) \nabla w_1- (\Delta w_2 +|\nabla w_2|^2)\nabla w_2,\nabla \varphi \rangle dv\\
&& \qquad+\int [(\Delta w_1+|\nabla w_1|^2)-(\Delta w_2+|\nabla w_2|^2)] \Delta \varphi \, dv =\frac{1}{2}\int \Delta p \Delta \varphi \, dv-\int \langle \nabla^2 p,\nabla^2 \varphi \rangle dv \nonumber\\
&&\qquad \qquad+ \frac{1}{2}\int \Delta_{\hat g} p \Delta_{\hat g} \varphi \, dv_{\hat g}+\int \langle \nabla^2_{\hat g} p,\nabla^2_{\hat g} \varphi \rangle_{\hat g} dv_{\hat g}+
\frac{1}{2} \int  |\nabla p|_{\hat g}^2 \langle \nabla p,\nabla \varphi \rangle_{\hat g} dv_{\hat g}, \nonumber 
\end{eqnarray}
in view of \eqref{12391}-\eqref{13261} and the formula
\begin{eqnarray}\label{n1305}
&& \int \langle \nabla^2_{\hat g} p,\nabla^2_{\hat g} \varphi \rangle_{\hat g} dv_{\hat g}-\int \langle \nabla^2 p,\nabla^2 \varphi \rangle dv=
\int \langle \nabla^2_{\hat g} p,\nabla^2_{\hat g} \varphi \rangle dv-\int \langle \nabla^2 p,\nabla^2 \varphi \rangle dv\\
&&=\int \left[\Delta q  \langle \nabla p,\nabla \varphi \rangle+ \frac{1}{2} \Delta p \langle  \nabla q,\nabla \varphi \rangle +\frac{1}{2} \langle \nabla p,\nabla q \rangle \langle  \nabla q,\nabla \varphi \rangle+\frac{1}{2} |\nabla q|^2 \langle \nabla p,\nabla \varphi \rangle +\frac{1}{2} \langle \nabla p,\nabla q \rangle \Delta \varphi \right] dv. \nonumber
\end{eqnarray}
To establish \eqref{n1305} we simply use \eqref{13261} and an integration by parts to get
\begin{equation} \label{10551}
\int \left[\nabla^2 p(\nabla q, \nabla \varphi)+\nabla^2 \varphi(\nabla q, \nabla p)\right] dv=
\int \langle \nabla q, \nabla \langle \nabla  p, \nabla \varphi\rangle ) dv=-\int \Delta q   \langle \nabla  p, \nabla \varphi\rangle  dv
\end{equation}
for all $\varphi \in C^\infty(M)$, in view of $\nabla p,\nabla q \in L^3$ and $\nabla^2 p,\nabla^2 q \in L^{\frac{3}{2}}$.
Thanks to   Bochner's identity
\begin{equation*}
\hbox{Ric}(\nabla p,\nabla p)=-\langle \nabla p,\nabla \Delta p \rangle-|\nabla^2  p|^2 +  \frac{1}{2}\Delta (|\nabla p|^2), \quad p \in C^3(M),
\end{equation*}
an integration by parts gives that $\int \hbox{Ric}(\nabla p,\nabla p) dv=\int (\Delta p)^2 dv-\int |\nabla^2  p|^2  dv$ and by differentiation
\begin{eqnarray} \label{14122}
\int \hbox{Ric}(\nabla p,\nabla \varphi) dv=\int \Delta p \Delta \varphi \, dv-\int \langle \nabla^2  p, \nabla^2 \varphi \rangle  dv
\end{eqnarray}
for all $\varphi \in C^\infty(M)$, where by density it is enough to assume $\nabla p, \, \nabla^2 p \in L^1 $. By inserting \eqref{14122} into \eqref{n1205}, we then deduce the validity of \eqref{difference}. \end{pf}
\begin{rem} \label{rem1} When $\partial M \not= \emptyset$ notice that the integrations by parts  in \eqref{10551}-\eqref{14122} 
and then \eqref{difference} are still valid for $\varphi \in C_0^\infty(M)$ as long as $\mathcal{N}(u), \ \mathcal{N}(v)$ exist in a distributional sense. 
\end{rem}

\medskip \noindent The usefulness of assumption $\frac{\gamma_2}{\gamma_3} \geq 6$ becomes apparent from the choice $\varphi=p$ in \eqref{difference} since it guarantees that the first four terms in the R.H.S. of \eqref{difference} have all the same sign. When $\gamma_2=6\gamma_3$ there are no lower-order terms and uniqueness is expected. Since in general $p$ is not an admissible function in \eqref{difference}, we will follow the strategy in \cite{GIS,IwSb,Iw} via a Hodge decomposition to build up admissible approximations of $p$ to be used in \eqref{difference}. 

\medskip  \noindent Letting $w_1$ and $w_2$ be smooth functions, consider the \emph{Hodge decomposition}
\begin{equation} \label{15277}
\frac{\nabla p}{(\delta^2+|\nabla p|^2+|\nabla q|^2)^{2\epsilon}}=\nabla \varphi+h,
\end{equation}
where $\epsilon>0$, $0<\delta\leq 1$ and $\varphi,h$ satisfy $\Delta \, \hbox{div} \ h=0$ and $\ov{\varphi} =0$. Notice that
\begin{equation} \label{1926}
\Delta \varphi=\frac{\Delta p}{(\delta^2+|\nabla p|^2+|\nabla q|^2)^{2\epsilon}}
-4 \epsilon \frac{  \nabla^2 p(\nabla p,\nabla p)+\nabla^2 q (\nabla p,\nabla q)}{(\delta^2+|\nabla p|^2+|\nabla q|^2)^{2\epsilon+1}}- \hbox{div } h.
\end{equation}
Even if $\hbox{div } h= { 0}$ when $\partial M=\emptyset$, we prefer to keep this term in order to include later the case $\partial M \not= \emptyset$. The function $\varphi$ is uniquely determined as the smooth solution of
$$\Delta^2 \varphi=\Delta \left[\frac{\Delta p}{(\delta^2+|\nabla p|^2+|\nabla q|^2)^{2\epsilon}}
-4 \epsilon \frac{  \nabla^2 p(\nabla p,\nabla p)+\nabla^2 q (\nabla p,\nabla q)}{(\delta^2+|\nabla p|^2+|\nabla q|^2)^{2\epsilon+1}}\right],\quad \ov{\varphi}=0, $$
in view of \eqref{1926}, and then $h$ is simply defined as $h=\frac{\nabla p}{(\delta^2+|\nabla p|^2+|\nabla q|^2)^{2\epsilon}}-\nabla \varphi$. Given distinct points $p_1,\dots,p_l \in M$ and $\alpha_1,\dots,\alpha_l \in \mathbb{R}$, we want to allow one between functions $w_i$, say $w_2$, to satisfy $w_2 \in C^\infty(M \setminus \{p_1, \dots, p_l\})$ and such that 
\begin{equation} \label{12292}
\lim_{x \to 0} |x|^k |\n^{(k)} (w_2-\alpha_i \log |x|)| =0, \quad  k = 1, 2, 3, 
\end{equation}
 holds in geodesic coordinates near each $p_i$. Let us justify \eqref{15277} more in general (i.e. for $w_1$ smooth 
 and $w_2$ singular) by introducing the Green's function $G(x,y)$ of $\Delta^2$ in $M$, i.e.  the solution of
$$\left\{ \begin{array}{ll}\Delta^2 G(x,\cdot)=\delta_x-\frac{1}{|M|} &\hbox{in } M\\ \int G(x,y) dv(y)=0. & \end{array} \right.$$
For all $F \in C^\infty(M, TM)$ the solution of $\Delta^2 \varphi=\Delta \hbox{div}\; F$ in $M$, $\ov{\varphi}=0$, takes the form 
$$\varphi(x)=\int G(x,y) \Delta \hbox{div}\; F(y) dv(y)=- \int \langle \nabla_y \Delta_y G(x,y), F(y) \rangle \, dv(y).$$
Hence $\nabla \varphi$ can be expressed as the  singular integral
$$\nabla \varphi(x)=- \big(\int \nabla_{xy} \Delta_y G(x,y) [F(y)] dv(y) \big)^{\sharp}=\mathcal K(F),$$
where $\sharp$ stands for the sharp musical isomorphism. Since $M$ is a smooth manifold, by the theory of singular integrals the operator $\mathcal K$ extends from $C^\infty(M, TM)$ to $L^s(M, TM)$ and $\nabla \varphi=\mathcal K(F)$, $h=F-\mathcal K(F)$ provide for the vector field $F$ the Hodge decomposition $F=\nabla \varphi+h$ with 
\begin{equation}\label{1116}
\|\nabla \varphi\|_s+\|h\|_s \leq C(s) \|F\|_s
\end{equation}
for all $s>1$. The key point is that $C(s)$ is locally uniformly bounded  in $(1,+\infty)$, see for example \cite{IwMa}.

\medskip \noindent Since $w_1$ is smooth and $w_2$ satisfies \eqref{12292}, in geodesic coordinates near each $p_i$ there holds
$$  |x|^2(\delta^2+|\nabla p|^2+|\nabla q|^2)=2 \alpha_i^2+o(1) , \quad |\Delta p|+|\nabla ^2 p|+|\nabla ^2 q|=O(\frac{1}{|x|^2}) \qquad \hbox{as } x\to 0,$$
and then $F=\frac{\nabla p}{(\delta^2+|\nabla p|^2+|\nabla q|^2)^{2\epsilon}}$ satisfies $\hbox{div }F =O(\frac{1}{|x|^{2(1-2\epsilon)}})$ as $x \to 0$. Since $w_2$ is smooth away from $p_1,\ldots,p_l$, we have that  $\hbox{div }F \in L^{2(1+2\epsilon)}(M)$ and then by elliptic regularity theory the solution $\varphi$ of $\Delta^2 \varphi=\Delta \hbox{div}\; F$ in $M$, $\ov{\varphi}=0$, is in $ W^{2,2(1+2\epsilon)}(M)$. The Hodge decomposition \eqref{15277} does hold with $h=\frac{\nabla p}{(\delta^2+|\nabla p|^2+|\nabla q|^2)^{2\epsilon}}-\nabla \varphi \in W^{1,2(1+2\epsilon)}(M)$ and by \eqref{1116} $\varphi$ satisfies
\begin{equation} \label{1118}
\|\nabla \varphi\|_ {\frac{4(1-\epsilon)}{1-4\epsilon}} \leq   K \| \frac{\nabla p}{(\delta^2+|\nabla p|^2+|\nabla q|^2)^{2\epsilon}}\|_ {\frac{4(1-\epsilon)}{1-4\epsilon}}\leq K
\|\nabla p\|_{4(1-\epsilon)}^{1-4\epsilon}.
\end{equation}
To show the smallness of $h$ in \eqref{15277} for $\epsilon$ small, we follow the approach introduced in \cite{IwSb} based on a general estimate for commutators in Lebesgue spaces. For the sake of completeness we include it in the Appendix and we just make use here of the following estimate:
\begin{equation} \label{0945}
\|h\|_{\frac{4(1-\epsilon)}{1-4\epsilon}} \leq   K \epsilon  \left(\delta^{1-4\epsilon}+\|\nabla p\|_{4(1-\epsilon)}^{1-4\epsilon}+\|\nabla q \|_{4(1-\epsilon)}^{1-4\epsilon}\right)
\end{equation}
for all $0<\epsilon \leq \epsilon_0$ and $0<\delta\leq 1$, for some $K>0$ and $\epsilon_0>0$ small. Thanks to the Hodge decomposition \eqref{15277} we are now ready to show the following result.
\begin{pro} \label{thm1915}  Let $\frac{\gamma_2}{\gamma_3} \geq 6$ and set 
\begin{equation} \label{eta}
\eta=|\gamma_2-6\gamma_3| \sup_M (|R|+\|\hbox{Ric}\|).
\end{equation}
There exist $\epsilon_0>0$ and $C>0$ so that
\begin{equation} \label{1537}
 \int \frac{|\nabla^2_{\hat g} p|_{\hat g}^2+|\nabla p|_{\hat g}^4}{(|\nabla p|^2+|\nabla q|^2)^{2\epsilon}} dv_{\hat g} \leq C  (\|F_1-F_2\|_{\frac{4(1-\epsilon)}{3}}^{\frac{4(1-\epsilon)}{3}}
+ \eta \| \nabla p\|_{2-4 \epsilon}^{2-4\epsilon}+\epsilon^{\frac{4}{3}} \|F_1\|_{\frac{4(1-\epsilon)}{3}}^{\frac{4(1-\epsilon)}{3}} +\epsilon^{\frac{4}{3}} \|F_2\|_{\frac{4(1-\epsilon)}{3}}^{\frac{4(1-\epsilon)}{3}}+\epsilon^{\frac{2}{3}} )
\end{equation}  
for all $0<\epsilon\leq \epsilon_0$ and all distributional solutions $w_i$ of $\mathcal{N}(w_i)=\hbox{div }F_i$, $i=1,2$, provided that $w_1$ is smooth and either $w_2$ is smooth or satisfies \eqref{12292}. Here $p=w_1-w_2$, $q = w_1 + w_2$ and $\hat g=e^q g$.
\end{pro}
\begin{pf} As already observed, we have that $\varphi \in W^{1,\frac{4(1-\epsilon)}{1-4\epsilon}}(M)\cap W^{2,2(1+2\epsilon)}(M)$. Letting $\varphi_k \in C^\infty (M)$ so that $\varphi_k \to \varphi $ in $ W^{1,\frac{4(1-\epsilon)}{1-4\epsilon}}(M)  \cap  W^{2,2(1+2\epsilon)}(M) $ as $k\to +\infty$, we can use \eqref{difference} with $\varphi_k$: thanks to
\eqref{12391}-\eqref{13261} and
$$|\nabla p|^2+|\nabla q|^2+|\Delta p|+|\nabla^2 p|  \in \bigcap_{1\leq q<2} L^q(M),$$  
let $k\to +\infty$ to get the validity of 
\begin{eqnarray} \label{18371} 
&& 3\gamma_3 \int \Delta_{\hat g} p \ \Delta_{\hat g} \varphi  \, dv_{\hat g}+6\gamma_3 \int \langle \nabla^2_{\hat g} p,\nabla^2_{\hat g} \varphi \rangle_{\hat g} dv_{\hat g}+3 \gamma_3  \int |\nabla p|_{\hat g}^2 \langle \nabla p,\nabla \varphi\rangle_{\hat g}  dv_{\hat g}\\
&&+(\frac{\gamma_2}{2}-3\gamma_3) \int \Delta p \Delta \varphi \, dv +(2\gamma_3-\frac{\gamma_2}{3}) \int [3\hbox{Ric}(\n p,\nabla \varphi)- R \langle \nabla p,\nabla \varphi \rangle] dv =- \int  \langle F_1-F_2, \nabla \varphi \rangle \ dv. \nonumber
\end{eqnarray}
Notice that such a Sobolev regularity of $\varphi$ might fail for a general solution $w_2 \in W^{\theta,2,2)}(M)$, see the definition in \eqref{W}, and this explains why, even tough SOLA lie in $W^{\theta,2,2)}(M)$, in Theorem \ref{thm3.2} we will not prove uniqueness in such a {\em grand Sobolev space}.

\medskip \noindent Setting $\rho=(\delta^2+|\nabla p|^2+|\nabla q|^2)^{-\epsilon}$, by \eqref{15277}-\eqref{1926} we deduce that
\begin{eqnarray}
 && |\Delta_{\hat g} \varphi-(\rho^2 \Delta_{\hat g} p- e^{-q}\hbox{div}\, h)|+|\nabla^2_{\hat g} \varphi-(\rho^2 \nabla_{\hat g}^2 p- \nabla h^{\flat})|_{\hat g} = 
\label{20082}  \\
&& = \epsilon \rho^2 O\left( |\nabla p|_{\hat g} |\nabla q|_{\hat g}+|\nabla q|_{\hat g}^2+|\nabla^2_{\hat g} p|_{\hat g}+|\nabla_{\hat g}^2 q|_{\hat g} \right) +O\left(|\nabla q|_{\hat g} |h|_{\hat g}\right)\nonumber
\end{eqnarray}
and
\begin{eqnarray}
|\Delta \varphi-(\rho^2 \Delta  p- \hbox{div}\, h)|=  \epsilon \rho^2 O\left( |\nabla p| |\nabla q| +|\nabla q|^2+|\nabla^2_{\hat g} p|+|\nabla_{\hat g}^2 q| \right), \label{20122} 
\end{eqnarray}
in view of \eqref{12391}-\eqref{13261}, where $\flat$ stands for the flat musical isomorphism. By \eqref{15277} and \eqref{20082}-\eqref{20122} let us re-write \eqref{18371} as
\begin{eqnarray} \label{20172} 
&& 3 \gamma_3 \int \rho^2(\Delta_{\hat g} p)^2 dv_{\hat g}
+6 \gamma_3  \int \rho^2 |\nabla^2_{\hat g} p|_{\hat g}^2 dv_{\hat g}+3 \gamma_3  \int \rho^2 |\nabla p|_{\hat g}^4 dv_{\hat g}+(\frac{\gamma_2}{2}-3\gamma_3) \int \rho^2(\Delta p)^2  dv\\
&&-3\gamma_3 \int e^{-q} \Delta_{\hat g} p \, \hbox{div}\, h  \, dv_{\hat g}
-6 \gamma_3  \int \langle \nabla^2_{\hat g} p,\nabla h^{\flat} \rangle_{\hat g} dv_{\hat g}-(\frac{\gamma_2}{2}-3\gamma_3) \int \Delta p \, \hbox{div}\, h \, dv \nonumber\\
&&=- \int \langle F_1-F_2, \nabla \varphi \rangle dv+\mathfrak{R}, \nonumber
\end{eqnarray}
where by \eqref{12391}-\eqref{13261} and H\"older's inequality $\mathfrak{R}$ satisfies 
\begin{eqnarray} \label{estR} \nonumber
\mathfrak{R}&=& \epsilon \left(\|\rho \nabla^2_{\hat g} p\|_{2,\hat g} +(\int \rho^2 |\nabla p|_{\hat g}^4 dv_{\hat g})^{\frac{1}{4}} (\int \rho^2 |\nabla q|_{\hat g}^4 dv_{\hat g})^{\frac{1}{4}}\right) O\left[(\int \rho^2 |\nabla p|_{\hat g}^4 dv_{\hat g})^{\frac{1}{4}} (\int \rho^2 |\nabla q|_{\hat g}^4 dv_{\hat g})^{\frac{1}{4}}\right. \\
&&\left.+(\int \rho^2 |\nabla q|_{\hat g}^4 dv_{\hat g})^{\frac{1}{2}}+\|\rho \nabla^2_{\hat g}p \|_{2,\hat g}+\|\rho \nabla^2_{\hat g}q \|_{2,\hat g}\right] +O \Big(\int [|\nabla^2_{\hat g} p|  |\nabla q|+|\nabla p|^3]  |h| dv \Big) \nonumber \\
&& +O\left(\eta \int [|\nabla p|^{2-4\epsilon}+|\nabla p||h|]dv \right).  
\end{eqnarray}
Notice that by \eqref{12391} and H\"older's inequality 
\begin{eqnarray} \label{1358} 
&& \int [|\nabla^2_{\hat g} p| |\nabla q|+|\nabla p|^3] |h| dv  \\
&&= O\Big(\|\rho \nabla^2_{\hat g} p\|_{2,\hat g} (\int \rho^2 |\nabla q|^4_{\hat g}dv_{\hat g})^{\frac{1}{4}}+(\int \rho^2 |\nabla p|^4_{\hat g}dv_{\hat g})^{\frac{3}{4}}
\Big) \|\rho^{-1} \|_{\frac{2(1-\epsilon)}{\epsilon}}^{\frac{3}{2}} \|h\|_{\frac{4(1-\epsilon)}{1-4\epsilon}}
\nonumber \\
&&  = \epsilon \, O\Big(\|\rho \nabla^2_{\hat g} p\|_{2,\hat g} (\int \rho^2 |\nabla q|^4_{\hat g}dv_{\hat g})^{\frac{1}{4}}+(\int \rho^2 |\nabla p|^4_{\hat g}dv_{\hat g})^{\frac{3}{4}}
\Big)  (\delta^{1-\epsilon}+\|\nabla p\|_{4(1-\epsilon)}^{1-\epsilon}+\|\nabla q\|_{4(1-\epsilon)}^{1- \epsilon}),\nonumber
\end{eqnarray}
thanks to \eqref{0945} and
\begin{eqnarray} \label{12541}
 \|\rho^{-1} \|_{\frac{2(1-\epsilon)}{\epsilon}} \leq \| \delta+|\nabla p|+|\nabla q|\|_{4(1-\epsilon)}^{2\epsilon}= O(\delta^{2\epsilon}+ \|\nabla p\|_{4(1-\epsilon)}^{2\epsilon}+\|\nabla q\|_{4(1-\epsilon)}^{2 \epsilon}).
\end{eqnarray}
The difficult term to handle is 
\begin{eqnarray*}
&& 3 \gamma_3 \int e^{-q} \Delta_{\hat g} p \, \hbox{div} \, h  \, dv_{\hat g}+6\gamma_3 \int \langle \nabla^2_{\hat g} p,\nabla h^{\flat} \rangle_{\hat g} dv_{\hat g}+(\frac{\gamma_2}{2}-3\gamma_3) \int \Delta p \, \hbox{div}\, h \, dv\\
&&=3 \gamma_3 \int \langle \nabla q,\nabla p \rangle \hbox{div} \, h  \, dv+6 \gamma_3 \int \langle \nabla^2_{\hat g} p,\nabla h^{\flat} \rangle dv+\frac{\gamma_2}{2} \int \Delta p \, \hbox{div}\, h \, dv
\end{eqnarray*}
in view of \eqref{12391}-\eqref{13261}. For smooth functions $w_1$ and $w_2$, integrating by parts we have that
\begin{eqnarray} \label{120221} 
\qquad 3 \gamma_3 \int \langle \nabla q,\nabla p \rangle \hbox{div} \, h  \, dv+\frac{\gamma_2}{2} \int \Delta p \, \hbox{div}\, h \, dv
=-3 \gamma_3 \int  \langle \nabla \langle \nabla q,\nabla p \rangle, h \rangle \, dv+ \frac{\gamma_2}{2} \int  \Delta p \, \hbox{div}\, h  \, dv, 
\end{eqnarray}
and
\begin{eqnarray} \label{12022} \nonumber
 \int \langle \nabla^2_{\hat g} p,\nabla h^{\flat}\rangle  dv&=&-\int g^{ij}h^k (\nabla^2_{\hat g} p)_{kj;i} \ dv \\
&=& -\int [\langle h, \nabla \Delta p\rangle  +\hbox{Ric}\ (h,\nabla p)]dv+\frac{1}{2} \int [\Delta p \langle \nabla q,h\rangle +\Delta q \langle \nabla p,h\rangle] dv 
  \\
 &=& \int [\Delta p \, \hbox{div} \, h  - \hbox{Ric}\ (h,\nabla p)+
\frac{1}{2} \Delta p \langle \nabla q,h\rangle +\frac{1}{2} \Delta q \langle \nabla p,h\rangle] dv 
\nonumber
\end{eqnarray}
in view of \eqref{13261} and
\begin{eqnarray*}
g^{ij} h^k p_{;jki}  = g^{ij} h^k p_{;jik} + R_{sk}  h^k (\nabla p )^s=\langle h,\nabla \Delta p \rangle + \hbox{Ric}\ (h,\nabla p).
\end{eqnarray*}
Since $\Delta \, \hbox{div} \ h=0$, by H\"older's inequality and \eqref{120221}-\eqref{12022} we then have 
\begin{eqnarray} && 3\gamma_3 \int e^{-q} \Delta_{\hat g} p \, \hbox{div}\, h  \, dv_{\hat g}+6\gamma_3 \int \langle \nabla^2_{\hat g} p,\nabla h^{\flat} \rangle_{\hat g} dv_{\hat g}+ (\frac{\gamma_2}{2}-3\gamma_3) \int \Delta p \, \hbox{div}\, h \, dv \nonumber\\
&=& O\Big (\int |h||\nabla p|dv+  \int  [|\nabla^2_{\hat g}  p||\nabla q| + |\nabla^2_{\hat g}  q| |\nabla p|+|\nabla q|^2 |\nabla p|] |h| dv  \Big) \nonumber \\
&=& O\Big( \|\nabla p\|_{\frac{4(1-\epsilon)}{3}} \|h\|_{\frac{4(1-\epsilon)}{1-4\epsilon}} \Big)
+O\Big( \|\rho \nabla^2_{\hat g} q\|_{2,\hat g} (\int \rho^2 |\nabla p|^4_{\hat g} dv_{\hat g})^{\frac{1}{4}}  \|\rho^{-1} \|^{\frac{3}{2}}_{\frac{2(1-\epsilon)}{\epsilon}} \|h\|_{\frac{4(1-\epsilon)}{1-4\epsilon}} \Big)  \nonumber \\
&&+O\Big( \|\rho \nabla^2_{\hat g} p\|_{2,\hat g}+(\int \rho^2 |\nabla p|^4_{\hat g} dv_{\hat g})^{\frac{1}{4}} (\int \rho^2 |\nabla q|^4_{\hat g} dv_{\hat g})^{\frac{1}{4}} \Big)  (\int \rho^2 |\nabla q|^4_{\hat g} dv_{\hat g})^{\frac{1}{4}}  \|\rho^{-1} \|^{\frac{3}{2}}_{\frac{2(1-\epsilon)}{\epsilon}} \|h\|_{\frac{4(1-\epsilon)}{1-4\epsilon}}  \nonumber \\
&=& \epsilon \, O ( \delta^{2-4\epsilon}+\|\nabla p\|_{4(1-\epsilon)}^{2-4\epsilon}+\|\nabla q\|_{4(1-\epsilon)}^{2-4\epsilon})
+ \epsilon ( \delta^{1-\epsilon}+\|\nabla p\|_{4(1-\epsilon)}^{1-\epsilon}+\|\nabla q\|_{4(1-\epsilon)}^{1-\epsilon}) \times \nonumber \\
&&\times O[ \|\rho \nabla^2_{\hat g} q\|_{2,\hat g} (\int \rho^2 |\nabla p|^4_{\hat g} dv_{\hat g})^{\frac{1}{4}} + \|\rho \nabla^2_{\hat g} p\|_{2,\hat g} (\int \rho^2 |\nabla q|^4_{\hat g} dv_{\hat g})^{\frac{1}{4}} +(\int \rho^2 |\nabla p|^4_{\hat g} dv_{\hat g})^{\frac{1}{4}} (\int \rho^2 |\nabla q|^4_{\hat g} dv_{\hat g})^{\frac{1}{2}} ] \label{12011}   
\end{eqnarray}
in view of \eqref{12391}-\eqref{13261}, \eqref{0945} and \eqref{12541}. When $w_2$ satisfies \eqref{12292}, notice that $p,q \in \displaystyle \bigcap_{1\leq q <2} W^{2,q}(M)$ and $h \in L^{\frac{4(1-\epsilon)}{1-4\epsilon}}(M) \cap W^{1,2(1+2\epsilon)}(M)$. By an approximation argument we see that \eqref{120221}-\eqref{12022} and $\int  \Delta p \ \hbox{div} h  \ dv=0$ still hold for $p$, $q$ and $h$ also in this case, and then \eqref{12011} again follows.

\medskip \noindent As
\begin{eqnarray*}
\|\nabla \varphi\|_{\frac{4(1-\epsilon)}{1-4\epsilon}}^{\frac{4(1-\epsilon)}{1-4\epsilon}}=O\left(\int (\rho^2 |\nabla p|)^{\frac{4(1-\epsilon)}{1-4\epsilon}} dv \right)=
O\left(\int \rho^2 |\nabla p|^4  (\frac{|\nabla p|^2}{\delta^2+|\nabla p|^2+|\nabla q|^2})^{\frac{6 \epsilon}{1-4\epsilon}} dv \right)
=O(\int \rho^2 |\nabla p|_{\hat g}^4 dv_{\hat g})
\end{eqnarray*}
in view of \eqref{1118}, notice that
\begin{equation}\label{new330}
\int \langle F_1-F_2, \nabla \varphi \rangle dv=O\left(\|F_1-F_2\|_{\frac{4(1-\epsilon)}{3}}(\int \rho^2 |\nabla p|_{\hat g}^4 dv_{\hat g})^{\frac{1-4\epsilon}{4(1-\epsilon)}} \right).
\end{equation}
Since
$$\eta \int |\nabla p||h| dv=O(\eta^{2-4\epsilon} \|\nabla p\|_{2-4\epsilon}^{2-4\epsilon} + \epsilon^{\frac{8}{3}}+\frac{1}{\epsilon^{\frac{8}{3}}}\|h\|_{\frac{4(1-\epsilon)}{1-4\epsilon}}^{\frac{4(1-\epsilon)}{1-4\epsilon}}),$$
inserting \eqref{estR}-\eqref{1358} and \eqref{12011}-\eqref{new330} into \eqref{20172}, by Young's inequality and \eqref{0945} one finally gets that
\begin{eqnarray} \label{12421} 
\int \rho^2 \Big[|\nabla^2_{\hat g} p|^2_{\hat g} +|\nabla p|_{\hat g}^4 \Big] dv_{\hat g}&=&
O\big(  \|F_1-F_2\|_{\frac{4(1-\epsilon)}{3}}^{\frac{4(1-\epsilon)}{3}} +\eta \| \nabla p\|_{2-4 \epsilon}^{2-4\epsilon}  \big)\\
&&+\epsilon^{\frac{4}{3}}  O\big(
 \|\rho \nabla^2_{\hat g}q \|_{2,\hat g}^2+ \|\nabla p\|_{4(1-\epsilon)}^{4-4\epsilon} +\|\nabla q\|_{4(1-\epsilon)}^{4-4 \epsilon}
+\epsilon^{-\frac{2}{3}} \big) \nonumber 
\end{eqnarray}
for all $0<\epsilon\leq \epsilon_0$ and $0<\delta \leq 1$, for some $\epsilon_0>0$ small.

\medskip \noindent Since \eqref{12421}  holds for any smooth functions $w_1$ and $w_2$, if we choose $w_2=F_2=0$ then $w_1=p=q$ satisfies
\begin{equation} \label{642}
\int \frac{|\nabla^2_{\tilde g} w_1|^2_{\tilde g}+
|\nabla w_1|_{\tilde g}^4}{(\delta^2+|\nabla w_1|^2)^{2 \epsilon}} dv_{\tilde g}
=O \Big(\|F_1 \|_{\frac{4(1-\epsilon)}
{3}}^{\frac{4(1-\epsilon)}{3}} +\| \nabla w_1 \|_{2-4 \epsilon}^{2-4\epsilon}+ \epsilon^{\frac{4}{3}} \|\nabla w_1 \|_{4(1-\epsilon)}^{4-4\epsilon}+\epsilon^{\frac{2}{3}}\Big)  
\end{equation}
for all $0<\epsilon\leq \epsilon_0$ and $0<\delta \leq 1$, where $\tilde g=e^{w_1} g$. Letting $\delta \to 0^+$ in \eqref{642}, by Fatou's Lemma we deduce that
$$\int \frac{|\nabla^2_{\tilde g} w_1|^2_{\tilde g}+
|\nabla w_1|_{\tilde g}^4}{|\nabla w_1|^{4\epsilon}} dv_{\tilde g}
=O \Big(\|F_1 \|_{\frac{4(1-\epsilon)}
{3}}^{\frac{4(1-\epsilon)}{3}} +\| \nabla w_1\|_{2-4 \epsilon}^{2-4\epsilon}+ \epsilon^{\frac{4}{3}} \|\nabla w_1 \|_{4(1-\epsilon)}^{4-4\epsilon}+\epsilon^{\frac{2}{3}}\Big)  $$
for all $0<\epsilon\leq \epsilon_0$. Since $\int \frac{|\nabla w_1|_{\tilde g}^4}{{|\nabla w_1|^{4\epsilon}}} dv_{\tilde g}=\int |\nabla w_1|^{4(1-\epsilon)}dv,$ 
by Young's inequality we obtain that
\begin{equation} \label{15081}
\int \frac{|\nabla^2_{\tilde g} w_1|^2_{\tilde g}
}{|\nabla w_1|^{4\epsilon}} dv_{\tilde g}+\| \nabla w_1\|_{4(1-\epsilon)}^{4(1- \epsilon)}=O(\|F_1 \|_{\frac{4(1-\epsilon)}{3}}^{\frac{4(1-\epsilon)}{3}} +1).
\end{equation}
If $w_2$ is either smooth or satisfies \eqref{12292}, we can still apply \eqref{12421}  with $w_1=F_1=0$ and get 
\begin{equation} \label{15141}
\int \frac{|\nabla^2_{g^\#} w_2|^2_{g^\#}}{|\nabla w_2|^{4\epsilon}} dv_{g^\#}+\| \nabla w_2 \|_{4(1-\epsilon)}^{4(1-\epsilon)}=O(\|F_2 \|_{\frac{4(1-\epsilon)}{3}}^{\frac{4(1-\epsilon)}{3}} +1)
\end{equation}
for all $0<\epsilon \leq \epsilon_0$, where $g^\#=e^{w_2} g$. Since $\rho\leq |\nabla w_1|^{-2\epsilon}, |\nabla w_2|^{-2\epsilon}$ and
\begin{eqnarray*}
e^{2q}[|\nabla^2_{\hat g} p|_{\hat g}^2+|\nabla^2_{\hat g} q|_{\hat g}^2]&=&2e^{2w_1} |\nabla^2_{\tilde g} w_1|_{\tilde g}^2+2e^{2w_2} |\nabla^2_{g^\#} w_2|_{g^\#}^2 +|dw_1 \otimes dw_2+dw_2 \otimes dw_1 -\langle \nabla w_1,\nabla w_2 \rangle g|^2\\
&&-2 \langle \nabla^2_{\tilde g} w_1+\nabla^2_{g^\#} w_2,dw_1 \otimes dw_2+dw_2 \otimes dw_1 -\langle \nabla w_1,\nabla w_2 \rangle g\rangle
\end{eqnarray*}
in view of \eqref{12391}-\eqref{13261}, by \eqref{15081}-\eqref{15141} we deduce that
\begin{equation} \label{1504}
\| \nabla p \|_{4(1-\epsilon)}^{4(1-\epsilon)}+\| \nabla q \|_{4(1-\epsilon)}^{4(1-\epsilon)}=O(\| \nabla w_1 \|_{4(1-\epsilon)}^{4(1-\epsilon)}+\| \nabla w_2 \|_{4(1-\epsilon)}^{4(1-\epsilon)})
=O( \|F_1\|_{\frac{4(1-\epsilon)}{3}}^{\frac{4(1-\epsilon)}{3}}+\|F_2\|_{\frac{4(1-\epsilon)}{3}}^{\frac{4(1-\epsilon)}{3}}  +1)
\end{equation}
and 
\begin{eqnarray} \label{11531}
\| \rho \nabla^2_{\hat g} p \|_{2,\hat g}^2+\| \rho \nabla^2_{\hat g} q \|_{2,\hat g}^2&=&O\Big(\int \frac{|\nabla^2_{\tilde g} w_1|_{\tilde g}^2}{|\nabla w_1|^{4\epsilon}} dv_{\tilde g}+\int \frac{|\nabla^2_{g^\#} w_2|_{g^\#}^2}{|\nabla w_2|^{4\epsilon}}dv_{g^\#} +\int |\nabla w_1|^{2-2\epsilon} |\nabla w_2|^{2-2\epsilon} dv \Big) \nonumber \\
&=& O(\|F_1\|_{\frac{4(1-\epsilon)}{3}}^{\frac{4(1-\epsilon)}{3}} +\|F_2 \|_{\frac{4(1-\epsilon)}{3}}^{\frac{4(1-\epsilon)}{3}}+1)
\end{eqnarray}
for all $0<\epsilon \leq \epsilon_0$. Inserting \eqref{1504}-\eqref{11531} into \eqref{12421} and letting $\delta \to 0^+$, estimate \eqref{1537} follows by Fatou's Lemma for some $\epsilon_0>0$ small.
\end{pf}
\begin{rem} \label{rem11} When $\partial M \not= \emptyset$, re-consider $G(x,y)$ as the Green function of $\Delta^2$ in $M$ with  boundary conditions $G(x,\cdot)=\partial_\nu G(x,\cdot)=0$ on $\partial M$. The Hodge decomposition \eqref{15277} does hold with $\varphi \in W^{2,2(1+2\epsilon)}_0(M)$ and $h \in W_0^{1,2(1+2\epsilon)}(M)$. Letting $\varphi_k \in C^\infty_0 (M)$ so that $\varphi_k \to \varphi $ in $ W_0^{1,\frac{4(1-\epsilon)}{1-4\epsilon}}(M) \cap  W_0^{2,2(1+2\epsilon)}(M) $ as $k\to +\infty$, thanks to Remark \ref{rem1} we can use \eqref{difference} with $\varphi_k$ and let $k\to +\infty$ to get the validity of \eqref{18371} for $\varphi$. The integrations by parts \eqref{120221}-\eqref{12022} are still valid since $h \in W_0^{1,2(1+2\epsilon)}(M)$, while $\int \Delta p \, \hbox{div}\, h \, dv=0$ does hold provided $w_1-w_2 \in W^{2,1}_0(M)$. Hence, Proposition \ref{thm1915} does hold when $\partial M\not= \emptyset$ provided that we assume $w_1-w_2 \in W^{2,1}_0(M)$.\end{rem}
\noindent Let $L^{\theta,q)}(M, TM)$ be the \emph{grand Lebesgue space}  of all vector fields $F \in \displaystyle \bigcup_{1\leq \tilde q<q} L^{\tilde q}(M,TM)$ with 
$$\|F\|_{\theta, q)}= \sup_{0<\epsilon \leq \epsilon_0} \epsilon^{\frac{\theta}{q}}\|F \|_{q(1-\epsilon)}<+\infty$$
and $W^{\theta,2,2)}$ be the \emph{grand Sobolev space}
\begin{equation} \label{W}
W^{\theta,2,2)}=\{ w \in W^{2,1}(M): \, \ov{w}=0, \|w\|_{W^{\theta,2,2)}}:= \|\Delta w\|_{\theta,2)}+\|\nabla w\|_{\theta,4)}<+\infty\}.
\end{equation}
Let $\mathcal M=\{\mu \hbox{ Radon measure in }M: \ \mu(M)=0\}$. For $\mu \in \mathcal M$ we say that a distributional solution $w$ of $\mathcal{N}(w)=\mu$ in $M$ is a SOLA if $w=\displaystyle \lim_{n \to +\infty} w_n$ a.e., where $w_n$ are smooth solutions of $\mathcal{N}(w_n)=f_n$ with $f_n \in C^\infty(M)$, $\ov{w}_n=\ov{f}_n=0$ and $f_n dv \rightharpoonup \mu$ as $n \to +\infty$. Letting $G_2$ be the Green's function of $\Delta$ in $M$, the function
$$H(\mu)=\int \nabla_x G_2(x,y) d \mu(y)$$
for $\mu \in \mathcal M$ satisfies by Jensen's inequality
\begin{equation} \label{19139} 
\epsilon^{\frac{3}{4}} \|H(\mu) \|_{\frac{4(1-\epsilon)}{3}} \leq   \epsilon^{\frac{3}{4}} |d \mu| \sup_{y \in M} \left( \int  |\nabla_x G_2(x,y)|^{\frac{4(1-\epsilon)}{3}} dv(x) \right)^{\frac{3}{4(1-\epsilon)}} \leq  C |d \mu| \end{equation}
for all $0<\epsilon\leq \epsilon_0$. Therefore, we have that $H: \mathcal M \to L^{1,\frac{4}{3})}(M,TM)$ is a linear bounded operator satisfying the property $\mu=\hbox{div} \ H(\mu)$, and we can now re-phrase Proposition \ref{thm1915} as the following main a-priori estimate.
\begin{pro} \label{thm1915bis} Let $\frac{\gamma_2}{\gamma_3} \geq 6$, $\frac{2}{3} \leq \theta<\frac{4}{3}$ and $\eta$ be given as in \eqref{eta}. There exists $C>0$ such  that
\begin{eqnarray} \label{11411}
\| w_1-w_2 \|_{W^{\theta,2,2)}}& \leq&  C \|F_1-F_2 \|_{\theta,\frac{4}{3})}^{\frac{4-3\theta}{6}} (\|F_1\|_{\theta,\frac{4}{3})}+\|F_2\|_{\theta,\frac{4}{3})}+1)^{\frac{\theta}{2}} \\
&&+C \|F_1-F_2\|_{\theta,\frac{4}{3})}^{\frac{4-3\theta}{12}} (\|F_1\|_{\theta,\frac{4}{3})}+\|F_2\|_{\theta,\frac{4}{3})}+1)^{\frac{4+3 \theta}{12}} \nonumber \\
&&+\eta (\|F_1\|_{\theta,\frac{4}{3})}+\|F_2 \|_{\theta,\frac{4}{3})}+1)^{\frac{1}{3}} \, O(\|\nabla (w_1-w_2)\|_2+ \|\nabla (w_1-w_2)\|_2^{\frac{1}{4}}) \nonumber 
\end{eqnarray}
for all SOLA's $w_1$, $w_2$ of $\mathcal{N}(w_1)=\mu_1 \in \mathcal{M}$, $\mathcal{N}(w_2)=\mu_2 \in \mathcal{M}$, where $F_1=H(\mu_1)$ and $F_2=H(\mu_2)$.  Estimate \eqref{11411} holds even if $w_2$ is a distributional solution which satisfies \eqref{12292}. 
\end{pro}
\begin{pf} Since $w_1$ is a SOLA, by definition let $f_{1,n}$ be the corresponding approximating sequence of $\mu_1=\hbox{div} \, F_1$. Letting $u_{1,n}$ be the smooth solution of $\Delta u_{1,n}=f_{1,n}$ in $M$, $\ov{u}_{1,n} =0$, we have that $u_{1,n}$ is pre-compact in $W^{1,q}(M)$ for all $1\leq q <\frac{4}{3}$, see for example Lemma 1 in \cite{BoccGa} in the Euclidean context, and then the following property does hold:
\begin{equation} \label{10460}
\sup_n \|f_{1,n}\|_1 <+\infty \quad \Rightarrow \quad  H(f_{1,n} dv) \hbox{ pre-compact in }L^q(M),\ 1\leq q <\frac{4}{3}, \end{equation} 
in view of $H(f_{1,n} dv)=\nabla u_{1,n}$. Up to a subsequence, we have that $u_{1,n} \to u_1$ in $W^{1,q}(M)$ for all $1\leq q <\frac{4}{3}$, where $u_1$ is a distributional solution of $\Delta u_1=\mu_1$ in $M$, $\ov{u}_1=0$. By uniqueness $\nabla u_1= H(\mu_1)$ and therefore $w_1=\displaystyle \lim_{n \to +\infty} w_{1,n}$ a.e., where $\mathcal{N} (w_{1,n})=\hbox{div} \, F_{1,n}$ with $F_{1,n}=\nabla u_{1,n} \to F_1$ in  $L^q(M)$ for all $1\leq q<\frac{4}{3}$. 

\medskip \noindent Assume that $w_2$ is either a SOLA or a distributional solution satisfying \eqref{12292} of $\mathcal{N}(w_2)=\mu_2=\hbox{div}\,F_2$. In the first case, let $f_{2,n}$ and $F_{2,n}$ be the corresponding sequences for $w_2$ so that $w_2=\displaystyle \lim_{n \to +\infty} w_{2,n}$ a.e., where $\mathcal{N} (w_{2,n})=\hbox{div} \, F_{2,n}$ with $F_{2,n}\to F_2$ in  $L^q(M)$ for all $1\leq q<\frac{4}{3}$.  In the second case, consider $w_{2,n}=w_2$ for all $n \in \mathbb{N}$. Apply \eqref{1537} to $w_{1,n}$ and $w_{2,n}$ to get by  \eqref{1504}
$$\int \frac{|\nabla^2_{\hat g_n} p_n|_{\hat g_n}^2+|\nabla p_n|_{\hat g_n}^4}{(|\nabla p_n|^2+|\nabla q_n|^2)^{2\epsilon}} dv_{\hat g_n} \leq C$$
in terms of $p_n=w_{1,n}-w_{2,n}$, $q_n=w_{1,n}+w_{2,n}$ and $\hat g_n=e^{q_n}g$. Notice that for $1\leq q<2$ by H\"older's estimate there holds
\begin{eqnarray*}
\int |\Delta p_n|^q dv \leq C \left(\int \frac{(\Delta_{\hat g_n} p_n)^2+|\nabla q_n|_{\hat g_n}^2|\nabla p_n|_{\hat g_n}^2}{(|\nabla p_n|^2+|\nabla q_n|^2)^{2\epsilon}} dv_{\hat g_n} \right)^{\frac{q}{2}} \left(\int (|\nabla p_n|^2+|\nabla q_n|^2)^{\frac{2\epsilon q}{2-q}} dv \right)^{\frac{2-q}{2}}
\end{eqnarray*}
in view of \eqref{12391}, and then $p_n$ is uniformly bounded in $W^{2,q}(M)$ for all $1\leq q<2$ thanks to \eqref{1504}. By Rellich's Theorem we deduce that $p_n \to w_1-w_2$ in $W^{1,q}(M)$ for all $1\leq q <4$. Letting $n \to +\infty$ into \eqref{1537} applied to $w_{1,n}$ and $w_{2,n}$, by Fatou's Lemma we get the validity of 
\begin{equation} \label{new1537}
 \int \frac{|\nabla^2_{\hat g} p|_{\hat g}^2+|\nabla p|_{\hat g}^4}{(|\nabla p|^2+|\nabla q|^2)^{2\epsilon}} dv_{\hat g} \leq C  (\|F_1-F_2\|_{\frac{4(1-\epsilon)}{3}}^{\frac{4(1-\epsilon)}{3}}
+ \eta \| \nabla p\|_{2-4 \epsilon}^{2-4\epsilon}+\epsilon^{\frac{4}{3}} \|F_1\|_{\frac{4(1-\epsilon)}{3}}^{\frac{4(1-\epsilon)}{3}} +\epsilon^{\frac{4}{3}} \|F_2\|_{\frac{4(1-\epsilon)}{3}}^{\frac{4(1-\epsilon)}{3}}+\epsilon^{\frac{2}{3}} )
\end{equation}  
for all $0<\epsilon\leq \epsilon_0$ and for all distributional solutions $w_i$ of $\mathcal{N}(w_i)=\hbox{div}\, F_i$, $i=1,2$, provided $w_1$ is a SOLA and $w_2$ is either a SOLA or satisfies \eqref{12292}, where $p=w_1-w_2$, $q=w_1+w_2$ and $\hat g=e^q g$. Re-written \eqref{new1537} as
\begin{eqnarray*}
\int \frac{(\Delta p+\langle \nabla q ,\nabla p \rangle)^2+|\nabla p|^4}{(|\nabla p|^2+|\nabla q|^2)^{2\epsilon}} dv & \leq & C( \epsilon^{-\theta} \|F_1-F_2\|_{\theta,\frac{4}{3})}^{\frac{4(1-\epsilon)}{3}}+  \eta \| \nabla p\|_{2-4 \epsilon}^{2-4\epsilon}) \\
&&+C \epsilon^{\frac{4}{3}-\theta} ( \|F_1\|_{\theta,\frac{4}{3})}^{\frac{4(1-\epsilon)}{3}} +\|F_2\|_{\theta,\frac{4}{3})}^{\frac{4(1-\epsilon)}{3}}+\epsilon^{\theta-\frac{2}{3}} ) 
\end{eqnarray*}  
in view of \eqref{12391}, by Young's inequality we deduce that
\begin{eqnarray} \label{10431}
\hspace{0.3cm}
&& \int |\Delta p+\langle \nabla q ,\nabla p \rangle|^{2(1-\epsilon)} dv +\int |\nabla p|^{4(1-\epsilon)}  dv  
\leq C \int [  (\Delta p+\langle \nabla q ,\nabla p \rangle)^2+|\nabla p|^4]^{1-\epsilon}  dv \\
&&=O \left(\int \frac{(\Delta p+\langle \nabla q ,\nabla p \rangle)^2+|\nabla p|^4}{(|\nabla p|^2+|\nabla q|^2)^{2\epsilon}}dv \right)+\epsilon O\left(  \|\nabla p\|_{4(1-\epsilon)}^{4 (1-\epsilon)}+\|\nabla q\|_{4(1-\epsilon)}^{4 (1-\epsilon)} \right) \leq C \eta \| \nabla p\|_{2-4 \epsilon}^{2-4\epsilon}  \nonumber\\
&&+C \epsilon^{-\theta} \|F_1-F_2\|_{\theta,\frac{4}{3})}^{\frac{4(1-\epsilon)}{3}} +C \epsilon^{\frac{4}{3}-\theta} ( \|F_1\|_{\theta,\frac{4}{3})}^{\frac{4(1-\epsilon)}{3}} +\|F_2\|_{\theta,\frac{4}{3})}^{\frac{4(1-\epsilon)}{3}}+\epsilon^{\theta-\frac{2}{3}} ) \nonumber
\end{eqnarray}  
for $0<\epsilon \leq \epsilon_0$ in view of \eqref{1504}. If $F_1 \not= F_2$, let $\epsilon_\delta>0$ be defined as 
$$\epsilon_\delta=\delta (\frac{\|F_1-F_2 \|_{\theta,\frac{4}{3})}}{\|F_1\|_{\theta,\frac{4}{3})} +\|F_2 \|_{\theta,\frac{4}{3})}+1})$$
for $0<\delta\leq \epsilon_0$. Since $0<\epsilon_\delta \leq \delta \leq \epsilon_0$ and $\|\cdot \|_{q(1-\delta)}=O( \|\cdot \|_{q(1-\epsilon_\delta)})$ by H\"older's inequality, inserting $\epsilon_\delta$ into \eqref{10431} we deduce that
\begin{eqnarray} \label{14561}
&& \| \Delta p+\langle \nabla q,\nabla p \rangle \|_{\theta, 2)}=\sup_{0<\delta\leq \epsilon_0} \delta^{\frac{\theta}{2}} \|\Delta p+\langle \nabla q,\nabla p \rangle \|_{2(1-\delta)}
=O(\sup_{0<\delta\leq \epsilon_0} \delta^{\frac{\theta}{2}} \|\Delta p+\langle \nabla q,\nabla p \rangle \|_{2(1-\epsilon_\delta)})\\
&&= \|F_1-F_2 \|_{\theta,\frac{4}{3})}^{\frac{4-3\theta}{6}} O(\|F_1 \|_{\theta,\frac{4}{3})}+\|F_2\|_{\theta,\frac{4}{3})}+1)^{\frac{ \theta}{2}}+\eta \, O(\|\nabla p\|_2+ \|\nabla p\|_2^{\frac{1}{2}})
\nonumber
\end{eqnarray} 
and
\begin{eqnarray} \label{12203}
\| \nabla p \|_{\theta, 4)}&=& \sup_{0<\delta\leq \epsilon_0} \delta^{\frac{\theta}{4}} \|\nabla p\|_{4(1-\delta)}
=O(\sup_{0<\delta\leq \epsilon_0} \delta^{\frac{\theta}{4}} \|\nabla p\|_{4(1-\epsilon_\delta)})\\
&=& \|F_1-F_2\|_{\theta,\frac{4}{3})}^{\frac{4-3\theta}{12}} O( \|F_1\|_{\theta,\frac{4}{3})}+\|F_2\|_{\theta,\frac{4}{3})}+1)^{\frac{\theta}{4}}+\eta \, O(\|\nabla p\|_2^{\frac{1}{2}}+ \|\nabla p\|_2^{\frac{1}{4}}). \nonumber 
\end{eqnarray} 
Considering as above the two cases $w_1=F_1=0$ and $w_2=F_2=0$ by \eqref{12203} and  Young's inequality we obtain that
$$\| \nabla q \|_{\theta, 4)}=O(\| \nabla w_1 \|_{\theta, 4)}+\| \nabla w_2 \|_{\theta, 4)})=O  (\|F_1\|_{\theta,\frac{4}{3})}+\|F_2\|_{\theta,\frac{4}{3})}+1)^{\frac{1}{3}},$$
which inserted into \eqref{14561} by H\"older's inequality gives 
\begin{eqnarray*} 
&& \| \Delta p\|_{\theta, 2)} = O(\| \Delta p+\langle \nabla q,\nabla p \rangle \|_{\theta, 2)}+\| \nabla p \|_{\theta, 4)}\| \nabla q \|_{\theta, 4)})\\
&& =  \|F_1-F_2 \|_{\theta,\frac{4}{3})}^{\frac{4-3\theta}{6}} O(\|F_1\|_{\theta,\frac{4}{3})}+\|F_2\|_{\theta,\frac{4}{3})}+1)^{\frac{\theta}{2}} +
\|F_1-F_2 \|_{\theta,\frac{4}{3})}^{\frac{4-3\theta}{12}} O(\|F_1\|_{\theta,\frac{4}{3})}+\|F_2 \|_{\theta,\frac{4}{3})}+1)^{\frac{4+3 \theta}{12}}  \\
&&+\eta (\|F_1\|_{\theta,\frac{4}{3})}+\|F_2\|_{\theta,\frac{4}{3})}+1)^{\frac{1}{3}} \, O(\|\nabla p\|_2+ \|\nabla p\|_2^{\frac{1}{4}}).
\end{eqnarray*} 
Therefore \eqref{11411} has been established. \end{pf}

\medskip

\noindent We have the following general result of independent interest.
\begin{thm} \label{thm3.2}
Let $\frac{\gamma_2}{\gamma_3} \geq 6$. For any $\mu \in \mathcal M$ there exists a SOLA $w$ of $\mathcal{N}(w)=\mu$ in $M$ so that $w \in W^{1,2,2)}$. When $\gamma_2=6\gamma_3$ such a SOLA is unique.
\end{thm}
\begin{pf}  Since $\eta=0$ when $\gamma_2=6\gamma_3$, uniqueness directly follows from estimate \eqref{11411} and we are just concerned with the existence issue. Letting $\rho_n$ be a sequence of mollifiers in $[0,+\infty)$, define the approximate measures $\mu_n=(f_n-\ov{f}_n) dv$, where $f_n(x)=\int \rho_n(d(x,y)) d \mu(y)$ are smooth functions. Since $\mu_n \rightharpoonup \mu$, by \eqref{19139} and \eqref{10460} we have that $F_n=H(\mu_n)$ is uniformly bounded in $L^{1,\frac{4}{3})}(M,TM)$ and is pre-compact in $L^q(M)$ for all $1\leq q<\frac{4}{3}$.
Up to a subsequence, it is easily seen that $F_n$ is a Cauchy sequence in $L^{\theta,\frac{4}{3})}(M,TM)$ for all $\theta>1$. In order to solve $\mathcal{N}(w_n)=f_n$ in $M$, notice that $\mathcal{N}(w)=\frac{J'(w)}{4}$, where
\begin{eqnarray*}
J(w)&=&\gamma_2 \int  (\Delta w)^2 dv-2\gamma_2 \int \hbox{Ric}(\nabla w,\nabla w) dv +12\gamma_3 \int (\Delta w + |\nabla w|^2 )^2 dv\\
&&+(\frac{2}{3}\gamma_2 - 4 \gamma_3) \int R|\nabla w|^2 dv, \quad w \in W^{2,2}(M) .
\end{eqnarray*}
Since by squares completion 
$$\beta \int  (\Delta w)^2 dv+12 \int (\Delta w + |\nabla w|^2 )^2 dv \geq \frac{24+\beta-\sqrt{576+\beta^2}}{2} \int [(\Delta w)^2+|\nabla w|^4]dv$$
with $\beta=\frac{\gamma_2}{\gamma_3}>0$, the functional $J(w)-4 \int f w\, dv$ is easily seen to attain a minimizer in $W^{2,2}(M) \cap \{ \ov{w}=0 \}$ as long as $f \in L^q(M)$ for some $q>1$. So we can construct $w_n \in W^{2,2}(M)$ solutions of $\mathcal{N}(w_n)=f_n$ in $M$, $\ov{w}_n=0$, which are smooth thanks to \cite{UVMRL}. Estimate \eqref{12203} provides by Young's inequality  
$$\| \nabla w_n \|_{1,4)} = O\Big( \|F_n \|_{1,\frac{4}{3})}^{\frac{1}{12}} (\|F_n \|_{1,\frac{4}{3})}+1)^{\frac{1}{4}}+1 \Big).$$
Therefore, by \eqref{11411} $w_n$ is a bounded sequence in $W^{1,2,2)}$. In particular, $w_n$ is uniformly bounded in $W^{2,q}(M)$ for all $1\leq q<2$ and by Rellich's Theorem we deduce that, up to a subsequence, $w_n \to w$ in $W^{1,q}(M)$ for all $1\leq q <4$. Since $\|\nabla(w_n-w_m)\|_2 \to 0$ as $n,m\to +\infty$, we can use again \eqref{11411} to show that $w_n$ is a Cauchy sequence in $W^{\theta,2,2)}$ for $1<\theta<\frac{4}{3}$. Then $w$ is a SOLA of $\mathcal{N}(w)=\mu$ in $M$ with $w \in W^{1,2,2)}$ by the boundedness of $w_n$ in $W^{1,2,2)}$.
\end{pf}
\begin{rem} \label{rem12}  Let $\partial M \not= \emptyset$ and $\Phi \in C^\infty(\overline{M})$. For a Radon measure $\mu$ on $M$ we say that a distributional solution $w$ of $\mathcal{N}(w)=\mu$ in $M$, $w=\Phi$ and $\partial_\nu w=\partial_\nu \Phi$ on $\partial M$, is a SOLA if $w=\displaystyle \lim_{n \to +\infty} w_n$ a.e., where $w_n$ are smooth solutions of $\mathcal{N}(w_n)=f_n$ in $M$, $w_n=\Phi$ and $\partial_\nu w_n=\partial_\nu \Phi$ on $\partial M$, for $f_n \in C^\infty(\overline{M})$ so that  $f_n dv \rightharpoonup \mu$ as $n \to +\infty$. The map $H$ is defined by using as $G_2(x,y)$ the Green function of $\Delta$ in $M$ with zero Dirichlet boundary condition on $\partial M$. By Remark \ref{rem11} we have that Proposition \ref{thm1915bis} still holds in this context provided $w_1-w_2 \in W^{2,1}_0(M)$ and Theorem \ref{thm3.2} does hold providing a SOLA $w \in W^{1,2,2)}(M)$ of $\mathcal{N}(w)=\mu$ in $M$, $w=\Phi$ and $\partial_\nu w=\partial_\nu \Phi$ on $\partial M$ for any Radon measure $\mu$.
\end{rem}

\section{Fundamental solutions} \label{s:fund-sol}
\setcounter{equation}{0}  

\noindent Let $\mu_s=\displaystyle \sum_{i=1}^l \b_i  \delta_{p_i}$ be a linear combination of Dirac masses centred at distinct points $p_1,\ldots,p_l \in M$. Given $U$ as in \eqref{Udef}, the parameters $\beta_1,\ldots,\beta_l \not= 0$ are chosen to satisfy
\begin{equation} \label{balance}
\sum_{i=1}^l \b_i= \int U dv. 
\end{equation} 
Since \eqref{balance} guarantees that $\mu_s -U \in \mathcal{M}$, for $\frac{\gamma_2}{\gamma_3}\geq 6$ we can apply Theorem \ref{thm3.2} to find a SOLA $w_s \in W^{1,2,2)}(M)$ (recall \eqref{W}) of $\mathcal{N}(w_s)=\displaystyle \sum_{i=1}^l \b_i  \delta_{p_i}-U$ in $M$, referred to as a fundamental solution corresponding to $\mu_s$. Unless $\gamma_2=6\gamma_3$, fundamental solutions $w_s$  corresponding to $\mu_s$ are not unique and the aim now is to establish a logarithmic behaviour of each $w_s$, no matter whether uniqueness holds or not.

\medskip \noindent Since
$$ \frac{d}{dx}[(\gamma_2+12 \gamma_3)x +18\gamma_3 x^2 +6\gamma_3 x^3] =(\gamma_2+12 \gamma_3)+36 \gamma_3 x +18 \gamma_3 x^2$$
has a given sign in view of $\Delta=-72 \gamma_3^2(\frac{\gamma_2}{\gamma_3}-6) \leq 0$, let $\a_i = \a(\b_i) \not= 0$ be the unique solution of
\begin{equation}\label{eq:alpha-beta}
-4 \pi^2  [(\gamma_2+12 \gamma_3)\alpha +18\gamma_3 \alpha^2 +6\gamma_3 \alpha^3]= \b_i.
\end{equation}
The function 
\begin{equation}\label{eq:w0}
w_0(x) = \sum_{i=1}^l \a_i \log \tilde{d}(x,p_i)
\end{equation}
is an approximate solution of $\mathcal{N}(w)=\displaystyle \sum_{i=1}^l \b_i  \delta_{p_i}-U$ in $M$, where $\tilde{d}(x,p_i)$ stands for the distance function 
smoothed away from $p_i$. Since $w_0$ satisfies \eqref{12292} and $\mathcal{N}(w_s)-\mathcal{N}(w_0)$ is sufficiently integrable, we can let $\epsilon \to 0$ in estimate \eqref{new1537} to obtain $W^{2,2}-$estimates w.r.t. $\hat g=e^{w_s+w_0}g$. Once re-written as $W^{2,2}-$estimates w.r.t. $g_0=e^{2w_0}g$, the argument in \cite{UVMRL} can be adapted to annular regions around the  singularities to show that such weighted $W^{2,2}$-estimates imply the validity of \eqref{12292} for $w_s$ too.

\medskip \noindent Concerning the role of $w_0$ we have the following result.
\begin{lem}\label{l:log-coeff}
The function $w_0$ in \eqref{eq:w0} is a distributional solution of
\begin{equation} \label{18351}
\mathcal{N}(w_0)= \displaystyle \sum_{i=1}^l \b_i  \delta_{p_i}+f_0
\end{equation}
with { $f_0- \gamma_2 \hbox{div}[\hbox{Ric}(\cdot,\nabla w_0)]-(2\gamma_3-\frac{\gamma_2}{3}) \hbox{div}(R \nabla w_0) \in L^\infty(M)$.} \end{lem}
\begin{pf} $w_0$ is a radial function in a neighbourhood of $p_i$, so in geodesic coordinates it satisfies
$$\Delta w_0=\frac{2 \alpha_i}{|x|^2},\quad |\nabla w_0|^2=\frac{\alpha_i^2}{|x|^2}, \quad
(\Delta w_0+|\nabla w_0|^2)\nabla w_0=(2+\alpha_i)\alpha_i^2 \frac{x}{|x|^4}$$
for all $x \not=0$.
Since
\begin{eqnarray*}
\mathcal{N}(w_0) & = & (\frac{\gamma_2}{2} +6\gamma_3)\Delta^2 w_0+6 \gamma_3 \Delta (|\nabla w_0|^2)-12\gamma_3 \hbox{div}[(\Delta w_0+|\nabla w_0|^2)\nabla w_0]+\gamma_2 \hbox{div}[\hbox{Ric}(\cdot,\nabla w_0)] \\ 
& &+(2\gamma_3-\frac{\gamma_2}{3}) \hbox{div}(R \nabla w_0)= \gamma_2 \hbox{div}[\hbox{Ric}(\cdot,\nabla w_0)]+(2\gamma_3-\frac{\gamma_2}{3}) \hbox{div}(R \nabla w_0)
\end{eqnarray*}
near $p_i$ and $\mathcal{N}(w_0)$ is a bounded function away from $p_1,\ldots,p_l$, we have that $w_0$ solves $\mathcal{N}(w_0)=f_0$ in $M\setminus \{p_1,\ldots,p_l\}$, with $f_0- \gamma_2 \hbox{div}[\hbox{Ric}(\cdot,\nabla w_0)] -(2\gamma_3-\frac{\gamma_2}{3}) \hbox{div}(R \nabla w_0)\in L^\infty(M)$.
 
\medskip \noindent Given $\epsilon>0$ small and $\varphi \in C^\infty(M)$, we have that
\begin{eqnarray*}
&&\int_{M \setminus \cup_{i=1}^l B_\epsilon(p_i)} f_0 \varphi dv = \int_{M \setminus \cup_{i=1}^l B_\epsilon(p_i)} \mathcal{N}(w_0) \varphi dv\\
&& =
- \sum_{i=1}^l \oint_{\partial B_\epsilon(p_i)}\left[(\frac{\gamma_2}{2}+6\gamma_3)\partial_\nu \Delta w_0 
+6 \gamma_3 \partial_\nu |\nabla w_0|^2 -12 \gamma_3 (\Delta w_0+|\nabla w_0|^2) \partial_\nu w_0 \right]\varphi d \sigma\\
&&+ \int_{M \setminus \cup_{i=1}^l B_\epsilon(p_i)}\left[(\frac{\gamma_2}{2}+6\gamma_3)  \Delta w_0 \Delta \varphi+6\gamma_3|\nabla w_0|^2 \Delta \varphi+12 \gamma_3  (\Delta w_0+|\nabla w_0|^2) \langle \nabla w_0,\nabla \varphi \rangle \right]dv\\
&&- \int_{M \setminus \cup_{i=1}^l B_\epsilon(p_i)} \left[\gamma_2 \hbox{Ric}(\nabla w_0,\nabla \varphi) +(2\gamma_3-\frac{\gamma_2}{3})  R \langle \nabla w_0,\nabla \varphi\rangle \right]dv+o_\epsilon(1), 
\end{eqnarray*}
where $o_\epsilon(1)\to 0$ as $\epsilon \to 0^+$. Since 
$$\partial_\nu \left[(\frac{\gamma_2}{2}+6\gamma_3)\Delta w_0+6\gamma_3|\nabla w_0|^2\right]=-\frac{2\alpha_i}{\epsilon^3}[\gamma_2+12\gamma_3+6\alpha_i \gamma_3],\quad
(\Delta w_0+|\nabla w_0|^2) \partial_\nu w_0=\frac{2\alpha_i^2 +\alpha_i^3}{\epsilon^3}$$
on $\partial B_\epsilon(p_i)$, as $\epsilon \to 0^+$ we get that
\begin{eqnarray*} 
&& \int [(\frac{\gamma_2 }{2}+6\gamma_3)  \Delta w_0 \Delta \varphi+6\gamma_3 |\nabla w_0|^2 \Delta \varphi+12\gamma_3  (\Delta w_0+|\nabla w_0|^2) \langle \nabla w_0,\nabla \varphi \rangle ]dv \\
&&- \int [\gamma_2 \hbox{Ric}(\nabla w_0,\nabla \varphi) +(2\gamma_3-\frac{\gamma_2}{3})  R \langle \nabla w_0,\nabla \varphi\rangle ]dv = \sum_{i=1}^l \b_i  \varphi(p_i)+\int f_0 \varphi \, dv \nonumber
\end{eqnarray*}
for all $\varphi \in C^\infty(M)$ in view of \eqref{eq:alpha-beta}, i.e. $w_0$ is a distributional solution of \eqref{18351}. 
\end{pf}
\begin{rem} \label{rem2} Let $\Phi \in C^\infty(\overline{B_r(p_i)})$, $i=1,\dots,l$, { so that $\Phi=0$ near $p_i$} and assume that $\{p_1,\dots,p_l\} \cap \overline{B_r(p_i)}=\{p_i\}$. 
Letting $-4 \pi^2  [(\gamma_2+12 \gamma_3)\alpha_i +18\gamma_3 \alpha_i^2 +6\gamma_3 \alpha_i^3]= \b_i$, choose $w_0(x) = \a_i \log \tilde{d}(x,p_i)$ in such a way that $w_0=0$ near $\partial B_r(p_i)$. We have that $w_0+\Phi$ is a distributional solution of \eqref{18351} in $B_r(p_i)$ such that { $f_0- \gamma_2 \hbox{div}[\hbox{Ric}(\cdot,\nabla w_0)]-(2\gamma_3-\frac{\gamma_2}{3}) \hbox{div}(R \nabla w_0) \in L^\infty(B_r(p_i))$.}
Moreover, thanks to Remark \ref{rem12} there exists a fundamental solution $w_s$ corresponding to $\mu_s$ and $\Phi$, namely a SOLA $w_s \in W^{1,2,2)}(B_r(p_i))$ of $\mathcal{N}(w_s)= \beta_i \delta_{p_i}-U$ in $B_r(p_i)$, $w_s=\Phi$ and $\partial_\nu w_s=\partial_\nu \Phi$ on $\partial B_r(p_i)$.
\end{rem}
\noindent The aim now is to show that any fundamental solution $w_s$ has a logarithmic behaviour near $p_1,\dots,p_l$. For problems involving the $p-$Laplace operator an extensive study on isolated singularities is available, see \cite{VeKi86,Serrin-1,Serrin-2} 
(see also \cite{LiNg} for some fully nonlinear equations in conformal geometry). We adapt the argument in \cite{UVMRL} to our  situation and in presence of singularities to show the following result.
\begin{thm}\label{p:ex-singg}
{ Let $\frac{\gamma_2}{\gamma_3} \geq 6$. Any fundamental solution $w_s$ corresponding to $\mu_s$ satisfies $w_s \in C^\infty(M \setminus \{p_1, \dots, p_l\})$ and \eqref{12292} with $\a_i$ given by \eqref{eq:alpha-beta}.}
\end{thm}
\begin{pf}
\noindent Recall that $w_s$ is a SOLA of $\mathcal{N}(w_s)=\mu_s-U:=\hbox{div}\, F$ and $w_0$ is a distributional solution of $\mathcal{N}(w_0)=\mu_s+f_0:=\hbox{div}\, F_0$. Since $F,F_0 \in L^{1,\frac{4}{3})}(M,TM)$ with $\hbox{div}\, (F-F_0)=-(f_0+U) \in L^q(M)$ for all $1\leq q <2$ in view of
\begin{equation} \label{1157}
f_0- \gamma_2 \hbox{div}[\hbox{Ric}(\cdot,\nabla w_0)] -(2\gamma_3-\frac{\gamma_2}{3}) \hbox{div}(R \nabla w_0) \in L^\infty(M)
\end{equation}
by Lemma \ref{l:log-coeff}, we can let $\epsilon \to 0^+$ in \eqref{new1537} and by Fatou's lemma end up with 
\begin{eqnarray} \label{crucestm} 
\int [|\nabla^2_{\hat g} p|^2+|\nabla p|^4]  dv \leq C(  \|F-F_0\|_{\frac{4}{3}}^{\frac{4}{3}}+\eta \| \nabla p\|_2^2 +1)<+\infty
\end{eqnarray}  
in view of \eqref{12391}, where $p=w_s-w_0$ and $\hat g=e^{w_s+w_0} g$. Setting $g_0=e^{2w_0} g$, by $2w_0=w_s+w_0-p$ we deduce that $\nabla^2_{g_0} p=\nabla^2_{\hat g} p+O(|\nabla p|^2)$ in view of \eqref{13261} and then \eqref{crucestm} re-writes as
\begin{eqnarray} \label{crucestmm} 
\int [|\nabla^2_{g_0} p|^2+|\nabla p|^4]  dv <+\infty.
\end{eqnarray}

\medskip \noindent Notice that $w_s$ and $w_0$ satisfy
\begin{equation}\label{new917}
\langle \mathcal{N}(w_s)-\mathcal{N}(w_0),\varphi \rangle=-\int (f_0+U)\varphi \, dv,\ \ \varphi \in C^\infty(M),
\end{equation}
and it is crucial to properly re-write the L.H.S. in terms of $g_0$ and not $\hat g$ as in \eqref{difference}. We can argue exactly as in Proposition \ref{crucProp} to get 
\begin{eqnarray} \label{newdifference}  
\langle \mathcal{N}(w_s)-\mathcal{N}(w_0),\varphi \rangle
&=& 3\gamma_3 \int (\Delta_{g_0} p +2|\nabla p|_{g_0}^2) \Delta_{g_0} \varphi  \, dv_{g_0}+6\gamma_3 \int \langle \nabla^2_{g_0} p,\nabla^2_{g_0} \varphi \rangle_{g_0} dv_{g_0}\\
&&+12 \gamma_3  \int (\Delta_{g_0}p+|\nabla p|_{g_0}^2) \langle \nabla p,\nabla \varphi\rangle_{g_0}  dv_{g_0}+(\frac{\gamma_2}{2}-3\gamma_3) \int \Delta p \Delta \varphi \, dv \nonumber\\
&&+(2\gamma_3-\frac{\gamma_2}{3}) \int [3 \hbox{Ric}(\n p,\nabla \varphi) dv- R \langle \nabla p,\nabla \varphi \rangle] dv \nonumber
\end{eqnarray}
for all $\varphi \in C^\infty(M)$. Setting $\Delta_0 p=\Delta p+2\langle \nabla w_0,\nabla p \rangle$, by  \eqref{12391} we can re-write \eqref{new917}-\eqref{newdifference} as
\begin{eqnarray} 
&& 3\gamma_3 \int [ \Delta_0 p+ 2|\nabla p|^2] \Delta_0 \varphi \, dv+6\gamma_3 \int \langle \nabla^2_{g_0} p,\nabla^2_{g_0}\varphi \rangle dv+
12 \gamma_3 \int [\Delta_0 p+|\nabla p|^2]\langle \nabla p,\nabla \varphi\rangle dv \label{234}  \\
&&+(\frac{\gamma_2}{2}-3\gamma_3) \int \Delta p \Delta \varphi \, dv +(2\gamma_3-\frac{\gamma_2}{3}) \int [ 3 \hbox{Ric}(\n p,\nabla \varphi) dv- R \langle \nabla p,\nabla \varphi \rangle] dv =-\int (f_0+U) \varphi \, dv   \nonumber
\end{eqnarray}
for all $\varphi \in C^\infty(M)$. 

\medskip \noindent Given $p=p_i$, $i=1,\ldots,l$, set $\alpha=\alpha_i$, $A=\{ x \in M: \ d(x,p) \in [\frac{r}{4}, 8 r] \}$, $r>0$ small, and fix $2\leq q<4$. Through geodesic coordinates at $p$ and the change of variable $x= ry$, notice that  
\begin{eqnarray}
\int_A |\Delta_0 \varphi |^q dv&=&
\int_{B_{8r} \setminus B_{\frac{r}{4}}} |\Delta \varphi+\frac{2 \alpha}{|x|} \partial_{|x|} \varphi |^q \sqrt{|g|} dx =r^{4-2q} \int_{B_{8} \setminus B_{\frac{1}{4}} } |\Delta_{g^r} \varphi^r+\frac{2 \alpha}{|y|} \partial_{|y|} \varphi^r |^q \sqrt{|g^r|} dy \nonumber \\
&\leq & C r^{4-2q} \int_{B_{8} \setminus B_{\frac{1}{4}} } |\Delta_{g^r} \varphi^r|^q \sqrt{|g^r|} dy 
=C \int_A |\Delta \varphi|^q dv  
\label{13113}
\end{eqnarray}
for all $\varphi \in W^{2,q}_0(A)$, where $\varphi^r(y)=\varphi(\hbox{exp}_p(ry)) \in W^{2,q}_0(B_{8}\setminus B_{\frac{1}{4}})$ 
and $g^r(y)=g(\hbox{exp}_p(ry)) \to \delta_{\text{eucl}}$ $C^2-$uniformly in $B_8 \setminus B_{\frac{1}{4}}$ as $r \to 0^+$. { We have used that $$\int_{B_{8} \setminus B_{\frac{1}{4}} } |\nabla \varphi^r |^q \sqrt{|g^r|} dy \leq C  \int_{B_{8} \setminus B_{\frac{1}{4}} } |\Delta_{g^r} \varphi^r|^q \sqrt{|g^r|} dy$$ 
in view of Poincar\'e's inequality.} Arguing in the same way, one can also show that
\begin{eqnarray} \label{15515}
\int_A |\nabla^2_{g_0} \varphi|^q dv \leq C' \int_A |\nabla^2 \varphi|^q dv \leq  C \int_A |\Delta \varphi|^q dv  
\end{eqnarray}
for all $\varphi \in W^{2,q}_0(A)$, and
\begin{equation}
(\int_A |\psi|^{\frac{4q}{4-q}} dv)^{\frac{4-q}{4q}} \leq C (\int_A |\nabla \psi|^{q}dv)^{\frac{1}{q}},\quad  (\int_A |\nabla \varphi|^{\frac{4q}{4-q}}dv)^{\frac{4-q}{4q}} \leq C (\int_A |\Delta \varphi|^q dv)^{\frac{1}{q}}
\label{13114} \end{equation}for all $\psi \in W^{1,q}(A)$ such that either $\psi |_{\partial A}=0$ or $\ov{\psi}^A=0$ and for all $\varphi \in W^{2,q}_0(A)$.

\medskip \noindent Given $\tilde \chi \in C_0^\infty (\frac{1}{4},8)$ so that $0\leq \tilde \chi \leq 1$ and $\tilde \chi=1$ on $[\frac{1}{2},4]$, set $\chi(x)= \tilde \chi(\frac{d (x,p)}{r})$ and let
\begin{equation}\label{epsilonr}
\epsilon_r^2=\int_A (\Delta_0 p)^2 dv+\int_A  |\nabla^2_{g_0} p|^2 dv
+(\int_A |\nabla p|^4 dv)^{\frac{1}{2}}.
\end{equation}
We can assume that $0<\epsilon_r\leq 1$ for $r>0$ small since
$$\lim_{r \to 0}\epsilon_r=0$$
in view of \eqref{crucestmm}. By \eqref{13113}, \eqref{13114} and  H\"older's estimate we have that
\begin{eqnarray*}
&&|\int [ \Delta_0 p+ 2|\nabla p|^2] [\varphi \Delta_0  \chi +2\langle \nabla \chi,\nabla \varphi \rangle] dv|+ |\int  [2 \langle \nabla \chi,\nabla p \rangle(1+p-\ov{p}^A)+\Delta_0 \chi (p-\ov{p}^A)]  \Delta_0 \varphi \, dv|  \\
&& \leq  \frac{C' \epsilon_r}{r^{\frac{2(q-2)}{q}}}[(\int_A |\varphi|^{\frac{2q}{q-2}}dv )^{\frac{q-2}{2q}}+(\int_A |\nabla \varphi|^{\frac{4q}{3q-4}}dv )^{\frac{3q-4}{4q}}+ (\int_A |\Delta \varphi|^{\frac{q}{q-1}}dv )^{\frac{q-1}{q}}] \\
&& \leq  \frac{C \epsilon_r}{r^{\frac{2(q-2)}{q}}} (\int_A |\Delta \varphi|^{\frac{q}{q-1}}dv )^{\frac{q-1}{q}} 
\end{eqnarray*}
for all $\varphi \in W^{2,\frac{q}{q-1}}_0(A)$, taking into account that 
\begin{eqnarray*}
&& |\int  \langle \nabla \chi,\nabla p \rangle(p-\ov{p}^A) \Delta_0 \varphi \, dv| +|\int \Delta_0 \chi (p-\ov{p}^A)  \Delta_0 \varphi \, dv|  \\
&& \leq \frac{C''}{r^2} \left[r \epsilon_r (\int_A |p-\ov{p}^A|^{\frac{4q}{4-q}} dv)^{\frac{4-q}{4q}}+ (\int_A |p-\ov{p}^A|^q dv)^{\frac{1}{q}}  \right] (\int_A |\Delta_0 \varphi|^{\frac{q}{q-1}}dv )^{\frac{q-1}{q}} \\ 
&&\leq
\frac{C'}{r^2} \left[r \epsilon_r  (\int_A |\nabla p|^q dv)^{\frac{1}{q}} +(\int_A |\nabla p|^{\frac{4q}{q+4}} dv)^{\frac{q+4}{4q}} \right](\int_A |\Delta \varphi|^{\frac{q}{q-1}}dv )^{\frac{q-1}{q}} \leq  \frac{C \epsilon_r}{r^{\frac{2(q-2)}{q}}} (\int_A |\Delta \varphi|^{\frac{q}{q-1}}dv )^{\frac{q-1}{q}}.
\end{eqnarray*}
Since 
\begin{eqnarray*}
( \Delta_0 p+2 |\nabla p|^2 )
\Delta_0 (\chi \varphi) 
&=& (\Delta_0 h +2 \langle \nabla h , \nabla p \rangle ) \Delta_0 \varphi +
(\Delta_0 p+2|\nabla p|^2) (\varphi \Delta_0  \chi +2\langle \nabla \chi,\nabla \varphi \rangle) \\
&&-[2 \langle \nabla \chi,\nabla p \rangle (1+p-\ov{p}^A)+\Delta_0 \chi (p-\ov{p}^A)]  \Delta_0
 \varphi, 
\end{eqnarray*}
where $h=\chi (p-\ov{p}^A)$, we have that for some $\mathcal{L}_1 \in W^{-2,q}(A)$:
\begin{equation} \label{new904}
\int ( \Delta_0 p+2 |\nabla p|^2 ) \Delta_0 (\chi \varphi) dv= \int_A (\Delta_0 h +2 \langle \nabla h , \nabla p \rangle ) \Delta_0 \varphi dv+\mathcal{L}_1, \quad \|\mathcal{L}_1 \|\leq \frac{C \epsilon_r}{r^{\frac{2(q-2)}{q}}}.
\end{equation}
Analogously, there holds
\begin{eqnarray} \label{new913}
&& 6\gamma_3 \int \langle \nabla^2_{g_0} p, \nabla^2_{g_0} (\chi \varphi) \rangle dv+(\frac{\gamma_2}{2}-3\gamma_3) \int  \Delta p \Delta (\chi \varphi) dv=6\gamma_3  \int_A \langle \nabla^2_{g_0} h, \nabla^2_{g_0} \varphi \rangle dv\\
&&+(\frac{\gamma_2}{2}-3\gamma_3) \int_A \Delta h \Delta \varphi dv+\mathcal{L}_2, \quad \|\mathcal{L}_2 \|\leq \frac{C \epsilon_r}{r^{\frac{2(q-2)}{q}}}, \nonumber
\end{eqnarray}
thanks to
\begin{eqnarray*}
&& |\int \langle O(|\nabla p||\nabla \chi|)+\nabla^2_{g_0} \chi (p-\ov{p}^A),  \nabla^2_{g_0} \varphi \rangle dv|+ |\int \langle \nabla^2_{g_0} p, \varphi \nabla^2_{g_0}  \chi +O(|\nabla \chi||\nabla \varphi|) \rangle dv| \\
&& \leq   \frac{C' \epsilon_r}{r^{\frac{2(q-2)}{q}}}[(\int_A |\varphi|^{\frac{2q}{q-2}}dv )^{\frac{q-2}{2q}}+(\int_A |\nabla \varphi|^{\frac{4q}{3q-4}}dv )^{\frac{3q-4}{4q}}+ (\int_A |\nabla^2_{g_0} \varphi|^{\frac{q}{q-1}}dv )^{\frac{q-1}{q}}] \\
&&\leq  \frac{C \epsilon_r}{r^{\frac{2(q-2)}{q}}} (\int_A |\Delta \varphi|^{\frac{q}{q-1}}dv )^{\frac{q-1}{q}}
\end{eqnarray*}
and
\begin{eqnarray*}
\langle \nabla^2_{g_0} p, \nabla^2_{g_0} (\chi \varphi)\rangle &=&  \langle \nabla^2_{g_0} h  - d \chi \otimes d p -d p \otimes  d \chi- \nabla^2_{g_0} \chi (p-\ov{p}^A),  \nabla^2_{g_0} \varphi \rangle \\
&& + \langle \nabla^2_{g_0} p, \varphi \nabla^2_{g_0}  \chi +d \chi \otimes d \varphi+
d \varphi \otimes d \chi \rangle, 
\end{eqnarray*}
in view of \eqref{13261} and \eqref{15515}-\eqref{13114}. Since in a similar way 
\begin{eqnarray*}
|\int  |\nabla \chi| \Big( |\Delta_0 p|+ |\nabla p|^2+1 \Big) \Big( |\nabla \varphi||p-\ov{p}^A|+ |\nabla p||\varphi| \Big) dv |  \leq  \frac{C \epsilon_r}{r^{\frac{2(q-2)}{q}}} (\int_A |\Delta \varphi|^{\frac{q}{q-1}}dv )^{\frac{q-1}{q}}
\end{eqnarray*}
for all $\varphi \in W^{2,\frac{q}{q-1}}_0(A)$, there holds
\begin{eqnarray}\label{new930}
&& 12 \gamma_3 \int [\Delta_0 p+|\nabla p|^2]\langle \nabla p,\nabla (\chi \varphi) \rangle dv +(2 \gamma_3-\frac{\gamma_2}{3}) \int [3\hbox{Ric}(\n p,\nabla (\chi\varphi)) dv -R \langle \nabla p,\nabla (\chi \varphi) \rangle] dv \\
&&=12 \gamma_3 \int_A [\Delta_0 p+|\nabla p|^2]\langle \nabla h,\nabla \varphi\rangle dv  +(2\gamma_3-\frac{\gamma_2}{3}) \int_A [3\hbox{Ric}(\n h,\nabla \varphi) dv- R \langle \nabla h,\nabla \varphi \rangle] dv \nonumber \\
&&+\mathcal{L}_3, \quad \|\mathcal{L}_3 \|\leq \frac{C \epsilon_r}{r^{\frac{2(q-2)}{q}}}. \nonumber
\end{eqnarray}
Since by density and \eqref{crucestmm}  we can use $\chi \varphi$, $\varphi \in W^{2,2}_0(A)$, into \eqref{234}, by collecting \eqref{new904}-\eqref{new930} one has that 
\begin{eqnarray} \label{new1110}
&& 3\gamma_3  \int_A [ \Delta_0 h+ 2 \langle \nabla h, \nabla p \rangle] \Delta_0 \varphi \, dv+6 \gamma_3 \int_A \langle \nabla^2_{g_0} h,\nabla^2_{g_0}\varphi \rangle dv
+
12 \gamma_3 \int_A [\Delta_0 p+|\nabla p|^2]\langle \nabla h,\nabla \varphi\rangle dv\\
&&+(\frac{\gamma_2}{2}-3\gamma_3) \int_A \Delta h \Delta \varphi dv
+(2\gamma_3-\frac{\gamma_2}{3}) \int_A [3\hbox{Ric}(\n h,\nabla \varphi) dv- R \langle \nabla h,\nabla \varphi \rangle] dv
=\mathcal{L}(\varphi) \nonumber
\end{eqnarray}
for some $\mathcal{L} \in W^{-2,2}(A)$, which can also be regarded as $\mathcal{L} \in W^{-2,q}(A)$ satisfying 
\begin{equation} \label{449}
\|\mathcal{L}\|\leq \frac{C \epsilon_r}{r^{\frac{2(q-2)}{q}}}+(\int_A |f_0+U|^{\frac{2q}{q+2}}dv)^{\frac{q+2}{2q}}, 
\end{equation} 
in view of
$$|\int (f_0+U) \chi \varphi \, dv|\leq (\int_A |f_0+U|^{\frac{2q}{q+2}}dv)^{\frac{q+2}{2q}} (\int_A |\varphi|^{\frac{2q}{q-2}}dv)^{\frac{q-2}{2q}}.$$ 
Since
\begin{eqnarray*}
&& |\int_A  \langle \nabla \tilde h, \nabla p \rangle  \Delta_0 \varphi \, dv|+|\int_A [\Delta_0 p+|\nabla p|^2]\langle \nabla \tilde h,\nabla \varphi\rangle  dv | \\
&&\leq \epsilon_r (\int_A |\nabla \tilde h|^{\frac{4q}{4-q}} dv)^{\frac{4-q}{4q}} (\int_A |\Delta \varphi|^{\frac{q}{q-1}} dv)^{\frac{q-1}{q}}+C \epsilon_r (\int_A |\nabla \tilde h|^{\frac{4q}{4-q}} dv)^{\frac{4-q}{4q}} (\int_A |\nabla \varphi|^{\frac{4q}{3q-4}} dv)^{\frac{3q-4}{4q}}
\end{eqnarray*}
and
\begin{eqnarray*}
|\int_A [3\hbox{Ric}(\n \tilde h,\nabla \varphi) dv- R \langle \nabla \tilde h,\nabla \varphi \rangle] dv| \leq C r^2 (\int_A |\nabla \tilde h|^{\frac{4q}{4-q}} dv)^{\frac{4-q}{4q}} (\int_A |\nabla \varphi|^{\frac{4q}{3q-4}} dv)^{\frac{3q-4}{4q}}
\end{eqnarray*}
for all $\varphi \in W^{2,\frac{q}{q-1}}_0(A)$, equation \eqref{new1110} written in $\tilde h$ is equivalent to
$$3\gamma_3  \int_A \Delta_0 \tilde h \Delta_0 \varphi \, dv +6\gamma_3 \int_A \langle \nabla^2_{g_0} \tilde h,\nabla^2_{g_0}\varphi \rangle dv
+(\frac{\gamma_2}{2}-3\gamma_3) \int_A \Delta \tilde h \Delta \varphi dv+T[\tilde h](\varphi)=\mathcal{L}(\varphi),$$
where $T: W^{2,q}_0(A) \to W^{-2,q}(A)$ is a linear operator which satisfies $\|T\|\leq C (\epsilon_r+r^2)$. The crucial point is that the linear operator $\Delta_0^2 : W^{2,q}_0(A) \to W^{-2,q}(A)$ is an isomorphism with uniformly bounded inverse, where 
$$\Delta_0^2 \tilde h(\varphi)= \int_A  \Delta_0 \tilde h \Delta_0 \varphi \, 
 dv
+2 \int_A \langle \nabla_{g_0}^2 \tilde h, \nabla_{g_0}^2 \varphi \rangle dv+(\frac{\gamma_2}{6\gamma_3}-1) \int_A \Delta \tilde h \Delta \varphi dv.$$
Since $\epsilon_r +r^2 \to 0$ as $r \to 0$ we have that $3\gamma_3 \Delta^2_0+T:W^{2,q}_0(A) \to  W^{-2,q}(A)$ is still an isomorphism with uniformly bounded inverse. Then $3 \gamma_3 \Delta^2_0 \tilde h+T[\tilde  h]=\mathcal{L}$ is uniquely solvable in $W^{2,q}_0(A)$ for all $2\leq q<4$ and such a solution $\tilde h$ coincides with $h \in W^{2,2}_0(A)$ by uniqueness in $W^{2,2}_0(A)$. So for all $2\leq q<4$ we have shown that 
\begin{equation}
\|h\|_{W^{2,q}_0(A)} \le C' \|\mathcal{L}\|_{W^{-2,q}(A)} \leq C\left[\frac{\epsilon_r}{r^{\frac{2(q-2)}{q}}}+(\int_A |f_0+U|^{\frac{2q}{q+2}}dv)^{\frac{q+2}{2q}}\right] \label{18148}
\end{equation} 
for some $C>0$ thanks to \eqref{449}.

\medskip \noindent In order to show that $\Delta^2_0 : W^{2,q}_0(A) \to W^{-2,q}(A)$ is an isomorphism with uniformly bounded inverse, notice first that 
{ \begin{equation} \label{633}
\delta_A:=\inf \left\{ \int_A (\Delta_0 h)^2 dv: \ h \in W^{2,2}_0(M), \ \int_A (\Delta h)^2 dv=1 \right\}>0.
\end{equation}
Indeed, letting $h_n$ be a minimizing sequence in \eqref{633}, we can assume that $h_n \rightharpoonup h$ in $W^{2,2}_0(A)$ and $h_n \to h $ in $W^{1,2}_0(A)$ as $n \to +\infty$ thanks to  Sobolev's embedding Theorem. When $h=0$ we have that $\int_A (\Delta_0 h_n)^2 dv \to 1$ as $n \to +\infty$ and then $\delta_A=1$. If $h \not= 0$, we need to show that $\Delta_0 h \not= 0$ since by weak lower semi-continuity $\delta_A \geq \int_A (\Delta_0 h)^2 \ dv$. Observe that $\Delta_0 h=\Delta h+2 \langle \nabla w_0,\nabla h \rangle=0$ has only the trivial solution in $W^{2,2}_0(A)$ as it follows by testing $\Delta_0 h$ against $e^{2w_0} h$ and integrating by parts:
$$0=\int_A (\Delta h+2 \langle \nabla w_0,\nabla h \rangle) e^{2w_0} h dv
=- \int_A e^{2w_0} |\nabla h|^2 dv.$$}
Since every $\mathcal{L} \in W^{-2,q}(A)$ can be viewed as an element in $W^{-2,2}(A)$ in view of $\frac{q}{q-1}\leq 2$ and by \eqref{633} there holds
$$\Delta^2_0 h(h)= \int_A  (\Delta_0 h)^2 dv+2 \int_A |\nabla_{g_0}^2 h|^2 dv +(\frac{\gamma_2}{6\gamma_3}-1) (\int_A \Delta h)^2 dv\geq \delta_A \int_A ( \Delta h)^2 dv$$
due to $\frac{\gamma_2}{\gamma_3}\geq 6$, we can minimize $\frac{1}{2}\Delta^2_0 h (h)-\mathcal{L}(h)$ in $W^{2,2}_0(A)$ and find a solution $h \in W^{2,2}_0(A)$ of $\Delta^2_0 h=\mathcal{L}$ in $W^{-2,2}(A)$. Thanks to \eqref{12391}-\eqref{13261} and \eqref{14122} let us now rewrite $\Delta^2_0 h(\varphi)$ as
$$\Delta^2_0 h(\varphi)=(2+\frac{\gamma_2}{6\gamma_3}) \int_A \Delta h \Delta \varphi \, dv +\tilde{ \mathcal{L}}(\varphi),$$
where $\tilde{\mathcal{L}}$ satisfies $|\tilde{\mathcal{L}}(\varphi)|\leq { \frac{C}{r} \|h\|_{W^{2,2}_0(A)} \|\varphi\|_{W^{2,\frac{4}{3}}_0(A)}}$. Since $\mathcal{L} \in W^{-2,q}(A)$ and $\tilde{ \mathcal{L}} \in W^{-2,4}(A)$, we can use elliptic estimates for the bi-Laplacian operator in \cite{Agm} to show that $h \in W^{2,q}_0(A)$. Moreover, by the inverse mapping theorem we know that $\|\Delta^2_0 h \|_{W^{-2,q}(A)}\geq \delta \|h\|_{W^{2,q}_0(A)}$ for some $\delta=\delta(r)>0$. To see that $\delta>0$ can be chosen independent of $r>0$, through geodesic coordinates at $p$ and the change of variable $x= ry$ as in \eqref{13113}
we simply observe that 
$$\|\Delta^2_0 h \|_{W^{-2,q}(A)}= r^{\frac{4-2q}{q}} \sup\{ \Delta^{2,r}_0 h^r(\psi): \ \psi \in W_0^{2,\frac{q}{q-1}}(B_8 \setminus B_{\frac{1}{4}}), \ \int_{B_8 \setminus B_{\frac{1}{4}}} |\Delta_{g^r} \psi|^{\frac{q}{q-1}} dv_{g^r} \leq 1 \}$$
and $\|h\|_{W^{2,q}_0(A)}=r^{\frac{4-2q}{q}} (\int_{B_8 \setminus B_{\frac{1}{4}}} |\Delta_{g^r} h^r|^q dv_{g^r})^{\frac{1}{q}}$, where $\nabla w_0^r(y)=\frac{\alpha y}{|y|^2}$ and 
\begin{eqnarray*}
\Delta_0^{2,r} h^r(\psi)&=&
\int_{B_8 \setminus B_{\frac{1}{4}}}( \Delta_{g^r} h^r+ 2 \langle \nabla w_0^r, \nabla h^r \rangle_{g^r}) (\Delta_{g^r} \psi+ 2 \langle \nabla w_0^r, \nabla \psi \rangle_{g^r}) dv_{g^r} \\
&&+2 \int_{B_8 \setminus B_{\frac{1}{4}}} \langle \nabla_{g_0^r}^2 h^r, \nabla_{g_0^r}^2 \psi \rangle_{g^r} dv_{g^r}+ (\frac{\gamma_2}{6\gamma_3}-1) \int_{B_8 \setminus B_{\frac{1}{4}}} \Delta_{g^r} h^r \Delta_{g^r} \psi\, dv_{g^r}.
\end{eqnarray*}
Since $g_r(y)=g(ry) \to \delta_{\hbox{eucl}}$ $C^2-$uniformly in $B_8 \setminus B_{\frac{1}{4}}$ as $r \to 0^+$, we have that
\begin{equation} \label{703} 
\sup\{ \Delta_0^{2,r} \tilde h (\psi): \ \psi \in W_0^{2,\frac{q}{q-1}}(B_8 \setminus B_{\frac{1}{4}}), \ \int_{B_8 \setminus B_{\frac{1}{4}}} |\Delta_{g^r} \psi|^{\frac{q}{q-1}} dv_{g^r} \leq 1 \} \geq \delta (\int_{B_8 \setminus B_{\frac{1}{4}}} |\Delta_{g^r} \tilde h|^q dv_{g^r})^{\frac{1}{q}}
\end{equation}
uniformly in $\tilde h$ for some $\delta>0$, and then $\|\Delta^2 _0 h \|_{W^{-2,q}(A)}\geq \delta \|h\|_{W^{2,q}_0(A)}$. { We have used that the desired inequality $\|\Delta^2_{0,eucl} \tilde h \|_{W^{-2,q}(B_8 \setminus B_{\frac{1}{4}})}\geq \delta \| \tilde h\|_{W^{2,q}_0(B_8 \setminus B_{\frac{1}{4}})}$ does hold in the euclidean case with some $\delta>0$ and the following convergences:
$$\hbox{L.H.S. in }\eqref{703} \to \sup\{ \Delta_{0,eucl}^{2} \tilde h(\psi): \ \psi \in W_0^{2,\frac{q}{q-1}}(B_8 \setminus B_{\frac{1}{4}}), \ \int_{B_8 \setminus B_{\frac{1}{4}}} |\Delta \psi|^{\frac{q}{q-1}} dx \leq 1 \}
=\|\Delta^2_{0,eucl} \tilde h \|_{W^{-2,q}(B_8 \setminus B_{\frac{1}{4}} )}$$
and 
$$\hbox{R.H.S. in }\eqref{703} \to (\int_{B_8 \setminus B_{\frac{1}{4}}} |\Delta \tilde h|^q dx)^{\frac{1}{q}}$$
as $r \to 0^+$ uniformly in $\tilde h$.}

\medskip \noindent  Set $\tilde A=\{ x \in M: \ d(x,p) \in [\frac{r}{2}, 4 r] \}$.  Notice that by \eqref{13261} and \eqref{13114} it follows that 
\begin{equation}
(\int_{\tilde A} |\Delta_0 p|^q dv)^{\frac{1}{q}} +(\int_{\tilde A} |\nabla_{g_0}^2 p|^q dv)^{\frac{1}{q}}+(\int_{\tilde A}  |\nabla p|^{\frac{4q}{4-q}} dv)^{\frac{4-q}{4 q}}+r^{\frac{2(2-q)}{q}} \|p-\ov{p}^{\tilde A} \|_{\infty,\tilde A}
\leq C \|h\|_{W^{2,q}_0(A)} \label{09439}
\end{equation} 
for some $C>0$, in view of  $(\int_A |\nabla w_0|^q |\nabla h|^q dv)^{\frac{1}{q}} \leq C (\int_A |\nabla h|^{\frac{4q}{4-q}} dv)^{\frac{4-q}{4q}}$ and through geodesic coordinates
\begin{equation} \label{1025}
\|\psi \|_{\infty,\tilde A}=\| \psi^r\|_{\infty,B_{4} \setminus B_{\frac{1}{2}}} \leq C (\int_{B_{4} \setminus B_{\frac{1}{2}} } |\Delta_{g^r} \psi^r|^q \sqrt{|g^r|} dy)^{\frac{1}{q}} =C r^{\frac{2(q-2)}{q}}  (\int_{\tilde A} |\Delta \psi|^q dv)^{\frac{1}{q}}
\end{equation}
for all $\psi \in W^{1,q}(\tilde A)$ with $\ov{\psi}^{\tilde A}$ and for $q>2$. To get stronger estimates, let $\tilde \chi \in C_0^\infty (\frac{1}{2},4)$ with $0\leq \tilde \chi \leq 1$ and $\tilde \chi=1$ on $[1,2]$, and define now $\chi(x)= \tilde \chi(\frac{d (x,p)}{r})$ and $h=\chi (p-\ov{p}^{\tilde A})$. Thanks to \eqref{18148} and \eqref{09439} we can repeat the above argument  and, integrating by parts all the terms involving second-order derivatives of $\varphi$, get that:
\begin{eqnarray*}
&&|\int [ \Delta_0 p+ 2|\nabla p|^2] [\varphi \Delta_0  \chi +2\langle \nabla \chi,\nabla \varphi \rangle] dv|+ |\int  [2 \langle \nabla \chi,\nabla p \rangle(1+p-\ov{p}^{\tilde A})+\Delta_0 \chi (p-\ov{p}^{\tilde A})]  \Delta_0 \varphi \, dv|
\\
&&+ |\int \langle d \chi \otimes d p +d p \otimes  d \chi+\nabla^2_{g_0} \chi (p-\ov{p}^{\tilde A}),  \nabla^2_{g_0} \varphi \rangle dv| + |\int \langle \nabla^2_{g_0} p, \varphi \nabla^2_{g_0}  \chi +d \chi \otimes d \varphi+
d \varphi \otimes d \chi \rangle dv|
 \\
&&+|\int \Delta p  [\varphi \Delta  \chi +2\langle \nabla \chi,\nabla \varphi \rangle] dv|+ |\int  [2 \langle \nabla \chi,\nabla p \rangle +\Delta \chi (p-\ov{p}^{\tilde A})]  \Delta \varphi \, dv|
\\
&&+\int |\nabla \chi| ( |\Delta_0 p|+ |\nabla p|^2+1)( |\nabla \varphi| |p-\ov{p}^{\tilde A}| + |\nabla p| |\varphi |) dv| \\
&& \leq  \frac{C}{r}[\frac{\epsilon_r}{r^{\frac{2(q-2)}{q}}}+(\int_A |f_0+U|^{\frac{2q}{q+2}}dv)^{\frac{q+2}{2q}}](\int_{\tilde A} |\nabla \varphi|^{\frac{q}{q-1}}dv )^{\frac{q-1}{q}}, 
\end{eqnarray*}
and
\begin{eqnarray*}
 && |\int_{\tilde A}  \langle \nabla \tilde h, \nabla p \rangle \Delta_0 \varphi \, dv|+|\int_{\tilde A} [\Delta_0 p+|\nabla p|^2]\langle \nabla \tilde h,\nabla \varphi\rangle  dv| +|\int_{\tilde A} [3\hbox{Ric}(\n \tilde h,\nabla \varphi) dv- R \langle \nabla \tilde h,\nabla \varphi \rangle] dv| \\
&& \leq C (\tilde \epsilon_r +r^2) \|\tilde h\|_{W^{3,q}_0(\tilde A)}(\int_{\tilde A} |\nabla \varphi|^{\frac{q}{q-1}} dv)^{\frac{q-1}{q}} 
\end{eqnarray*}
for all $\varphi \in W^{2,\frac{q}{q-1}}_0(\tilde A)$, where $\tilde \epsilon_r$ is given by \eqref{epsilonr} on $\tilde A$. Notice that quadratic or cubic terms in $p$  have been estimated in the above expression by using \eqref{09439} on $p$ and \eqref{epsilonr} for the remaining powers of $p$. Hence, equation \eqref{new1110} in $\tilde h$ is equivalent to
$$3\gamma_3 \Delta^2_0 \tilde h(\varphi)+T[\tilde h](\varphi)=\mathcal{L}(\varphi),$$
where $T: W^{3,q}_0(\tilde A) \to W^{-1,q}(\tilde A)$ is a linear operator so that $\|T\|\leq C (\tilde \epsilon_r+r^2) $ and $\mathcal{L} \in W^{-1,q}(\tilde A)$ satisfies 
$$\|\mathcal{L}\|\leq \frac{C}{r}[\frac{\epsilon_r}{r^{\frac{2(q-2)}{q}}}+(\int_A |f_0+U|^{\frac{2q}{q+2}}dv)^{\frac{q+2}{2q}}]+(\int_{\tilde A} |f_0+U|^{\frac{4q}{q+4}}dv)^{\frac{q+4}{4q}}$$ 
in view of 
$$|\int_{\tilde A} (f_0+U) \chi \varphi \, dv|  \leq  (\int_{\tilde A} |f_0+U|^{\frac{4q}{q+4}}dv)^{\frac{q+4}{4q}}(\int_{\tilde A} |\nabla \varphi|^{\frac{q}{q-1}}dv )^{\frac{q-1}{q}}.$$
Arguing as before, since the operator $\Delta^2_0 : W^{3,q}_0(\tilde A) \to W^{-1,q}(\tilde A)$ is an isomorphism with uniformly bounded inverse, $3 \gamma_3 \Delta^2_0 \tilde h+T[\tilde  h]=\mathcal{L}$ is uniquely solvable in $W^{3,q}_0(A)$, $2< q<4$,  and such a solution $\tilde h$ coincides with $h \in W^{2,2}_0(\tilde A)$ by uniqueness in $W^{2,2}_0(\tilde A)$. Then, for all $2< q<4$ there holds 
\begin{equation}
\|h\|_{W^{3,q}_0(\tilde A)} \leq C \left[ \frac{\epsilon_r}{r^{\frac{3q-4}{q}}}+ \frac{1}{r} (\int_A |f_0+U|^{\frac{2q}{q+2}}dv)^{\frac{q+2}{2q}}+(\int_{\tilde A} |f_0+U|^{\frac{4q}{q+4}}dv)^{\frac{q+4}{4q}}\right] \label{12002}
\end{equation} 
for some $C>0$. Since arguing as in \eqref{1025} there holds
$$r \|\nabla h\|_{\infty,\tilde A}=\|\nabla h^r\|_{\infty,B_{4} \setminus B_{\frac{1}{2}}} \leq C (\int_{B_{4} \setminus B_{\frac{1}{2}} } |\Delta_{g^r} h^r|^{\frac{4q}{4-q}} \sqrt{|g^r|} dy)^{\frac{4-q}{4q}} =C r^{\frac{3q-4}{q}}  (\int_{\tilde A} |\Delta h|^{\frac{4q}{4-q}} dv)^{\frac{4-q}{4q}}
$$
in view of $\frac{4q}{4-q}>4$, by \eqref{13114} and \eqref{12002} for all $2<q<4$ we finally deduce that
\begin{equation} \label{18078}
r \|\nabla p\|_{\infty,B_{2r} \setminus B_r} \leq
C  \left[\epsilon_r + r^{\frac{2(q-2)}{q}}   (\int_{A}   |f_0+U|^{\frac{2q}{q+2}}dv)^{\frac{q+2}{2q}}+r^{\frac{3q-4}{q}}  (\int_{\tilde A}   |f_0+U|^{\frac{4q}{q+4}}dv)^{\frac{q+4}{4q}}\right]
\end{equation}
for some $C>0$. Estimate \eqref{18078} establishes the validity of \eqref{12292} when $k=1$ in view of \eqref{1157}. Iterating the argument one shows that \eqref{12292} does hold for $k=2,3$ too.

\medskip \noindent When $p \in M \setminus \{p_1,\dots,p_l\}$, there is no need to work on annuli as in the previous argument, and it is therefore possible to show that $w \in W^{3,q}_0(B_r(p))$, $2<q<4$. Then $w \in C^\infty(M \setminus \{p_1,\dots,p_l\})$ by an iteration. \end{pf}
\begin{rem}\label{rem4} According to the terminology in Remark \ref{rem2}, any fundamental solution corresponding to $\mu_s=\beta_i \delta_{p_i}$ and $\Phi \in C^\infty(\ov{B_r(p_i)})$ satisfies  the conclusions of Theorem \ref{p:ex-singg} in $\ov{B_r(p_i)}$.
\end{rem}

\section{Blow-up analysis} \label{blowup}
\setcounter{equation}{0}  
\noindent In this section we are concerned with the asymptotic analysis of sequences of 
solutions $w_n$ to \eqref{EL1}. The first issue is to determine a minimal volume quantization in the blow-up scenario, as it will follow by Adams' inequality and \eqref{00828}. The blow-up threshold is not optimal but it can be sharpened by using a Pohozaev identity along with the logarithmic behaviour of the singular limit for $w_n-\ov{w}_n$. However, it is not clear whether $\ov{w}_n$ tends to minus infinity or not, determining whether the limiting measure of $\mu_n e^{4w_n}$ is purely concentrated or  presents some residual $L^1-$part. The latter is usually excluded by comparison with the purely concentrated case.

 In our setting {\em maximum principles} are not available for the fourth-order operator $\mathcal{N}$ and a new approach has to be devised, based only on the scaling invariance of the PDE:  we apply  asymptotic analysis and  Pohozaev's identity to a slightly rescaled sequence $u_n$ for which the  limiting measure is purely concentrated, getting the optimal blow-up threshold; since the concentrated part is sufficiently strong, the fundamental solution in the purely-concentrated case has a low exponential integrability and,  by using  $W^{1,2,2)}$-bounds to make a comparison, the same remains true for $\displaystyle \lim_{n\to +\infty}(w_n-\ov{w}_n)$ when $\inf_n \ov{w}_n >-\infty$, in contrast to $\int e^{4w_n} dv=1$ (which is assumed in Theorem \ref{p:minimal-blow-up}). In order to have an asymptotic description of $u_n$, observe that scaling-invariant uniform estimates on $w_n$ are needed, which is precisely the content of  Theorem \ref{BMO}.

\medskip \noindent Let $g_n$ be a metric on $B_r$ with volume element $dv_{g_n}$, $U_n \in C^\infty(\ov{B_r})$ and $\mathcal{N}_n$ be the operator associated to $g_n$ through \eqref{linearL}. We consider a sequence of solutions $u_n$ to
\begin{equation} \label{v_n}
\mathcal{N}_n(u_n)+U_n=\mu_n e^{4u_n} \qquad \hbox{in }B_r.
\end{equation} 
We assume that $\mu_n \to \mu_0$,
\begin{equation} \label{154200}
\sup_n \int_{B_r} e^{4u_n}dv_{g_n}<+\infty,\qquad \sup_n \int_{B_r} (u_n-c_n)^4 dv_{g_n}<+\infty, 
\end{equation}
and
\begin{equation} \label{bis154200}
U_n \to U_\infty \hbox{ in }C^1(\ov{B_r}),\qquad g_n \to g_\infty \hbox{ in }C^4(\ov{B_r})
\end{equation}
for some $U_\infty \in C^\infty(\ov{B_r})$, a metric $g_\infty$ and $c_n \in \mathbb{R}$. Notice that \eqref{154200} implies 
\begin{equation} \label{154200bis}
\sup_n  \int_{B_r} (u_n-\ov{u}_n^r)^4 dv_{g_n}<+\infty
\end{equation}
in terms of the average $\ov{u}_n^r=\fint_{B_r} u_n dv_{g_n}$ of $u_n$ on $B_r$ w.r.t. $g_n$, since by H\"older's inequality  
$$|\ov{u}_n^r-c_n| \leq \fint_{B_r} |u_n-c_n| dv_{g_n}\leq \frac{C}{r} (\int_{B_r} (u_n-c_n)^4 dv_{g_n})^{\frac{1}{4}}.$$
We have the following local result on minimal volume quantization.
\begin{pro}\label{bound} Let $\frac{\gamma_2}{\gamma_3}>\frac{3}{2}$. There exists $\epsilon_0>0$ so that
\begin{equation} \label{H2est}
\sup_n \int_{B_{\frac{r}{2}}} [(\Delta_{g_n} u_n)^2+|\nabla u_n|_{g_n}^4] \ dv_{g_n} <+\infty
\end{equation}
provided $|\mu_n| \int_{B_r} e^{4 u_n}dv_{g_n} \leq  \epsilon_0$. Moreover, assuming $u_n -c_n \rightharpoonup u_0$ in $W^{2,2}_{g_\infty}(B_{\frac{r}{2}})$ and $\frac{\gamma_2}{\gamma_3}\geq 6$, there exists $0<r_0\leq \frac{r}{4}$ so that
\begin{equation} \label{C4est}
\sup_n \|u_n-c_n \|_{C^{4,\alpha}(B_{r_0})}<+\infty
\end{equation}
for any $\alpha \in (0,1)$.
\end{pro}
\begin{pf} By \eqref{154200bis}, it is enough to establish the proposition with $c_n=\ov{u}_n^r$. For simplicity we omit the dependence on $n$ and the dependence of geometric quantities on $g_n$. Let $\chi \in C_0^\infty (B_r)$ be so that $0\leq \chi \leq 1$, $\chi=1$ in $B_{\frac{r}{2}}$ and $|\Delta \chi|+|\nabla \chi| =O(1)$. In view of  Remark \ref{rem0}, re-write \eqref{00828} with $\psi(s)=s$:
\begin{eqnarray*}
&& \int_{B_r}  \chi^4 [\mu e^{4u}-U] (u-c) \ dv= \int_{B_r} \chi^4 [(\frac{\gamma_2}{2}+6\gamma_3) (\Delta u)^2 +18 \gamma_3 \Delta u |\nabla u|^2
+12 \gamma_3 |\nabla u|^4 ] \, dv\\
&&+O\left(\int_{B_r} [\chi^4+\chi^2 |u-c|+\chi^3 |\nabla u|(1+|u-c|)][1+|\Delta u|+|\nabla u|^2] \, dv\right).
\end{eqnarray*}
By Young's inequality and \eqref{154200} we have that
\begin{eqnarray*}
O(\int_{B_r} [\chi^4+\chi^2 |u-c|+\chi^3 |\nabla u|(1+|u-c|)][1+|\Delta u|+|\nabla u|^2] \, dv) \leq \epsilon \int_{B_r} \chi^4 [(\Delta u)^2+|\nabla u|^4] \, dv +C_\epsilon
\end{eqnarray*}
for all $\epsilon>0$, with some $C_\epsilon>0$. Setting $\beta=\frac{\gamma_2}{\gamma_3}$, arguing as in \eqref{n244} when $\psi(s)=s$ we have that
\begin{eqnarray} \label{103400}
&&  \int \chi^4 \left[ (\beta+12)(\Delta u)^2  +36 \Delta u |\nabla u|^2 +24 |\nabla u|^4 \right] dv\\
&& \geq (\beta+12-\frac{27}{2(1-\delta)}) \int \chi^4 (\Delta u)^2 dv +24 \delta \int \chi^4  |\nabla u|^4 dv \geq \frac{2 \delta_0}{|\gamma_3|}\int \chi^4  [(\Delta u)^2+|\nabla u|^4]dv
\nonumber 
\end{eqnarray}
for some $\delta_0>0$, thanks to $\beta >\frac{3}{2}$ and for a suitable choice of $\delta \in (0,1)$. Since $\Delta [\chi^2(u-c)] = \chi^2 \Delta u +O(|\nabla \chi^2||\nabla u|+|u-c|)$ and $\nabla[\chi(u-c)]=\chi \nabla u+O(|u-c|)$, by Young's inequality we obtain 
$$\int_{B_r} [\Delta (\chi^2(u-c))]^2+ |\nabla (\chi(u-c))|^4] \ dv \leq (1+\epsilon) \int_{B_r} \chi^4[ (\Delta u)^2+|\nabla u|^4] \ dv +C_\epsilon$$
for all $\epsilon>0$ with some $C_\epsilon>0$, thanks to \eqref{154200}. Re-collecting all the above estimates we  proved that
\begin{eqnarray} \label{104900}
\int_{B_r} [\Delta (\chi^2(u-c))]^2+ |\nabla (\chi(u-c))|^4] \ dv  \leq C_\epsilon+\frac{(1+\epsilon)|\mu| }{\delta_0-\epsilon}\int_{B_r} \chi^4 e^{4u}|u-c| \, dv
\end{eqnarray}
for all $0<\epsilon<\delta_0$ and some $C_\epsilon>0$. To estimate the R.H.S. we use the inequality 
$$\chi^4 |s| e^{s} \leq \frac{2}{\l} e^s + e^{\l \chi^4 s^2}$$
with $s = 4 (u-c)$ and $\l = \frac{\pi^2}{ \| \Delta (\chi^2 (u-c)) \|_{L^2(B_r)}^2}$, to get by Jensen's inequality that 
\begin{eqnarray*}
\int_{B_r} \chi^4 e^{4u}|u-c| dv  
\leq \frac{  \int_{B_r} e^{4 u}dv }{2 \pi^2}  \int_{B_r} [\Delta (\chi^2(u-c))]^2 \ dv+   \frac{\fint_{B_r} e^{4u} \, dv }{4 }   \int_{B_r} e^{\frac{16 \pi^2 \chi^4 (u-c)^2}{\| \Delta (\chi^2 (u-c))\|_{L^2(B_r)}^2}} \ dv.
\end{eqnarray*}
Setting $\epsilon_0=\pi^2 \delta_0$, we can find $\epsilon>0$ small so that $\frac{(1+\epsilon)|\mu|}{ 2\pi^2 (\delta_0-\epsilon)} \int_{B_r} e^{4 u}dv \leq \frac{3}{4}$  and then \eqref{104900} produces
$$\int_{B_r} [\Delta (\chi^2(u-c))]^2+ |\nabla (\chi(u-c))|^4] \ dv \leq C+C \fint_{B_r} e^{4u} dv  \, \int_{B_r} e^{\frac{16 \pi^2 \chi^4(u-c)^2}{\| \Delta (\chi^2 (u-c))\|_{L^2(B_r)}^2}} \ dv$$
for some $C>0$. Thanks to \eqref{bis154200} and $16\pi^2 < 32 \pi^2$ we can apply Adams' inequality in \cite{Adams,Fon} to $\chi^2(u-c)$ and finally get  the validity of \eqref{H2est}.

\medskip  We are now in the case $u -c \rightharpoonup u_0$ in $W^{2,2}_{g_\infty}(B_{\frac{r}{2}})$ and $\frac{\gamma_2}{\gamma_3}\geq 6 $. By contradiction, assume that for all $0<r_0 \leq \frac{r}{4}$ there holds, up to a subsequence,
$$\|u-c\|_{C^{4,\alpha}(B_{r_0})} \to +\infty$$
for some $\alpha \in (0,1)$ and $c \to c_0$, where $c_0 \in [-\infty,+\infty)$ thanks to Jensen's inequality and \eqref{154200}. By Adams' inequality it is straightforward to show that 
\begin{equation} \label{new1449}
\mu e^{4u} \to \mu_0 e^{4u_0+4c_0} \qquad \hbox{in }L^q_{g_\infty} (B_{\frac{r}{2}}), \, q \geq 1.
\end{equation}
Since the limiting function $u_0 \in W^{2,2}_{g_\infty}(B_{\frac{r}{2}})$ solves $\mathcal{N}_{g_\infty} (u_0)=\mu_0 e^{4u_0+4c_0}-U_{\infty}$ in $B_{\frac{r}{2}}$ in view of \eqref{new1449}, by the regularity result in \cite{UVMRL} we have that $u_0 \in C^\infty(B_{\frac{r}{2}})$ and then $\mathcal{N}(u_0) \to \mu_0 e^{4u_0+4c_0}-U_{\infty}$  holds locally uniformly in $B_{\frac{r}{2}}$ in view of \eqref{bis154200}.  We can make use of \eqref{difference} with $w_1=u-c$, $w_2=u_0$ and $\varphi \in C_0^\infty(B_{\frac{r}{2}})$ thanks to Remark \ref{rem1}. Setting $p=u-c-u_0$ and $q=u-c+u_0$, \eqref{difference} re-writes as
\begin{eqnarray} \label{diff1}
&& 3 \gamma_3 \int (\Delta p+\langle \nabla q,\nabla p \rangle)  (\Delta \varphi+\langle \nabla q,\nabla \varphi \rangle) dv
+6 \gamma_3 \int \langle \nabla^2_{\hat g} p,\nabla^2_{\hat g} \varphi \rangle dv + 3\gamma_3\int |\nabla p|^2 \langle \nabla p,\nabla \varphi \rangle  dv\\
&&+(\frac{\gamma_2}{2}-3\gamma_3) \int \Delta p \Delta \varphi dv+(2\gamma_3-\frac{\gamma_2}{3}) \int [3\hbox{Ric}(\nabla p,\nabla \varphi)-R \langle \nabla p,\nabla \varphi \rangle]dv \nonumber \\
&&=\int [\mu e^{4u}-U-\mathcal{N}(u_0)] \varphi \, dv \nonumber 
\end{eqnarray}
for all $\varphi \in C_0^\infty (B_{\frac{r}{2}})$ in view of \eqref{12391}, where $\hat g=e^q g$. Take $\varphi=\chi^4 p$ and $\chi \in C_0^\infty(B_{\frac{r}{2}})$ in \eqref{diff1} to get
\begin{eqnarray} \label{09232}
&& \int \chi^4 \left[ 3 \gamma_3  (\Delta p+\langle \nabla q,\nabla p \rangle)^2 +6\gamma_3 |\nabla^2_{\hat g} p|^2 +3\gamma_3 |\nabla p|^4 +(\frac{\gamma_2}{2}-3 \gamma_3) (\Delta p)^2\right]dv\\
&&=O(\int_{B_{\frac{r}{2}}}\chi^4 |\mu e^{4u}-U-\mathcal{N}(u_0)| |p| dv) \nonumber \\
&&+O( \int_{B_{\frac{r}{2}}} |p||\nabla p|^3 dv+ \int_{B_{\frac{r}{2}}} (|p|+ |\nabla p|+  |p| |\nabla q|)(|\nabla p|+|\nabla p||\nabla q|+|\nabla^2 p|)dv ). \nonumber
\end{eqnarray}
Since $p \rightharpoonup 0$ in $W^{2,2}_{g_\infty}(B_{\frac{r}{2}})$, by \eqref{bis154200} we have  that 
\begin{equation} \label{new903}
\int \left[ |\nabla p|^4+|\nabla q|^4 +|\nabla^2 p|^2 \right]dv=O(1),\quad \int \left[|p|^4 +|\nabla p|^{\frac{8}{3}}\right] dv \to 0.
\end{equation}
Inserting \eqref{new1449} and \eqref{new903} into \eqref{09232} we deduce that
$$\int \chi^4 (\Delta_{g_\infty} p)^2 dv_{g_\infty} \to 0,$$
and by taking $\chi=1$ on $B_{\frac{r}{4}}$ we end up with $u -c \to u_0$ in $W^{2,2}_{g_\infty}(B_{\frac{r}{4}})$. Since $u_0 \in C^{\infty}(B_{\frac{r}{2}})$, for all $\delta>0$ we can find $0<r_0\leq \frac{r}{4}$ so that 
$$\int_{B_{r_0}} [(\Delta u)^2+|\nabla u|^4] \, dv \leq \delta:$$
this is the crucial assumption in \cite{UVMRL} to derive upper bounds in strong norms on $u$ which do not depend on $g$. Then $u-c$ is uniformly bounded in $C^{4,\alpha}(B_{r_0})$ for any $\alpha \in (0,1)$, which is a contradiction, and the proof is thereby complete.
\end{pf}

\medskip \noindent Hereafter we assume $\frac{\gamma_2}{\gamma_3}\geq 6$. Let $w_n$ be as in Theorem \ref{p:minimal-blow-up} and let us restrict our attention to the case $\|w_n-\ov{w}_n \|_{C^{4,\alpha}(M)} \to +\infty$ as $n \to +\infty$ for some $\alpha \in (0,1)$. Thanks to Theorem \ref{BMO} we have that $[w_n]_{BMO} \leq C$, which implies the validity of \eqref{154200}-\eqref{bis154200} for $w_n$ with $c_n=\ov{w}_n$, $\tilde U_n$ and $g_n \equiv g$. Up to a subsequence, assume that $e^{4w_n} \rightharpoonup \hat \mu$ as $n \to +\infty$ in the weak sense of distributions on $M$, where $\hat \mu$ is a probability measure on $M$. Consider the finite set 
$$S=\{ p \in M:\  |\mu_0| \hat \mu (B_r(p)) \geq  \epsilon_0 \ \forall \ 0<r \leq i_0 \},$$
where $\epsilon_0>0$ is given by Proposition \ref{bound}. For any compact set $K \subset M \setminus S$, by \eqref{H2est} we deduce 
\begin{equation} \label{H2estbis}
\sup_n \int_K [(\Delta w_n)^2+|\nabla w_n|^4] \ dv <+\infty.
\end{equation}
By \eqref{new0855} and \eqref{H2estbis} we have that $w_n-\ov{w}_n$ is uniformly bounded in $W^{2,2}(K)$ and then, up to a subsequence and a diagonal process, $w_n-\ov{w}_n \rightharpoonup w_0$ weakly in $W^{2,2}_{loc}(M \setminus S)$. For any $p \in M \setminus S$ by \eqref{C4est} we can find $r(p)>0$ small so that $\|w_n- \ov{w}_n\|_{C^{4,\alpha}(B_{r(p)}) }\leq C(p)$. By compactness $w_n-\ov{w}_n$ is uniformly bounded in $C^{4,\alpha}_{loc}(M\setminus S)$ and then, up to a further subsequence, $w_n -\ov{w}_n \to w_0$ in $C^4_{loc}(M \setminus S)$. In particular $S \not= \emptyset$, $\mu_0 \not=0$ and
$\max_M w_n \to +\infty$ as $n \to +\infty$.

\medskip \noindent Since $e^{4 \ov{w}_n}\leq \frac{1}{\hbox{vol} M}$ by Jensen's inequality, up to a subsequence assume that $\ov{w}_n \to c \in [-\infty,+\infty)$ as $n \to + \infty$. Since $e^{4 w_n} \to e^{4w_0+4c}$ locally uniformly in $M \setminus S$, we have that
$$ e^{4w_n} \rightharpoonup e^{4w_0+4c} dv+\sum_{i=1}^l \tilde \beta_i \delta_{p_i} \qquad 
 \hbox{ as } n \to +\infty $$
weakly in the sense of measures, where $S=\{p_1,\dots,p_l \}$ and $\tilde \beta_i \geq \frac{\epsilon_0}{|\mu_0|}$. The function $w_0$ is a SOLA of
\begin{equation} \label{205000}
\mathcal{N}(w_0)=\mu_0 e^{4w_0+4c}+\sum_{i=1}^l  \beta_i \delta_{p_i}-U \qquad \hbox{in }M
\end{equation}
for $\beta_i=\mu_0 \tilde \beta_i$.  

\medskip \noindent  We aim to compute the values of the $\beta_i$'s, and we will prove below a quantization result in a suitable general form. In particular, it will apply to the following scaling of $w_n$, $\tilde U_n$ and $g$: 
\begin{equation} \label{new1055}
u_n(y)=w_n[\hbox{exp}_p (r_n y) ]+ \log r_n,\quad U_n(y)=r_n^4 \tilde U_n [  \hbox{exp}_p (r_n y)], \quad g_n(y)=g [  \hbox{exp}_p (r_n y)]
\end{equation}
for $|y| \leq \frac{i_0}{r_n}$, where $p \in M$ and $r_n \to 0^+$. The function $u_n$  is a solution of \eqref{v_n} for $|y| \leq \frac{i_0}{r_n}$ which satisfies
$$\int_{B_1(0)} |u_n -\ov{u}_n^1|^4 dv_{g_n}=\frac{1}{r_n^4} \int_{B_{r_n}(p)} |w_n- \ov{w}_n^{r_n}|^4 dv \leq C' \fint_{B_{r_n}(p)} |w_n- \ov{w}_n^{r_n}|^4 dv\leq C$$
in view of $[w_n]_{BMO} \leq C$. Therefore $u_n$ satisfies \eqref{154200}-\eqref{bis154200} on any $B_r \subset B_1(0)$ with $c_n=\ov{u}_n^1$, $U_n \to 0$ in $C^1(\ov{B_1(0)})$ and $g_n \to \delta_{eucl}$ in $C^4(\ov{B_1(0)})$. The result we have is the following.
\begin{lem}\label{l:quant}
Let $u_n$ be a solution of \eqref{v_n} which satisfies \eqref{154200}-\eqref{bis154200} in $B_1(0)$. Suppose that 
\begin{equation}\label{pureconc1}
\mu_n e^{4u_n} dv_{g_n} \rightharpoonup  \beta \, \delta_0
\end{equation}
weakly in the sense of measures in $B_1(0)$ as $n \to +\infty$, for some $\beta \not= 0$. Then $\beta =  8 \pi^2 \gamma_2$. 
\end{lem}
\begin{pf} Arguing as we did for $w_n$, we can apply Proposition \ref{bound} to $u_n$ to get that $u_n-\ov{u}_n^1$ is uniformly bounded in $W^{2,2}_{loc}(B_1\setminus \{0\} )$ in view of \eqref{154200bis}. Up to a subsequence and a diagonal process, we have that $u_n-\ov{u}_n^1 \rightharpoonup u_0$ weakly in $W^{2,2}_{loc}(B_1(0) \setminus \{0\})$ and in turn 
\begin{equation} \label{new1046} 
u_n-\ov{u}_n^1 \to u_0 \quad \hbox{in }C^4_{loc}(B_1(0) \setminus \{0\}), \quad \ov{u}_n^1 \to -\infty ,
\end{equation}
as $n \to +\infty$ in view of \eqref{pureconc1}. According to Remark \ref{rem12} $u_0$ is a SOLA of $\mathcal{N}_{g_\infty} \, u_0+U_\infty=\beta \, \delta_0$ in $B_{\frac{1}{2}}(0)$, $u_0=\Phi$ and $\partial_\nu u_0 =\partial_\nu \Phi$ on $\partial B_{\frac{1}{2}}(0)$, where $\Phi$ is a smooth extension in $B_{\frac{1}{2}}(0)$ of $u_0 \Big|_{\partial B_{\frac{1}{2}}(0)}$. We continue the proof dividing it into the following  steps. 

\medskip

\noindent {\bf Step 1.}  Up to a subsequence, there exist $p_n ^1,\dots,p_n^J $, $J \in \N$, such that $p_n ^1,\dots,p_n^J \to 0$ as $n \to +\infty$ and 
\begin{equation} \label{1}
d_n(y)^4 e^{4u_n} \leq C_1 \qquad \hbox{in }B_1(0)
\end{equation}
where $d_n(y)=\min \{ d_{g_n}(y,p_n^1),\dots, d_{g_n}(y,p_n^J) \}$.

To prove \eqref{1}, we first take $p_n^1 \to 0$ as the maximum point of $u_n$ in $B_1(0)$. Let $z_n^1$ be the scaling of $u_n$ around $p_n^1$ with scale $\mu_n^1=\hbox{exp}[-u_n(p_n^1)] \to 0$ in view of $u_n(p_n^1) \to +\infty$. Since $z_n^1 \leq z_n^1(0)=0$, by Proposition \ref{bound} we deduce that 
\begin{equation} \label{4}
z_n^1 \to z^1\qquad \hbox{in }C^4_{loc}(\mathbb{R}^4).
\end{equation}
Given $r_n^1 >> \mu_n^1$  we have that the scaling $\tilde z_n^1$ of $u_n$ around $p_n^1$ with scale $r_n^1$ still blows up and by Proposition \ref{bound} $|\mu_n| \int_{B_1(0)} e^{4 \tilde z_n^1} dv_{\tilde g_n}\geq \epsilon_0$, where $\tilde g_n=g_n(r_n^1y+p_n^1)$, or equivalently $|\mu_n| \int_{B_{r_n^1}(p_n^1)} e^{4 u_n} dv_{g_n} \geq \epsilon_0$.\\
We now proceed as follows.
If \eqref{1} were not valid with $d_n(y)=d_{g_n}(y,p_n^1)$, by \eqref{new1046} we would find a sequence $p_n^2\to 0$ of maximum points for $d_{g_n}(y,p_n^1)e^{u_n}$ in $B_1(0)$ so that 
\begin{equation} \label{5}
d_{g_n}(p_n^1,p_n^2) e^{u_n(p_n^2)} \to +\infty.
\end{equation} 
Let $z_n^2$ be the scaling of $u_n$ around $p_n^2$ with scale $\mu_n^2=\hbox{exp}[-u_n(p_n^2)] \to 0$ in view of \eqref{5}. Thanks to \eqref{4}-\eqref{5} we have that
$$\frac{d_{g_n}(p_n^1,p_n^2)}{\mu_n^1} \to +\infty,\qquad \qquad \frac{d_{g_n}(p_n^1,p_n^2)}{\mu_n^2} \to +\infty. $$
By the maximality property of $p_n^2$, $z_n^2$ is bounded from above and then by Proposition \ref{bound}
$$z_n^2 \to z^2 \qquad \hbox{in }C^4_{loc}(\mathbb{R}^4).$$
Arguing as above, for $r_n^2 >> \mu_n^2$ we have that $|\mu_n| \int_{B_{r_n^2}(p_n^2)} e^{4 u_n} dv_{g_n} \geq \epsilon_0$. Iterating as long as \eqref{1} is not valid, we can find points $p_n^1,\dots,p_n^J \to 0$ so that 
\begin{equation} \label{8} \frac{\mu_n^i+\mu_n^j}{d_{g_n}(p_n^i,p_n^j)} \to 0 \qquad \forall \ i \not= j \end{equation}
and  $|\mu_n| \int_{B_{r_n^i}(p_n^i)} e^{4 u_n}dv_{g_n} \geq \epsilon_0$ for all $i=1,\dots,J$, for a choice $r_n^i >> \mu_n^i$. Now we define  radii $r_n^i$ by
$r_n^i=\frac{1}{2} \min\{ d_{g_n}(p_n^i,p_n^j): \ j\not=i \}$, in such a way that $B_{r_n^i}(p_n^i) \cap B_{r_n^j}(p_n^j)$ for all $i \not= j$ and $r_n^i>>\mu_n^i$ thanks to \eqref{8}. Since
$$|\mu_n| \int_{B_1(0)} e^{4 u_n} dv_{g_n} \geq J \epsilon_0,$$
by $|\mu_n| \int_{B_1(0)} e^{4 u_n} dv_{g_n} \to |\beta|$ we have that such an iterative procedure must stop after $J$ times, and then \eqref{1} does hold with $p_n^1,\dots,p_n^J$.

\medskip \noindent {\bf Step 2.}  Assume that $d_{g_n}(y,p_n)^4 e^{4u_n} \leq C_1$ does hold in $B_1(0)$ for some $p_n \to 0$. Then $\beta =  8 \pi^2 \gamma_2$. 

To show this, first notice that by Proposition \ref{bound} and $d_{g_n}(y,p_n)^4 e^{4u_n} \leq C_1$ in $B_1(p_n)$ there exists $\tilde{C}_1>0$ such that for all $s \in (0,1/4)$ one has 
\begin{equation} \label{H2est-s}
 \int_{B_{2s}(p_n) \setminus B_s(p_n)} [(\Delta_{g_n} u_n)^2+|\nabla u_n|_{g_n}^4] \ dv_{g_n} \leq \tilde{C}_1
\end{equation}
for all $n$.  Since by \eqref{H2est-s} the remainder volume integrals in the Pohozaev identity \eqref{eq:poho-1} converge to zero as $r \to 0$ uniformly in $n$, we can apply Proposition \ref{p:poho} in $B_r(p_n)$ and letting $n \to +\infty$ get that
 $$-\beta=
 \mathcal{B}_{g_0}(0,B_r(0),u_0)+o_r(1),$$
 in view of \eqref{bis154200} and \eqref{pureconc1}-\eqref{new1046}. By Remark \ref{rem4} $u_0$ satisfies \eqref{12292} at $0$, and a straightforward computation 
for the boundary integrals in \eqref{eq:poho-12} leads as $r \to 0^+$ to the identity 
$$  - [9 \gamma_3 \alpha^4 + (\gamma_2 + 12 \gamma_3) \alpha^2 + 24 \gamma_3 \alpha^3] 2 \pi ^2 = 
  - \beta =  4 \pi^2 [(\gamma_2+12 \gamma_3) \alpha + 18 \gamma_3 \alpha^2 + 6\gamma_3 \alpha^3]$$
in view of \eqref{eq:alpha-beta}, which has a unique solution in $\mathbb{R}\setminus \{0\}$ given by $\alpha = -2$. Hence we have shown that $\beta =  8 \pi^2 \gamma_2$, as claimed.

\medskip \noindent Since \eqref{1} does not allow the direct use of Step 2 when $J\geq 2$, the idea is to properly group the points $p_n^1,\dots,p_n^J$ in clusters and substitute the corresponding points by a representative in the cluster. Up to re-ordering, assume that $d_{g_n}(p_n^1,p_n^2) =\inf \{ d_{g_n}(p_n^i, p_n^j): \, i \not=j \}$ and $d_{g_n}(p_n^i,p_n^j) \leq C d_{g_n}(p_n^1,p_n^2)$ for all $i,j=1,\dots,I$, $i \not= j$, for some $C>0$, where $2\leq I\leq J$. Setting $s_n=\frac{Cd_{g_n}(p_n^1,p_n^2)}{2}$, as in the previous step by \eqref{1} the remainder volume integrals in \eqref{eq:poho-1}-\eqref{eq:poho-2} are well controlled on the disjoint balls $B_{s_n}(p_n^j)$, $j=1,\dots,I$, leading to
\begin{eqnarray}
\label{new845}
&& \mathcal{B}_{g_n}(p_n^j,B_{s_n}(p_n^j),u_n)=- \mu_n \int_{B_{s_n}(p_n^j)} e^{4 u_n} dv_{g_n}+\frac{\mu_n}{4} \oint_{\pa B_{s_n}(p_n^j)} e^{4 u_n} (x_{n,p_n^j})^i \nu_i d \s_{g_n}+o(1); \\
&& \mathcal{B}_{g_n}(p_n^j,B_{s_n}(p_n^j),a_n,u_n) = \frac{\mu_n}{4} \oint_{\pa B_{s_n}(p_n^j)} e^{4 u_n} a_n^i \nu_i \, d \s_{g_n}+o(1) \label{new846}
\end{eqnarray}
as $n\to +\infty$, for any infinitesimal vector field $(a_n^i)_i$ with constant components in a $g_n-$geodesic coordinate system $(x_{n,p_n^j}^i)_i$ centred at $p_n^j$. The key point is to replace $p_n^1,\dots,p_n^I$ by the representative $p_n^1$ in such a way that \eqref{new845}-\eqref{new846} continue to hold for $p_n^1$ with $r_n>>d_{g_n}(p_n^1,p_n^2)$, as it follows by Step 3 below.

\medskip \noindent {\bf Step 3.} Assume that 
$$ d_{g_n}(p_n^1,p_n^2) \leq d_{g_n}(p_n^i,p_n^j) \leq C d_{g_n}(p_n^1,p_n^2) \qquad \forall i,j=1,\dots,I,\, i \not= j$$
for some $C>1$ and \eqref{new845}-\eqref{new846} are valid in $B_{s_n}(p_n^j)$, $j=1,\dots,I$, for $s_n=\frac{Cd_{g_n}(p_n^1,p_n^2)}{2}$. Then \eqref{new845}-\eqref{new846} are valid in $B_{r_n}(p_n^1)$ for any $r_n>>d_{g_n}(p_n^1,p_n^2)$ provided  \eqref{1} does hold in $A_n := B_{r_n}(p_n^1) \setminus B_n$ with $d_n(y)=\min \{d_{g_n}(y,p_n^1),d_{g_n}(y,p_n^{I+1}),\dots,d_{g_n}(y,p_n^J)\}$, where $B_n :=  \displaystyle \bigcup_{j=1}^I B_{s_n}(p_n^j)$.

To see this, by \eqref{1} in $A_n$ with $d_n(y)=\min \{d_{g_n}(y,p_n^1),d_{g_n}(y,p_n^{I+1}),\dots,d_{g_n}(y,p_n^J)\}$ we deduce that { the remainder volume  integrals in \eqref{eq:poho-1}-\eqref{eq:poho-2} tend to zero in $A_n$: \begin{eqnarray}\label{part-1}
&&  \mathcal{B}_{g_n}(p_n^1,A_n,u_n) = - \mu_n \int_{A_n} e^{4 u_n} dv_{g_n} + 
  \frac{\mu_n}{4} \oint_{\pa A_n} e^{4 u_n} (x_{n,p_n^1})^i \nu_i d \s_{g_n}
  +o(1)\\
 &&   \mathcal{B}_{g_n}(p_n^1,A_n,a_n,u_n) =   \frac{\mu_n}{4} \oint_{\pa A_n} e^{4 u_n} a_n^i \nu_i d \s_{g_n}
  +o(1) \label{part-1-bis}
\end{eqnarray}
for any infinitesimal vector field $(a_n^i)_i$ which is constant in a $g_n-$geodesic coordinate system $(x_{n,p_n^1}^i)_i$ centred at $p_n^1$. Letting $a_{n,j}=\left(x_{n,p_n^1}(p_n^j)\right)^i$, we have that $a_{n,j} \to 0$ as $n \to + \infty$ and by the validity of \eqref{new845}-\eqref{new846} in $B_{s_n}(p_n^j)$, $j=1,\dots,I$, we can deduce that
\begin{eqnarray}\label{part-2}
&& \sum_{j=1}^J \left[ \mathcal{B}_{g_n}(p_n^j,B_{s_n}(p_n^j),u_n) + \mathcal{B}_{g_n}(p_n^j,B_{s_n}(p_n^j),a_{n,j}, u_n) \right] 
 = - \mu_n \int_{B_n} e^{4 u_n} dv_{g_n} \\
&& +\frac{\mu_n}{4} \sum_{j=1}^J    \oint_{\pa B_{s_n}(p_n^j)} e^{4 u_n}[a_{n,j}^i+ (x_{n,p_{n,j}})^i] \nu_i d \s_{g_n}+ o(1), \nonumber
\end{eqnarray}
and
\begin{eqnarray}\label{part-2-bis}
&& \sum_{j=1}^J  \mathcal{B}_{g_n}(p_n^j,B_{s_n}(p_n^j),a_n, u_n)  
 = \frac{\mu_n}{4} \sum_{j=1}^J    \oint_{\pa B_{s_n}(p_n^j)} e^{4 u_n}a_n^i  \nu_i d \s_{g_n}+ o(1). 
\end{eqnarray}
It is possible to orient the geodesic coordinates both at $p_n^1$ and at $p_j^n$ so that the coordinates of $y \in \pa B_n$ in these systems satisfy (with exact equality 
for the Euclidean metric)
$$
  (x_{n,p_n^1})^i(y) = a_{n,j}^i + (x_{n,p^j_n})^i(y) + o(s_n). 
$$
By Proposition \ref{bound} and a scaling argument, there exists $\tilde{C} > 0$ 
such that 
\begin{equation*}\label{eq:est-bd-An}
  |\n u_n| \leq \frac{\tilde{C}}{s_n}; \quad |\n^2 u_n| \leq \frac{\tilde{C}}{s_n^2}; \quad 
  |\n^3 u_n| \leq \frac{\tilde{C}}{s_n^3}  \qquad \quad \hbox{ on } \partial B_n.  
\end{equation*}
The last two formulas then imply that there is approximate compensation for the 
boundary integrals on $\pa B_n$ and on the inner boundaries of $\pa A_n$. More precisely, one has 
$$
  \mathcal{B}_{g_n}(p_n^1,A_n,u_n) + \sum_{j=1}^J \left[ \mathcal{B}_{g_n}(p_n^j,B_{s_n}(p_n^j),u_n) + \mathcal{B}_{g_n}(p_n^j,B_{s_n}(p_n^j),a_{n,j}, u_n) \right] 
  = \mathcal{B}_{g_n}(p_n^1,B_{r_n}(p_n^1),u_n) 
+ o(1), 
$$
and
$$ \oint_{\pa A_n} e^{4 u_n} (x_{n,p_n^1})^i \nu_i d \s_{g_n}+\sum_{j=1}^J    \oint_{\pa B_{s_n}(p_n^j)} e^{4 u_n}[ a_{n,j}^i+(x_{n,p_{n,j}})^i] \nu_i d \s_{g_n}= \oint_{\pa B_{r_n}(p_n^1)} e^{4 u_n} (x_{n,p_n^1})^i \nu_i d \s_{g_n}+o(1) .
$$
The latter formulas, together with \eqref{part-1} and \eqref{part-2}, imply the validity of \eqref{new845} for $r_n$ and $p_n^1$. Summing up \eqref{part-1-bis} and \eqref{part-2-bis}, we also deduce the validity of \eqref{new846} for $r_n$ and $p_n^1$.}

\medskip \noindent {\bf Conclusion.} We arrange the remaining points $p_n^{I+1},\dots,p_n^J$, if any, in clusters in a similar way and substitute them by a representative. We continue to arrange the representative points in clusters and to perform a substitution thanks to Step 3. At the end, we  find a unique cluster which we collapse again to a single point $p_n$, obtaining the validity of \eqref{new845} for $p_n$ and $r>0$ with $o_n(1)+o_r(1)$ as in Step 2. Letting $n \to +\infty$ and then $r \to 0^+$ we get that
 $$-\beta=- [9 \gamma_3 \alpha^4 + (\gamma_2 + 12 \gamma_3) \alpha^2 + 24 \gamma_3 \alpha^3] 2 \pi ^2 .$$
Comparing with \eqref{eq:alpha-beta} we deduce that $\alpha=-2$  and $\beta =  8 \pi^2 \gamma_2$, completing the proof of Lemma \ref{l:quant}. 
\end{pf}    

\medskip

\begin{rem}\label{r:gur}
By studying the  point-wise limiting behaviour of the rescaled blowing-up solutions, it should be 
possible to obtain the spherical profiles classified in \cite{GurClass}. Even without this information, 
in Lemma \ref{l:quant}
we proved that such profiles would exhaust the volume accumulating near each blow-up point. 
\end{rem}

\medskip \noindent We next have the following result.
\begin{lem} \label{c}
In the above notation, there holds $c=-\infty$. 
\end{lem}
\begin{pf} 
By contradiction assume $c>-\infty$, and fix some $p=p_i \in S$, $\tilde \beta=\tilde \beta_{p_i}$. Given $0< R \leq \min \{i_0,\frac{1}{2} \hbox{dist}(p_i,p_j):\ j \not= i\}$, we have that
$$e^{4w_n} \rightharpoonup e^{4w_0+4c} dv+\tilde \beta \, \delta_{p}$$
weakly in the sense of measures on the ball $B_R=B_R(p)$ as $n \to +\infty$. Since
$$\int_{B_r} e^{4w_n}dv \to \int_{B_r} e^{4w_0+4c}dv+\tilde \beta>\tilde \beta$$
as $n \to +\infty$ for all $0<r\leq R$, we can find a sequence $r_n \to 0$ so that
\begin{equation} \label{103100}
\int_{B_{r_n^2}} e^{4w_n}dv =\tilde \beta.
\end{equation}
Since $\int_{B_r} e^{4w_0+4c} dv \to 0$ as $r \to 0$ and
$$0 \leq \int_{B_{r_n} \setminus B_{r_n^2}} e^{4w_n}dv \leq \int_{B_r } e^{4w_n}dv -\tilde \beta
\to \int_{B_r} e^{4w_0+4c} dv$$
for all $r>0$, notice that
\begin{equation} \label{103200}
\int_{B_{r_n} \setminus B_{r_n^2}} e^{4w_n}dv \to 0
\end{equation}
as $n \to +\infty$. We consider now the scaling $u_n$ of $w_n$ as given by \eqref{new1055}, which satisfies, as already observed there, the assumptions  \eqref{154200}-\eqref{bis154200} in $B_1(0)$ with $c_n=\ov{u}_n^1$, $U_\infty=0$ and $g_\infty= \delta_{eucl}$. By \eqref{103100}-\eqref{103200} we have 
$$\int_{B_1} e^{4u_n} dv_{g_n}=\int_{B_{r_n}} e^{4w_n} dv \to \tilde \beta,$$
and
\begin{eqnarray*}
 \int_{B_1} e^{4u_n} \phi \, dv_{g_n}&=&
\phi(0) \int_{B_{r_n}} e^{4u_n} dv_{g_n}+\int_{B_{r_n}} e^{4u_n} [\phi-\phi(0)] dv_{g_n}+ \int_{B_1 \setminus B_{r_n}} e^{4u_n} \phi \, dv_{g_n}\\
&=&
\phi(0) \int_{B_{r_n^2}} e^{4w_n} dv+o(\int_{B_{r_n^2}} e^{4w_n} dv) +O(\int_{B_{r_n} \setminus B_{r_n^2}} e^{4w_n}  dv) \to \tilde \beta \phi(0)
\end{eqnarray*}
for all $\phi \in C(B_1)$ as $n \to +\infty$. Hence we deduce that
$$e^{4u_n} dv_{g_n} \rightharpoonup \tilde \beta \delta_0$$
weakly in the sense of measures on $B_1$ as $n \to +\infty$.  We now apply Lemma \ref{l:quant} to deduce that $\beta=\mu_0 \tilde \beta=8\pi^2 \gamma_2$, or equivalently $\alpha = - 2$ in view of \eqref {eq:alpha-beta}. 

\medskip \noindent Let $w_0=\displaystyle \lim_{n \to +\infty} w_n-c$ be a SOLA of \eqref{205000}. Given $0<r\leq i_0$, thanks to Remark \ref{rem2} let $w_s$ be a fundamental solution in $B_r(p)$ corresponding to $\mu_s=\beta \delta_{p}$ and the boundary values as $w_0$, namely $w_s$ solves $\mathcal{N}(w_s)+U=\beta \delta_p$ in $B_r(p)$, $w_s=w_0$ and $\partial_\nu w_s=\partial_\nu w_0$ on $\partial B_r(p)$. In order to compare $w_0$ and $w_s$, consider the following scaling of $w_0$, $w_s$ and $g$: 
$$w_{0,r}(y)=w_0[\hbox{exp}_p (r y) ]+ \log r,\quad w_{s,r}(y)=w_s[\hbox{exp}_p (r y) ]+ \log r, \quad g_r(y)=g [  \hbox{exp}_p (r y)]$$
for $|y| \leq 1$. Letting $U_r$ the $U-$curvature and $\mathcal{N}_r$ be the operator associated to $g_r$, we have that
$$\mathcal{N}_r(w_{0,r})+U_r=\mu_0 e^{4w_{0,r}+4c}+\beta \delta_p \hbox{ and }\mathcal{N}_r(w_{s,r})+U_r=\beta \delta_p \quad \hbox{in }B_1(0)$$
with $w_{0,r}=w_{s,r}$ and $\partial_\nu w_{0,r}=\partial_\nu w_{s,r}$ on $\partial B_1(0)$. According to Remark \ref{rem12}  we have the validity of \eqref{11411} on $w_{0,r}-w_{s,r}$, with constants which are uniform in $r$ in view of $g_r \to \delta_{eucl}$ in $C^4(\ov{B_1(0)})$ as $r \to 0^+$. The constant $\eta_{r}$ corresponding to $g_r$ through \eqref{eta} satisfies $\eta_r \to 0$ as $r \to 0^+$, and then \eqref{11411} simply reduces to
$$\|w_{0,r}-w_{s,r}\|_{W^{1,2,2)}} \leq C_0 (\|\mu_0 e^{4w_{0,r}+4c}\|_1^{\frac{1}{12}}+\eta_r^{\frac{4}{3}})\qquad \hbox{(w.r.t. $g_r$)}$$
for some $C_0>0$ in view of \eqref{19139} and
$$\int_{B_1(0)} e^{4w_{0,r}}dv_{g_r}=\int_{B_r(p)} e^{4w_0}dv \leq C ,\quad \int_{B_1(0)} |U_r| \, dv_{g_r}=\int_{B_r(p)} |U| \, dv \leq C.$$
In particular, there holds
\begin{eqnarray} \label{14534} 
\epsilon^{\frac{1}{4}} (\int_{B_1(0)} |\nabla (w_{0,r}-w_{s,r})|^{4(1-\epsilon)} dv_{g_r})^{\frac{1}{4(1-\epsilon)}} \leq C_0 (\|\mu_0 e^{4w_{0,r}+4c}\|_1^{\frac{1}{12}}+\eta_r^{\frac{4}{3}})
\end{eqnarray}
for some $C_0>0$ and for all $0<\epsilon \leq \epsilon_0$. 

\medskip \noindent In order to derive exponential estimates from \eqref{14534}, let us recall the optimal Euclidean inequality
\begin{equation} \label{09519}
(\int_{\mathbb{R}^4} |U|^k dx)^{\frac{1}{k}} \leq C(k) (\int_{\mathbb{R}^4} |\nabla U|^{\frac{4k}{4+k}} dx)^{\frac{4+k}{4k}} \qquad  U \in C_0^\infty(\mathbb{R}^4)
\end{equation}
for all $k \geq 1$, where 
$$C(k)=\pi^{-\frac{1}{2}}4^{-\frac{4+k}{4k}} (\frac{3k-4}{16})^{\frac{3k-4}{4k}}\left[ \frac{\Gamma(3) \Gamma(4)}{\Gamma(\frac{4+k}{k}) \Gamma(\frac{15k -20 }{4k})}\right]^{\frac{1}{4}}$$
see \cite{Aub,Tal}. One has the following behaviour 
\begin{equation} \label{10100}
\frac{C(k)}{k^{\frac{3}{4}}}  \to  C_1= \frac{3}{8} \pi^{-\frac{1}{2}}  \Gamma^{-\frac{1}{4}}(\frac{15}{4}) \quad \hbox{ as } k \to + \infty.
\end{equation}
Since $w_{0,r}-w_{s,r}  \in W^{1,2}_0(B_1(0))$, we can extend it as zero outside $B_1(0)$ into a function $U \in D^{1,4}(\mathbb{R}^4)$. Since by density \eqref{09519} does hold for $U$, by \eqref{14534} we have that
\begin{eqnarray*}
\int_B e^{q |w_{0,r}-w_{s,r}|} dv_{g_r} & \leq & 2 \sum_{k=0}^\infty \frac{q^k}{k!} \int_{\mathbb{R}^4}  |U|^k dx \leq 2 \sum_{k=0}^\infty \frac{q^k C(k)^k}{k!} (\int_{\mathbb{R}^4} |\nabla U|^{4(1-\frac{4}{4+k})} dx)^{\frac{4+k}{4}} \\
&\leq & 4 \sum_{k=0}^\infty \frac{2^{\frac{k}{4}} q^k C(k)^k}{k!} \left( \int_B |\nabla(w_{0,r}-w_{s,r})|^{4(1-\frac{4}{4+k})} dv_{g_r} \right)^{\frac{4+k}{4}} \\
& \leq & 4 \sum_{k=0}^\infty \frac{ q^k C(k)^k}{k!} (\frac{4+k}{2})^{\frac{k}{4}} C_0^k (\|\mu_0 e^{4w_{0,r}+4c}\|_1^{\frac{1}{12}}+\eta_r^{\frac{4}{3}})^k
\end{eqnarray*}
in view of $\frac{dx}{2}\leq dv_{g_r}\leq 2 dx$ for $r>0$ small. Thanks to \eqref{10100} we have that
$$\frac{C(k)^k}{k!} (\frac{4+k}{2})^{\frac{k}{4}} \sim \frac{ C_1^k k^k }{2^{\frac{k}{4}} k!}\sim \frac{ e^k C_1^k}{\sqrt k 2^{\frac{k}{4}} }$$
in view of Stirling's formula. Then $e^{|w_{0,r}-w_{s,r}|} \in L^q(B_1(0))$ w.r.t. $g_r$ for all $q<q_r$, where
$$q_r=\frac{2^{\frac{1}{4}}}{eC_0C_1 (\|\mu_0 e^{4w_{0,r}+4c}\|_1^{\frac{1}{12}}+\eta_r^{\frac{4}{3}})}. $$
Since $q_r \to 0$ as $r \to 0^+$, we deduce that
\begin{equation} \label{new1051}
r^{-4} \int_{B_r(p)} e^{q |w_0-w_s|} dv  =\int_B e^{q |w_{0,r}-w_{s,r}|} dv_{g_r} <+\infty
\end{equation}
for all $q\geq 1$ provided $r>0$ is sufficiently small. Since $w_s$ satisfies \eqref{12292} in $B_r(p)$ with $\alpha=-2$ in view of Remark \ref{rem4}, we have that $w_s=-2(1+o(1)) \log |x|$ as $x \to 0$ in geodesic coordinates near $p$ and then $\int_{B_r(p)} e^{\gamma w_s} dv=+\infty$ for $\gamma>2$. This is in contradiction for $\gamma<4$ to 
$$\int_{B_r(p)} e^{\gamma w_s} dv =\int_{B_r(p)} e^{\gamma (w_s-w_0)} e^{\gamma  w_0} dv \leq
(\int_{B_r(p)} e^{\frac{4\gamma}{4-\gamma} (w_s-w_0)} dv)^{\frac{4-\gamma}{4}} (\int_{B_r(p)} e^{4  w_0},  dv)^{\frac{\gamma}{4}}<+\infty,$$
in view of $\int e^{4w_0} dv <+\infty$ and \eqref{new1051}. This concludes the proof that $c=-\infty$.
\end{pf}

\medskip \noindent Once we  established that $c=-\infty$, we have that 
$$ \mu_n e^{4w_n} \rightharpoonup  \sum_{i=1}^l \beta_i \delta_{p_i} \qquad 
 \hbox{ as } n \to +\infty $$
weakly in the sense of measures. We apply Lemma \ref{l:quant} near each $p_i$, ending up with $\beta_i=8\pi^2 \gamma_2$ for all $i=1,\dots,l$. The proof of Theorem \ref{p:minimal-blow-up} is now complete.

\section{Moser-Trudinger inequalities and existence results}\label{s:mt-ex}

In this section we show first a sharp Moser-Trudinger inequality of independent interest. 
We also derive an improved version of Adams' inequality involving also the functional 
$III$, a crucial ingredient   for 
the existence of critical metrics in Theorem \ref{t:ex} via a variational and 
topological argument. 

\subsection{Sharp and improved Moser-Trudinger inequalities}
\setcounter{equation}{0}

In \cite{ChangYangAnnals} (see also \cite{Adams}), the following inequality was proved 
\begin{equation}\label{eq:ineq-CY}
\log \int e^{4  w}  dv \leq \frac{1}{8 \pi^2}
       \int (\Delta u)^2 dv + 4 \ov{w} + C_g
      \qquad \quad \hbox{ for all } w \in W^{2,2}(M).
\end{equation}
If the Paneitz operator is positive-definite (see \eqref{eq:Paneitz}), the integral of $(\Delta u)^2$ in the R.H.S. of the above formula can be replaced by the quadratic form induced by $P$. We have the following sharp inequality despite of the sign of the Paneitz 
operator, see also  \cite{DJLW, NolTar} for related results. 

\begin{thm}\label{t:sharp-ineq}
Suppose  $\int U dv \leq 8 \pi^2 \g_2$. Then, if $F_\g = \gamma_1 I + \gamma_2 II + \gamma_3 III$ 
with $\g_2, \g_3  > 0$ and $\tfrac{\g_2}{\g_3} > \tfrac{3}{2}$, then  
for all functions $w \in W^{2,2}(M)$ one has the lower bound 
$$
 F_\g(w) \geq - C
$$
for some constant $C$.
\end{thm}

\begin{pf}
For $\e  > 0$, consider the following functional 
$$
  F_\e(w) := F_\g(w) + \e \log \left( \int e^{4(w- \ov{w})} dv \right). 
$$
 Supposing by contradiction that $F_\g$ is unbounded from below, we then have that
$$
  m_\e := \inf_{W^{2,2}} F_\e \to - \infty \qquad \quad \hbox{ as } \e \searrow 0. 
$$
By \eqref{eq:ineq-CY} (and some easy reasoning, exploiting the quartic gradient terms, if the Paneitz operator has negative eigenvalues) we know that $F_\e$ admits a global minimum, which we call $w_\e$. 
Hence we have that $F_\e(w_\e) = m_\e \to - \infty$ as $\e \searrow 0$. 

Looking at the Euler-Lagrange equation satisfied by $w_\e$, by Theorem \ref{t:weak-norms} it follows that $\int |\n w_\e|^2 dv \leq C$. W.l.o.g., assume also that $\ov{w}_\e = 0$.
Therefore, from the explicit form of $F_\e$ and Poincar\'e's inequality, 
we have that 
$$
  m_\e = F_\e(w_\e) \geq \g_2 \int (\D w_\e)^2 dv - ( 8 \pi^2 \gamma_2 - \e) \log \left( \int e^{4(w- \ov{w})} dv \right) - C. 
$$
Inequality \eqref{eq:ineq-CY} and the last formula imply that $F_\e(w_\e) \geq - 2 C$, which contradicts 
$F_\e(w_\e) \to - \infty$ as $\e \searrow 0$. 
\end{pf}

\noindent Next, we show that if $e^{4w}$ has integral bounded from
below into $(\ell + 1)$ distinct regions of $M$, the Moser-Trudinger constant
can be basically divided by $(\ell + 1)$.
When dealing with the functional $II$ only, such an inequality was proved in 
\cite{DM08}, relying on some previous argument in \cite{CL91}. 

\begin{lem}\label{l:imprc}
Suppose $\g_2, \g_3 > 0$. 
For a fixed integer $\ell$, let $\Omega_1, \dots, \Omega_{\ell+1}$
be subsets of $M$ satisfying $dist(\Omega_i,\Omega_j) \geq \d _0$
for $i \neq j$, where $\d_0$ is a positive real number, and let
$\g_0 \in \left( 0, \frac{1}{\ell+1} \right)$. Then, for any
$\tilde{\e} > 0$ there exists a constant $C = C(\ell, \tilde{\e},
\d_0, \g_0)$ such that
$$
  8 (\ell + 1) \pi^2 \log \int e^{4(w - \ov{w})} dv \leq \ (1 + \tilde{\e}) 
  \left( \langle w, P w \rangle + \frac{\gamma_3}{\gamma_2} III(w) \right) + C 
$$
for all the functions $w \in W^{2,2}(M)$ satisfying
$$    \frac{\int_{\Omega_i} e^{4w} dv}{\int e^{4w} dv} \geq \g_0,
  \qquad \quad \quad \forall \; i \in \{1, \dots, \ell+1\}.$$
\end{lem}

\begin{pf}
Assume without loss of generality that $\ov{w} = 0$. With the same argument as in the 
proof of Lemma 2.2 in \cite{DM08}, it is possible to show under the above conditions that 
$$
  8 (\ell + 1) \pi^2 \log \int e^{4(w - \ov{w})} dv \leq \ (1 + \tfrac{\tilde{\e}}{2}) 
  \int (\D u)^2 dv + C. 
$$
Relabelling $C$, it is then enough to prove the inequality 
\begin{equation}\label{eq:show}
 (1 + \tfrac{\tilde{\e}}{2}) 
   \int (\D u)^2 dv \leq (1 + \tilde{\e}) 
     \left( \langle w, P w \rangle + \frac{\gamma_3}{\gamma_2} III(w) \right) + C. 
\end{equation}
However, using Poincar\'e's inequality and the expressions of $P$ and $III$ we can write that 
$$
  \langle w, P w \rangle +  \frac{\gamma_3}{\gamma_2} III(w) \geq 
  \int (\D u)^2 dv + 12  \frac{\g_3}{\g_2} \int (\D u + |\n u|^2)^2 dv - C \int |\n w|^2 - C.  
$$
For $\varsigma > 0$ sufficiently small, one has that 
$$
  \int (\D u)^2 dv + 12  \frac{\g_3}{\g_2} \int (\D u + |\n u|^2)^2 dv \geq  (1 - 2\varsigma) \int (\D u)^2 dv 
  + \varsigma \int |\n u|^4 dv
$$
Choosing $\varsigma$ small compared to $\tilde{\e}$ and using  Young's inequality, from the 
last two formulas we obtain \eqref{eq:show}, yielding the conclusion. 
\end{pf}

For $j \in \N$, we define the family of probability measures 
$$
  M_j = \left\{ \mu \in \mathcal{P}(M) \; : \; card(supp(\mu)) \leq j  \right\}. 
$$
We define the {\em distance} of  an $L^1-$
function $f$ in $M$  from $M_j$, $j \leq k$, as
$$    {\bf d}(f,M_j) = \inf_{\s \in M_j} \sup \left\{ \left| \int f  \, \psi \, dv - \langle \s, \psi 
  \rangle \right| \; : \;  \|\psi\|_{C^1(M)} \leq 1 \right\},  $$
where $\langle \s, \psi \rangle$ stands for the duality product between
$\mathcal{P}(M)$ and the space of $C^1$ functions. From Lemma \ref{l:imprc} and 
Poincar\'e's  inequality (to treat linear terms in $w$) we deduce immediately the following result.

\begin{pro}\label{l:diste}
Suppose that $\g_2, \g_3 > 0$ and that $\int U dv < 8(k+1) \g_2 \pi^2$ with $k \geq 1$. Then for
any  $\e > 0$ there exists a large positive $\Xi = \Xi (\e)$ such that
for every $w \in W^{2,2}(M)$ with $F_\g(w) \leq - \Xi $ and $\int
e^{4w} dv = 1$, we have ${\bf d}(\tfrac{e^{4w}}{\int e^{4w}},M_k) \leq \e$.
\end{pro}

\

\noindent From the result in Section 3 of \cite{DM08}, one can deduce a further continuity property 
from $W^{2,2}(M)$ into $\mathcal{P}(M)$, endowed with the above distance {\bf d}.

\begin{pro}\label{c:Psik} For $\g_2, \g_3 > 0$ and $\int U dv < 8(k+1) \g_2 \pi^2$ there exist a large positive number  $\Xi $ and a continuous  map 
$\Psi_k : \left\{ F_\g \leq - \Xi   \right\} \to M_k$ such that, if $e^{2 w_n} \rightharpoonup \s \in M_k$, 
then $\Psi_k(w_n) \rightharpoonup \s$.  
\end{pro}

\medskip

\subsection{The topological argument}\label{s:msub}

The proof essentially follows the lines of Section 4 in \cite{DM08}, so we will 
mainly recall the principal steps. We first map $M_k$ into some low sub-levels of
$F_\g $ and finally, once we map back 
onto $M_k$ using Proposition \ref{c:Psik}, we obtain a map homotopic to the identity. 
The main difference with respect to the above reference is the 
energy estimate in Lemma \ref{l:pff}, where we need to estimate the functional $III$ on 
suitable test functions. We first recall a topological characterization of $M_k$. 
\begin{lem}\label{l:nonc} (\cite{DM08})
For any $k \geq 1$, the set $M_k$ is a {\em stratified set},
namely union of open manifolds of different dimensions, whose
maximal one is $3k - 1$. Furthermore $M_k$ is non-contractible.
\end{lem}

For $\d > 0$ small, consider a
smooth non-decreasing cut-off function $\chi_\d : \R_+ \to \R$
such that 
$$    \left\{%
\begin{array}{ll}
    \chi_\d(t) =  t & \hbox{ for } t \in [0,\d] \\
    \chi_\d(t) =  2 \d & \hbox{ for } t \geq 2 \d \\
    \chi_\d(t) \in [\d, 2 \d] & \hbox{ for } t \in [\d, 2 \d]. \\
\end{array}%
\right. $$
Then, given $\s \in M_k$ $\left( \hbox{i.e. }\s = \displaystyle \sum_{i=1}^k t_i \d_{x_i}
\right)$ and $\l
> 0$, we define the function $\var_{\l,\s} : M \to \R$ as
 \begin{equation}\label{eq:pls}
  \var_{\l,\s} (y) = \frac 14 \log
  \sum_{i=1}^k t_i \left( \frac{2 \l}{1 + \l^2 \chi_\d^2
  \left( d(y,x_i) \right)} \right)^4, \quad \quad y \in M.
\end{equation}

\

\noindent We prove next an energy estimate on the above test functions.

\begin{lem}\label{l:pff}
Suppose that $\g_2, \g_3 > 0$ and that $\var_{\l,\s}$ is as in \eqref{eq:pls}. Then as $\l \to +
\infty$ one has
$$
   F_\g(\var_{\l,\s})  \leq
   \left( 32 k \pi^2 \gamma_2 + o_\d(1) \right) \log \l + C_\d $$
uniformly in $\s \in M_k$, where $o_\d(1) \to 0$ as $\d \to 0$ and  $C_\d$ is a constant
independent of $\l$ and $x_1,\dots,x_k$.
\end{lem}

\begin{pf}
In \cite{DM08} it was proven that 
$$
  \langle P \var_{\l,\s}, \var_{\l,\s} \rangle \leq
   \left( 32 k \pi^2 + o_\d(1) \right) \log \l + C_\d 
$$
does hold uniformly in $\s \in M_k$, and moreover, as for formula (40) in \cite{DM08}, one has that 
$$
 \left| \int U (\var_{\l,\s} - \overline{ \var}_{\l,\s}) dv \right| \leq   o_\d(1)  \log \l + C_\d.
$$
Therefore it is sufficient to show that the following estimate  
\begin{equation}\label{eq:est-III}
  |III(\var_{\l,\s})| = o_\l(1) \log \l 
\end{equation}
does hold uniformly in $\s \in M_k$. In order to do this, we can focus on the term $(\D \var_{\l,\s} + |\n \var_{\l,\s}|^2)^2$, 
since the others are shown in \cite{DM08} to be of lower order. 
Setting 
$$
  \mathcal{F}_i(y) := \frac{2 \l}{1 + \l^2 \chi_\d^2    \left( d(y,x_i) \right)}, 
$$
we compute  explicitly $\D \var_{\l,\s} + |\n \var_{\l,\s}|^2$: 
$$
    \D \var_{\l,\s} + |\n \var_{\l,\s}|^2 = \frac{\sum_i t_i \mathcal{F}_i^3 \D \mathcal{F}_i}{\sum_j t_j \mathcal{F}_j^4} 
    + 3 \frac{\sum_i t_i \mathcal{F}_i^2 |\n \mathcal{F}_i|^2}{\sum_j t_j \mathcal{F}_j^4} - 3 
    \frac{\left| \sum_i t_i \mathcal{F}_i^3 \n \mathcal{F}_i \right|^2}{\left( \sum_j t_j \mathcal{F}_j^4 \right)^2}.   
$$
This can be rewritten as 
$$
   \D \var_{\l,\s} + |\n \var_{\l,\s}|^2 = \frac{\sum_i t_i \mathcal{F}_i^3 \D \mathcal{F}_i}{\sum_j t_j \mathcal{F}_j^4} 
      + 3 
      \frac{\sum_{i,k} t_i t_k \mathcal{F}_i^2 \mathcal{F}_k^2 \left( \mathcal{F}_k^2 |\n \mathcal{F}_i|^2 - \mathcal{F}_i \mathcal{F}_k \n \mathcal{F}_i \cdot \n \mathcal{F}_k \right)}{\left( \sum_j t_j \mathcal{F}_j^4 \right)^2}.
$$
At this point, symmetrizing in $i, k$ and playing with elementary inequalities,  it is enough to uniformly estimate in terms of $o_\l(1) \log \l $ the square $L^2$-norm of the following quantities
\begin{equation}\label{eq:main-terms}
  \frac{\D \mathcal{F}_i}{\mathcal{F}_i}; \qquad \quad \qquad \mathcal{G}_{i,k} := \frac{|\mathcal{F}_k^2 \mathcal{F}_i \n \mathcal{F}_i - \mathcal{F}_i^2 \mathcal{F}_k \n \mathcal{F}_k|^2}{( \mathcal{F}_i^4 + \mathcal{F}_k^4)^2}. 
\end{equation}
For the first term, working in normal coordinates $y$ at $x_i$ one finds  
$$
   \D \mathcal{F}_i(y) = \D_{\delta_{eucl}}\mathcal{F}_i (y) + O(|y|) |\n \mathcal{F}_i|(y) + O(|y|^2) |\n^2 \mathcal{F}_i|(y). 
$$
Using also the fact that 
$$
    \D_{\delta_{eucl}}\left(  \frac{1}{1+\l^2 |x|^2} \right)  =  -\frac{8  \l^2}{\left(1+\l^2 |x|^2\right)^3},
$$
one gets the following  bounds
 $$
  \mathcal{F}_i(y) \geq \begin{cases}
  C^{-1} \l & d(y,x_i) \leq \frac{1}{\l}; \\ 
  \frac{C}{\l d^2(y,x_i)} & \frac{1}{\l} \leq d(y,x_i)\leq \delta, 
  \end{cases}  \qquad \quad 
  |\D \mathcal{F}_i(y)| \leq \begin{cases}
  C \l^3 & d(y,x_i) \leq \frac{1}{\l}; \\ 
  \frac{C}{\l^3 d^6(y,x_i)} & \frac{1}{\l} \leq d(y,x_i)\leq \delta .  
  \end{cases}
$$
These imply 
$$
  \int \left(\frac{\D \mathcal{F}_i}{\mathcal{F}_i} \right)^2 dv \leq \int_{B_{\frac{1}{\l}}(x_i)} C  \lambda^4 \, dv 
  + \int_{ B_\delta(x_i) \setminus B_{\frac{1}{\l}}(x_i)} \frac{C}{\l^4 d^8(y,x_i)} dv  +C \leq C. 
$$
For the latter quantity in \eqref{eq:main-terms} we distinguish between two cases. 

\

\noindent {\bf Case 1:} $d(x_i,x_k) \geq \frac{\d}{2}$. When we integrate near $x_i$, $\mathcal{F}_k$ 
and its gradient are bounded by $\frac{C_\d}{\l}$. Using also the fact that 
$$
  |\n \mathcal{F}_i| \leq \frac{C \l^3 d(y,x_i)}{\left(1+\l^2 d^2(y,x_i)\right)^2}, 
$$
we find the upper bound 
$$
  \int_{B_{\frac{\d}{4}}(x_i)} \mathcal{G}_{i,k}^2
  dv \leq  C \int_{B_{\frac{\d}{4}}(x_i)} \left[ \frac{d(y,x_i)^4 (1 + \l^2 d^2(y,x_i))^4}{\l^8} 
  +  \frac{(1 + \l^2 d^2(y,x_i))^8}{\l^{16}}  \right] dv \leq C, 
$$
where the latter inequality follows from a change of variable. In the same way, one 
finds a similar bound on  $B_{\frac{\d}{4}}(x_k)$. In the exterior of these two balls, 
it is easily seen that $\mathcal{G}_{i,k}$ is uniformly bounded, and therefore 
$\mathcal{G}_{i,k}$ is uniformly bounded also in $L^2(M)$. In particular, there holds $\int \mathcal{G}_{i,k}^2 dv = o_\l(1) \log \l$.

\

\noindent {\bf Case 2:} $d(x_i,x_k) \leq \frac{\d}{2}$. In this case the functions 
$\mathcal{F}_i$ and $\mathcal{F}_k$ can be simultaneously {\em large} at the same  point.  By symmetry, it is 
sufficient to estimate $\mathcal{G}_{i,k}$ in the set 
$$  M_{i,k} := \left\{ y \in M \; : \; d(y,x_i) \leq d(y,x_k) \right\}.$$
Set $\eta_{i,k}=\max\{\frac{1}{\lambda}, d(x_i,x_k)\}$. In $\left( M_{i,k} \cap B_\d (x_i) \right) \setminus B_{\mathfrak{C} \eta_{i,k}}(x_i)$, $\mathfrak{C}\geq 1$, one has the estimates 
$$ \mathcal{F}_k =   \mathcal{F}_i (1+o_\mathfrak{C}(1)),\quad \n \mathcal{F}_k = \n \mathcal{F}_i + o_\mathfrak{C}(1) |\n \mathcal{F}_i| $$
with $o_{\mathfrak{C}}(1) \to 0$ as $\mathfrak{C} \to + \infty$, in view of
$$1\leq \frac{d(y,x_k)}{d(y,x_i)} \leq 1+\frac{d(x_i,x_k)}{d(y,x_i)} \leq 1 +\frac{1}{\mathfrak{C}}.$$
Since these estimates imply some cancellations in the numerator of $\mathcal{G}_{i,k}$, we have that 
$$  \mathcal{G}_{i,k}^2 \leq \frac{o_\mathfrak{C}(1)}{|y-x_i|^4} \quad   \hbox{ in } \left(M_{i,k} \cap  B_\d(x_i) \right) \setminus  B_{\mathfrak{C} \eta_{i,k}}(x_i), $$
and therefore we find 
\begin{equation} \label{eq:intGik-3}
\int_{\left(M_{i,k} \cap B_\d(x_i) \right) \setminus  B_{\mathfrak{C} \eta_{i,k}(x_i)}} \mathcal{G}_{i,k}^2 dv = o_\mathfrak{C}(1) \log \l. 
\end{equation}
In $\left(M_{i,k} \cap B_\delta(x_i)\cap B_{\mathfrak{C} \eta_{i,k}}(x_i)\right)  \setminus B_{\frac{1}{\l}}(x_i) $ we next have the following inequalities 
$$ \frac{1}{\l d^2(y,x_i)}\leq  \mathcal{F}_i \leq \frac{2}{\l d^2(y,x_i)}, \quad  |\n \mathcal{F}_i| \leq \frac{C}{\l d^3(y,x_i)}, \quad  |\mathcal{F}_k |   \leq \frac{C}{\l \eta_{i,k}^2}, \quad |\n \mathcal{F}_k | \leq \frac{C}{\l \eta_{i,k}^3}$$
in view of
$$d(y,x_k)\geq \left\{ \begin{array}{ll}d(x_i,x_k)-d(y,x_i) \geq \frac{1}{2}\eta_{i,k}&\hbox{if } \frac{1}{\lambda} \leq d(y,x_i)\leq \frac{1}{2}d(x_i,x_k)\\
d(y,x_i) \geq \frac{1}{2} \eta_{i,k}& \hbox{if } y \in M_{i,k}, \ d(y,x_i)\geq \frac{1}{2}\eta_{i,k}, \end{array} \right.$$
which imply 
$$ \mathcal{G}_{i,k}^2 \leq  C \left( \frac{d^{12}(y,x_i)}{\eta_{i,k}^{16}} + \frac{d^{16}(y,x_i)}{\eta_{i,k}^{20}} \right),$$ 
and therefore  
\begin{equation}\label{eq:intGik-2}
\int_{\left( M_{i,k} \cap B_\delta(x_i) \cap B_{\mathfrak{C}\eta_{i,k}}(x_i)\right)  \setminus B_{\frac{1}{\l}}(x_i) } \mathcal{G}_{i,k}^2 dv  \leq C \mathfrak{C}^{20}.
\end{equation}
Finally the estimate $\frac{|\n \mathcal{F}_i|}{\mathcal{F}_i}+\frac{|\n \mathcal{F}_k|}{\mathcal{F}_k} \leq C \l$ implies 
\begin{equation}\label{eq:intGik-1}
  \int_{M_{i,k} \cap B_{\frac{1}{\l}}(x_i)} \mathcal{G}_{i,k}^2 dv \leq C. 
\end{equation}
By first choosing $\mathfrak{C}$ and then $\lambda$ large, by \eqref{eq:intGik-3}-\eqref{eq:intGik-1} we have shown that $\int_{M_{i,k} \cap B_\delta(x_i)} \mathcal{G}_{i,k}^2 dv = o_\l(1) \log \l$. By the symmetry of $\mathcal{G}_{i,k}$, exchanging $i$ and $k$ we also have that 
$$\int_{M_{k,i} \cap B_{\frac{\delta}{2}}(x_i)} \mathcal{G}_{i,k}^2 dv \leq \int_{M_{k,i} \cap B_\delta(x_k)} \mathcal{G}_{k,i}^2 =o_\l(1) \log \l,$$
which combines with 
$$\int_{M \setminus B_{\frac{\d}{2}}(x_i)\cup B_{\frac{\d}{2}}(x_k)} \mathcal{G}_{i,k}^2 dv \leq C$$ 
to show that also in {\bf Case 2} there holds $\int \mathcal{G}_{i,k}^2 dv = o_\l(1) \log \l$. 
\end{pf}

\medskip The above results can be collected into the following proposition.

\

\begin{pro}\label{p:sublev}
Suppose that $\g_2, \g_3 > 0$,  $\int U dv \in  (8 k \g_2 \pi^2, 8 (k+1) \g_2 \pi^2)$, and let  $\var_{\l,\s}$ be defined as in \eqref{eq:pls}. Then, as $\l \to
+ \infty$ the following properties hold true

\

\noindent (i) $e^{4 \var_{\l,\s}} \rightharpoonup \s$ weakly in the
sense of distributions;

\

\noindent (ii) $F_\g (\var_{\l,\s}) \to - \infty$ uniformly in $\s
\in M_k$;

\

\noindent  (iii) if 
$\Psi_k$ is given by Proposition \ref{c:Psik} and if $\var_{\l,\s}$ is as in \eqref{eq:pls}, then for
$\l$ sufficiently large the map $\s \mapsto \Psi_k(\var_{\l,\s})$ is
homotopic to the identity on $M_k$.
\end{pro}
\noindent We next introduce a variational scheme for obtaining existence of 
solutions of the Euler-Lagrange equation. Let $\hat{M}_k$ be the {\em topological cone} over $M_k$, which can be represented as
$\hat{M}_k = M_k \times [0,1]$ with $M_k \times \{0\}$
identified to a single point. Let first $\Xi $ be so large that
Proposition \ref{c:Psik} applies with $\frac \Xi 4$, and then let
$\ov{\l}$ be so large that $F_\g(\var_{\ov{\l},\s}) \leq - \Xi $
uniformly for $\s \in M_k$ (see Proposition \ref{p:sublev} $(ii)$). Fixing this value
of $\ov{\l}$, we define the family of maps 
\begin{equation}\label{eq:PiPi}
    \Pi_{\ov{\l}} = \left\{ \varpi : \hat{M}_k \to W^{2,2}(M)
 \; : \; \varpi \hbox{ is continuous and } \varpi(\cdot \times \{1\})
 = \var_{\ov{\l},\cdot} \hbox{ on } M_k \right\}.
\end{equation}

\begin{lem}\label{l:min-max}
 $\Pi_{\ov{\l}}$ is non-empty and moreover, letting
$$
  \ov{\Pi}_{\ov{\l}} = \inf_{\varpi \in \Pi_{\ov{\l}}}
  \; \sup_{m \in \hat{M}_k} F_\g(\varpi(m)), \qquad
  \hbox{ one has } \qquad \ov{\Pi}_{\ov{\l}} > - \frac
  \Xi 2.
$$

\end{lem}

\begin{pf}
To show that  $\Pi_{\ov{\l}} \neq \emptyset$, it  suffices to consider 
the map
\begin{equation}\label{eq:ovPi}
  \ov{\varpi}( \sigma,t) = t \var_{\ov{\l},\sigma}, \qquad \qquad (\sigma,t)
  \in \hat{M}_k. 
\end{equation} 
Arguing  by contradiction, suppose that 
$\ov{\Pi}_{\ov{\l}} \leq  - \frac \Xi 2$. Then there would exist a map $\varpi
\in \Pi_{\ov{\l}}$ with $\sup_{m \in \hat{M}_k} F_\g(\varpi(m))
\leq - \frac 38 \Xi $.  Since by our choice of $\Xi $ Proposition \ref{c:Psik} applies
with $\frac \Xi 4$, writing $m = (\sigma, t)$, with $\sigma \in M_k$, the map
$$
  t \mapsto \Psi \circ \varpi(\cdot,t)
$$
realizes a homotopy in $M_k$ between $\Psi \circ
\var_{\ov{\l},\cdot}$ and a constant map. However this cannot be, as 
$M_k$ is non-contractible (see Lemma \ref{l:nonc}) and since
$\Psi \circ \var_{\ov{\l},\cdot}$ is homotopic to the identity on
$M_k$, by Proposition \ref{p:sublev} $(iii)$. Hence we deduce
$\ov{\Pi}_{\ov{\l}} > - \frac \Xi 2$.
\end{pf}

\

\noindent By the statement of Lemma \ref{l:min-max} and standard variational 
arguments, one can find a {\em Palais-Smale sequence} $(w_n)_n$ for 
$F_\g$ at level $\ov{\Pi}_{\ov{\l}}$, namely a sequence for which 
$$
  F_\g(w_n) \to \ov{\Pi}_{\ov{\l}}; \qquad \qquad 
  \nabla F_\g(w_n) \to 0. 
$$
Unfortunately it is not known whether Palais-Smale sequences converge. To show 
this property, from the fact that $w \mapsto e^{4w}$ is compact from 
$W^{2,2}(M)$ to $L^1(M)$, it would be sufficient to show that any Palais-Smale sequence 
is bounded  in $W^{2,2}$. 

This is in fact proven  indirectly, following an argument in \cite{str}, by 
making in the functional $F_\g$ the substitutions  $\int Q dv \mapsto t \int Q dv$, $\gamma_1 \mapsto t \gamma_1$, $\mu \mapsto t \mu$ and $II \ \mapsto \ II - \Theta (t-1) \gamma_2 \int |\n w|^2 dv$ for $t$ close to $1$, 
 where $\Theta$ 
 is a large positive constant ($\Theta$ can be taken zero if $P$ has no negative eigenvalues). 
We choose a small $t_0 > 0$, and allow $t$ to
vary  in the interval $[ 1 - t_0, 1 + t_0]$. 
We consider then the functional $F_\g$ for these values of $t$, 
denoting it by $(F_\g)_t$. If
$t_0$ is sufficiently small, the interval $\left[ (1 - t_0) \int U dv, (1 +
t_0) \int U  dv \right]$ will be compactly contained in $(8 k \g_2 \pi^2, 8 (k+1) \g_2 \pi^2)$.
Following the  previous estimates with minor changes, one easily checks
that the  min-max scheme applies uniformly for $t \in [ 1 -
t_0, 1 + t_0 ]$ and for $\ov{\l}$ sufficiently large. 
Precisely, given any large  $\Xi  > 0$, there exist $t_0$
sufficiently small and $\ov{\l}$ so large that for $t \in [1 -
t_0, 1 + t_0]$
$$    \sup_{m \in \partial
   \hat{M}_k} (F_\g)_t (\varpi(m)) < - 2 \, \Xi ; \quad
   \ov{\Pi}_{t} := \inf_{\varpi \in \Pi_{\ov{\l}}}
  \; \sup_{m \in \hat{M}_k} (F_\g)_t (\varpi(m)) > -
  \frac{\Xi }{2},$$
where $\Pi_{\ov{\l}}$ is defined in \eqref{eq:PiPi}. Moreover, using
for example the test map \eqref{eq:ovPi}, one shows that for
$t_0$ sufficiently small there exists a large constant $\ov{\Xi }$
such that
$$  \ov{\Pi}_{t} \leq \ov{\Xi } \qquad \qquad \hbox{ for every }
  t \in [1 - t_0, 1 + t_0].$$

\medskip 

\noindent If the above constant $\Theta$ is chosen large enough (compared to the 
negative values of the Paneitz operator), it is easy to show that 
$ t \mapsto \frac{\ov{\Pi}_t}{t} $ is
  non-increasing in  $[1 - t_0, 1 + t_0]$. 
From this we deduce that the function $t
\mapsto \frac{\ov{\Pi}_t}{t}$ is differentiable almost
everywhere, and we obtain the following corollary.

\begin{cor}\label{c:c}
Let $\ov{\l}$ and $t_0$ be as above, and let $\L
\subset [1 - t_0, 1 + t_0]$ be the (dense) set of $t$ for
which the function $\frac{\ov{\Pi}_t}{t}$ is differentiable.
Then for $t \in \L$ the functional $F_\g$ possesses a bounded
Palais-Smale sequence $(w_l)_l$ at level $\ov{\Pi}_{t}$, weakly converging 
to a solution of 
$$  \mathcal{N} (w)  + 2 \g_2 \, \Theta \, (t-1) \, \Delta w + t \, U = t \, \mu \frac{e^{4w}}{\int e^{4w} dv}. $$
\end{cor}

\begin{pf}
The existence of a Palais-Smale sequence $(w_l)_l$ follows from
Lemma \ref{l:min-max}, and the boundedness is proved exactly as in
\cite{DJLW2}, Lemma 3.2.
\end{pf}

\

\noindent  We can finally prove our second main result. 

\

\begin{pfn} {\sc of Theorem \ref{t:ex}.} We assume that $\g_2, \g_3 > 0$: obvious 
changes have to be made for opposite signs. 
From the above result we obtain a sequence $t_n \to 1$ and a sequence 
$w_n$ solving 
$$
 \mathcal{N} (w_n)  + 2 \g_2 \, \Theta \, (t_n-1) \, \Delta w_n + t_n U = t_n \mu \frac{e^{4w_n}}{\int e^{4w_n} dv}, 
$$
which can be chosen to satisfy $\int e^{4w_n} dv = 1$ for all $n$. Since the extra term $ 2 \g_2 \, \Theta \, t_n \, \Delta w_n$ does 
not affect the analysis in Theorem \ref{p:minimal-blow-up}, we can then pass to the limit using 
assumption $\int U dv  \notin   8 \pi^2 \gamma_2 \mathbb{N}$. 
This concludes the proof. 
\end{pfn}

\section{Appendix}\label{s:app}

In this appendix we collect a commutator estimate, useful in Section \ref{linear}, and 
a Pohozaev-type identity that is used in Section \ref{blowup}. 

\medskip \noindent Given $Q \in L^r(M, TM)$ and $\delta>0$, define $S^x$ as
\begin{equation*}
\begin{array}{cccl} S^x: & L^r (M, TM) & \to & L^{\frac{r}{1-x}}(M,TM)\\
& F & \to & S^x F=(\frac{\|F \|_r^2+\|Q\|_r^2}{\delta^2+|F|^2+|Q|^2})^{\frac{x}{2}}  F. \end{array} 
\end{equation*}
We have the following result: 
\begin{thm}  \label{Hodge}
Let $r>1$, $0<\rho <\min\{1,r-1\}$ and $\Lambda: L^s(M, TM) \to L^s(M, TM)$, $\frac{r}{1+\rho} \leq s \leq \frac{r}{1-\rho}$, be a linear operator so that
$$K_0=\sup_{\frac{r}{1+\rho} \leq s \leq \frac{r}{1-\rho}}\|\Lambda\|_{\mathcal L(L^s)}<+\infty.$$
There exists $K>0$ so that
\begin{equation} \label{1421}
\|\Lambda(S^x F)-S^x (\Lambda F)\|_{\frac{r}{1-x}} \leq 
 K |x|  \left( \delta^2+\|F\|_r^2+\|Q\|_r^2 \right)^{\frac{\rho}{2}} \|F\|_r^{1-\rho}
\end{equation} 
for all $|x|\leq \rho$, $\delta>0$ and $Q \in L^r(M, TM)$.
\end{thm}
\begin{pf} Let $T=\{z=x+iy: \; |x| \leq \rho \}$ and $r_x=\frac{r}{1-x}$, $q_x=\frac{r}{r-1+x}$ be conjugate exponents. Set 
\begin{eqnarray*}
&& R_z: F \in L^r(M,TM+iTM) \to 
R_z F=(\frac{\|F\|_r^2+\|Q\|_r^2}{\delta^2+|F|^2+|Q|^2})^{\frac{z}{2}}  F \in L^{r_x}(M,TM+iTM)\\
&& Q_z: G \in L^{q}(M,TM+iTM) \to 
Q_z G=(\frac{|G|}{\|G\|_{q}})^{\frac{\bar z}{r-1}}  G \in L^{q_x}(M,TM+iTM)
\end{eqnarray*}
for all $z \in T$, where $q=\frac{r}{r-1}$. The map $Q_z$ satisfies $\|Q_z G\|_{q_x}=\|G \|_q$ and is invertible with inverse $(Q_z)^{-1}H=(\frac{|H|}{\|H\|_{q_x}})^{-\frac{q_x \bar z}{r}} H$. Given $F,G \in L^r(M,TM+iTM) $ define the map $\phi:T \to \mathbb C$ as
$$\phi(z)=\int \hbox{Re }\langle \Lambda (R_z F)-R_z (\Lambda F), \overline{Q_z G} \rangle dv.$$
Notice that $\phi(z)$ is a well defined holomorphic function in $T$ in view of $r_x \in [\frac{r}{1+\rho},\frac{r}{1-\rho}]$. Since by H\"older's estimate there holds
\begin{eqnarray*}
\|R_z F\|_{r_x}& =&(\|F\|_r^2+\|Q\|_r^2)^{\frac{x}{2}}\|  (\delta^2+|F|^2+|Q|^2)^{-\frac{x}{2}}|F|\|_{r_x}\\
&\leq& 
 (\|F\|_r^2+\|Q\|_r^2)^{\frac{x}{2}} \times \left\{ \begin{array}{ll}
\|F\|_r \|  (\delta^2+|F|^2+|Q|^2)\|_{\frac{r}{2}}^{-\frac{x}{2}}   & \hbox{if }x<0\\
\|F\|_r^{1-x}  &\hbox{if }x>0, \end{array} \right. 
\end{eqnarray*}
we have that
\begin{eqnarray*}
\|R_z F\|_{r_x}& \leq&  \left\{ \begin{array}{ll}
 [\frac{\delta^2 |M|^{\frac{2}{r}}}{\|F\|_r^2+\|Q\|_r^2}+1]^{-\frac{x}{2}} \|F\|_r  & \hbox{if }x<0\\
(1+\frac{\|Q\|_r^2}{\|F\|_r^2})^{\frac{x}{2}} \|F\|_r  &\hbox{if }x>0 \end{array} \right. \leq   \|F\|_r^{1-\rho}
\left(\delta^2 |M|^{\frac{2}{r}}+\|F\|_r^2 +\|Q\|_r^2 \right)^{\frac{\rho}{2}}.
\end{eqnarray*}
Hence we can deduce the following estimate on $\phi$:
$$|\phi(z)|\leq 2 K_0 c_0^{\frac{\rho}{2}}  \left( \delta^2+\|F\|_r^2 +\|Q\|_r^2 \right)^{\frac{\rho}{2}} \|F\|_r^{1-\rho}
\|G\|_q,$$
where $c_0=\max\{1,K_0^{-2},|M|^{\frac{2}{r}},|M|^{\frac{2}{r}} K_0^{-2}\}$. Since $\phi(0)=0$, Schwartz's lemma on $B_\rho(0) \subset T$ gives that
$$|\phi(z)|\leq \frac{2K_0 c_0^{\frac{\rho}{2}}}{\rho}    \left(\delta^2+\|F\|_r^2 +\|Q\|_r^2\right)^{\frac{\rho}{2}} \|F\|_r^{1-\rho}
\|G\|_q  |z|,$$
and then
\begin{eqnarray*}
\|\Lambda (R_z F)-R_z (\Lambda F)\|_{r_x}&=& \sup_{ \|H\|_{q_x}\leq 1} |\int \hbox{Re }\langle \Lambda (R_z F)-R_z (\Lambda F), \overline{H} \rangle dv|\\
&=& \sup_{ \|G\|_q\leq 1 } |\int \hbox{Re }\langle \Lambda (R_z F)-R_z (\Lambda F), \overline{Q_z G} \rangle dv| \\
&\leq& \frac{2K_0 c_0^{\frac{\rho}{2}}} {\rho} \left(\delta^2+\|F\|_r^2+\|Q\|_r^2 \right)^{\frac{\rho}{2}} \|F\|_r^{1-\rho} |z|.
\end{eqnarray*}
Setting $K=\frac{2K_0}{\rho} \max\{1,K_0^{-2},|M|^{\frac{2}{r}},|M|^{\frac{2}{r}} K_0^{-2}\}^{\frac{\rho}{2}}$, we have established the validity of \eqref{1421} for all $|x|\leq \rho$ in view of $R_x=S^x$ . 
\end{pf}

\noindent Notice that \eqref{0945} follows by Theorem \ref{Hodge} applied with $\Lambda=\hbox{Id}-\mathcal K$, $F=\nabla p$, $Q=\nabla q$, $x=4\epsilon$ and $r=4(1-\epsilon)$ thanks to \eqref{1116}. We next prove a Pohozaev identity, useful to characterize volume quantization in Theorem \ref{p:minimal-blow-up}.

\begin{pro}\label{p:poho}
{ Let $p \in M$ and let $\Omega \subseteq M$ be contained in a normal 
neighbourhood of $p$. }
Suppose  $u$ solves
\begin{equation}\label{eq:eq-poho}
 \mathcal{N}_{g} (u) + \tilde U= 
  \mu \, e^{4 u} \qquad \quad \hbox{ in } \Omega. 
\end{equation}
Let $(x^i)_i$ be a system of geodesic coordinates centred at $p$, and consider in these coordinates a  
vector field $a = a^i \frac{\pa}{\pa x^i}$ with constant components $(a^i)_i$. 
Then   the following identities hold 
\begin{eqnarray} \label{eq:poho-1} \nonumber
  { \mathcal{B}_g(p,\Omega,u) } & = & - \mu \int_{\Omega} e^{4u} (1 + O(|x|^2)) dv + \frac{\mu}{4} \oint_{\pa \Omega} e^{4 u} x^i \nu_i d \s+O( \int_{\Omega}  |\nabla u|  (|x| + |\nabla u|)dv)
  \\ & + & 
{    O\left( \int_{\Omega}  |x|  (|\nabla^2 u| \,  |\nabla u| + |\nabla u|^3) dv  +\int_{\Omega}  |x|^2 (|\nabla^2 u|^2 + |\nabla u|^4 ) dv  \right) }
\end{eqnarray}
and
\begin{eqnarray} \label{eq:poho-2} \nonumber
  { \mathcal{B}_g(p,\Omega,a,u) } & = &\frac{\mu}{4} \oint_{\pa \Omega} e^{4 u} a^i \nu_i d \s   - \mu \int_{\Omega} e^{4u} O(|x| |a|) dv +O( \int_{\Omega}  |x||\nabla u|  (1 + |a||\nabla u|)dv) \\ 
  & + & 
     O \left(\int_{\Omega}  |a|  (|\nabla^2 u| \,  |\nabla u| + |\nabla u|^3) dv
     + \int_{\Omega} |x| |a| (|\nabla^2 u|^2 + |\nabla u|^4 )  dv \right) ,
\end{eqnarray}
where 
{ \begin{eqnarray} \label{eq:poho-12} \nonumber
 \mathcal{B}_g(p,\Omega,u)   & = & \left( \frac{\gamma_2 }{2}+ 6 \gamma_3  \right)\oint_{\pa \Omega} ( x^i u_{;i} \frac{\pa \Delta u}{\pa \nu} - \Delta u \frac{\pa (x^i u_{;i})}{\pa \nu} + 
  \frac 12  (\Delta u)^2 x^j \nu_j)    d \sigma  
    \\  &  - &  12 \gamma_3 \oint_{\pa \Omega} (|\nabla u|^2 u_{;k} \nu^k x^j u_{;j} 
    - \frac 14 |\nabla u|^4 x^j \nu_j) d \sigma  \\ 
    &  + &  6 \gamma_3 \oint_{\pa \Omega} \left[ x^i u_{;i} \left( \frac{\pa}{\pa \nu} |\n u|^2 - 2 \Delta u \frac{\pa u}{\pa \nu}  \right) 
    + |\n u|^2 \left( x^i \nu_i \Delta u -   \frac{\pa u}{\pa \nu}  - \n^2 u [x, \nu] \right)\right] d \s \nonumber
\end{eqnarray} }
and 
{ \begin{eqnarray} \label{eq:poho-22} \nonumber
\mathcal{B}_g(p,\Omega,a,u)   & = & \left( \frac{\gamma_2}{2} + 6 \gamma_3  \right)\oint_{\pa \Omega} ( a^i u_{;i} \frac{\pa \Delta u}{\pa \nu} - \Delta u \frac{\pa (a^i u_{;i})}{\pa \nu} + 
  \frac 12  (\Delta u)^2 a^j \nu_j)    d \sigma  
    \\  &  - &  12 \gamma_3 \oint_{\pa \Omega} (|\nabla u|^2 u_{;k} \nu^k a^j u_{;j} 
    - \frac 14 |\nabla u|^4 a^j \nu_j) d \sigma \\ 
    &  + &  6 \gamma_3 \oint_{\pa \Omega} \left[ a^i u_{;i} \left( \frac{\pa}{\pa \nu} |\n u|^2 - 2 \Delta u \frac{\pa u}{\pa \nu}  \right) 
    + |\n u|^2 \left( a^i \nu_i \Delta u - \n^2 u [a, \nu] \right)\right] d \s. \nonumber 
\end{eqnarray} } 
\end{pro}

\begin{pf}
 Multiply \eqref{eq:eq-poho} by $x^i u_{;i}$ and integrate by parts: 
starting with the bi-Laplacian of $u$ we find  
$$
  \int_{\Omega} x^i u_{;i} \Delta^2 u \, dv = \int_{\Omega} (x^i u_{;ij}^{\; \; \; \, j} + 2 x^i_{\, ;j} u_{;i}^{\;\;\; j} + x^{i \; \; j}_{;j} u_{;i}) \Delta u \, dv 
  + \oint_{\pa \Omega} ( x^i u_{;i} \frac{\pa \Delta u}{\pa \nu} - \Delta u \frac{\pa (x^i u_{;i})}{\pa \nu} ) 
  d \sigma. 
$$
Using the fact that in normal coordinates $g_{ij} = \delta_{ij} + O(|x|^2)$ one has that 
$$
  x^j_{\, ;k} = \delta^j_k + O(|x|^2); \qquad x^{j \; \; k}_{;k} = O(|x|); \qquad 
  u_{;jk}^{\; \; \; \, k} = (\Delta u)_j + O(|\nabla u|). 
$$
From these we deduce that the L.H.S. in the above formula becomes 
$$
  2 \int_{\Omega} (\Delta u)^2 dv
 + \int_{\Omega} \Delta u \, x^j (\Delta u)_{;j} dv + \int_{\Omega} \left( |x|^2 |\nabla^2 u|^2 + |x|  |\nabla^2 u| |\nabla u|  \right) dv. 
$$
Integrating by parts the second term, the whole expression transforms into 
$$ \int_{\Omega} x^i u_{;i} \Delta^2 u \, dv =
  \oint_{\pa \Omega} ( x^i u_{;i} \frac{\pa \Delta u}{\pa \nu} - \Delta u \frac{\pa (x^i u_{;i})}{\pa \nu} + 
  \frac 12  (\Delta u)^2 x^j \nu_j)    d \sigma +  \int_{\Omega} \left( |x|^2 |\nabla^2 u|^2 + |x|  |\nabla^2 u| |\nabla u|  \right) dv. 
$$
Similarly, we obtain that 
$$
  \int_{\Omega} {\rm div} (|\nabla u|^2 \nabla u) x^j u_{;j} dv = \oint_{\pa \Omega} (|\nabla u|^2 u_{;k} \nu^k x^j u_{;j} 
  - \frac 14 |\nabla u|^4 x^j \nu_j) d \sigma +  \int_{\Omega}  O(|x|^2 |\nabla u|^4) dv.
$$
On the other hand, we can also multiply the equation by $a^i u_{;i}$ and using the relations 
$$
  a^j_{\, ;k} = O(|x| |a|); \qquad \qquad a^j_{\, ;kk} = O(|a|)
$$
we find that 
$$
\int_{\Omega} a^i u_{;i} \Delta^2 u dv = \oint_{\pa \Omega} ( a^i u_{;i} \frac{\pa \Delta u}{\pa \nu} - \Delta u \frac{\pa (a^i u_{;i})}{\pa \nu} + 
  \frac 12  (\Delta u)^2 a^j \nu_j)    d \sigma  + \int_{\Omega} \left( |x| |a| |\nabla^2 u|^2 + |a|  |\nabla^2 u| |\nabla u|  \right) dv$$
and
$$
  \int_{\Omega} {\rm div} (|\nabla u|^2 \nabla u) a^j u_{;j} dv = \oint_{\pa \Omega} (|\nabla u|^2 u_{;k} \nu^k a^j u_{;j}  
  - \frac 14 |\nabla u|^4 a^j \nu_j) d \sigma +  \int_{\Omega} O( |x| |a| |\nabla u|^4) dv. 
$$
Analogously, we have the following two formulas 
\begin{eqnarray*} 
& & \int_{\Omega} x^i u_{;i} {\rm div} (\nabla |\nabla u|^2 - 2 \Delta u \nabla u) dv  = \int_{\Omega} (|x| |\n u|^3 + |x|^2 |\n u|^2 |\n^2 u|) dv \\ 
& + & \oint_{\pa \Omega} \left[ x^i u_{;i} \left( \frac{\pa}{\pa \nu} |\n u|^2 - 2 \Delta u \frac{\pa u}{\pa \nu}  \right) 
+ |\n u|^2 \left( x^i \nu_i \Delta u -  \frac{\pa u}{\pa \nu}  - \n^2 u [x, \nu] \right)\right] d \s
\end{eqnarray*}
and
\begin{eqnarray*} 
& & \int_{\Omega} a^i u_{;i} {\rm div} (\nabla |\nabla u|^2 - 2 \Delta u \nabla u) dv  = \int_{\Omega} (|a| |\n u| |\n^2 u| + |a| |x| |\n u|^2 |\n^2 u|) dv \\ 
& + & \oint_{\pa \Omega} \left[ a^i u_{;i} \left( \frac{\pa}{\pa \nu} |\n u|^2 - 2 \Delta u \frac{\pa u}{\pa \nu}  \right) 
+ |\n u|^2 \left( a^i \nu_i \Delta u  - \n^2 u [a, \nu] \right)\right] d \s. 
\end{eqnarray*}
Finally, integrating by parts the exponential terms we find 
$$
  \int_{\Omega} \mu e^{4u} x^i u_{;i} dv = \frac{1}{4} \mu \oint_{\pa \Omega} x^i \nu_i e^{4u} d \s
  - \mu \int_{\Omega} e^{4u} (1 + O(|x|^2)) dv
$$
and
$$
  \int_{\Omega} \mu e^{4u} a^i u_{;i} dv = \frac{1}{4} \mu \oint_{\pa \Omega} a^i \nu_i e^{4u} d \s
  - \mu \int_{\Omega} e^{4u}  O(|x| |a|) dv. 
$$
Putting together all the above formulas, recalling the expression of the Paneitz operator and taking into account the lower-order terms, we obtain the conclusion. 
\end{pf}

\bibliography{LogDet_references-29-11-18}

\bibliographystyle{plain}

\end{document}